\documentclass{article}

\usepackage[english]{babel}

\usepackage[letterpaper,top=2.5cm,bottom=2.5cm,left=2.12cm,right=2.12cm,marginparwidth=1.75cm]{geometry}

\usepackage{amsmath, amssymb}
\usepackage{amsthm}
\usepackage{graphicx}
\usepackage{xcolor}
\usepackage{booktabs}
\usepackage[most]{tcolorbox}
\usepackage[table]{xcolor} 
\usepackage{colortbl}  
\usepackage{cancel}
\usepackage{float}
\newtheorem{example}{Example}[section]
\newtheorem{theorem}{Theorem}[section]

\newtheorem{lemma}[theorem]{Lemma}
\newtheorem{definition}{Definition}[section]

\newcommand{\Aut}{\operatorname{Aut}}

\usepackage{amsmath}
\usepackage{bbm}
\usepackage{graphicx}
\newtheorem{proposition}{Proposition}[section]
\newtheorem{remark}{Remark}[section]
\usepackage[colorlinks=true, allcolors=blue]{hyperref}
\numberwithin{equation}{section}

\AtBeginEnvironment{theorem}{\setlength{\parindent}{0pt}}
\AtBeginEnvironment{lemma}{\setlength{\parindent}{0pt}}
\AtBeginEnvironment{proposition}{\setlength{\parindent}{0pt}}
\AtBeginEnvironment{definition}{\setlength{\parindent}{0pt}}
\AtBeginEnvironment{remark}{\setlength{\parindent}{0pt}}
\AtBeginEnvironment{proof}{\setlength{\parindent}{0pt}}

\usepackage[most]{tcolorbox}
\tcbuselibrary{theorems}

\newtcbtheorem[number within=section]{framedtheorem}{Theorem}%
{colback=blue!5!white, colframe=blue!50!black, fonttitle=\bfseries}{th}

\title{Dependencies in Multiplex Networks: A Motif Count Approach}

\author{
Karl Sawaya\thanks{Institute of Mathematics, EPFL. Email: first.last@epfl.ch}
\hspace{1cm}
Sofia Olhede\footnotemark[1]
}
\begin{document}
\maketitle

\begin{abstract}
Multiplex networks are a powerful framework for representing systems with multiple types of interactions among a common set of entities. Understanding their structure requires statistical tools capturing higher-order cross-layer correlations. We develop a comprehensive framework for estimating and testing dependence in exchangeable multiplex networks through motif counts. We first propose a moment-based estimation methodology that extends the multi-layer stochastic block model network histogram to arbitrary motif counts. This allows us to estimate the $2^d-1$ graphons defining a $d$-layer multiplex network. We then derive the joint asymptotic distribution of cross-layer motif counts, that is aligned motifs shared across layers. Extending existing results from the unilayer setting, we show that the limiting distribution in the motif-regular case exhibits a covariance structure involving minimum-based distances between graphons. Finally, we construct hypothesis tests to detect inter-layer similarity and dependence. This work provides a rigorous extension of motif-count asymptotics and inference procedures to the multiplex setting, providing new tools to study high-order dependencies in complex networks. 
\end{abstract}

\noindent\textbf{Keywords:} Multiplex Network $\cdot$ Graphon $\cdot$  Graph inference $\cdot$  Motif Counts $\cdot$  Limits for large graphs

\section{Introduction}
Multiplex networks are a data structure used to represent multiple sets of relations between the same objects or nodes, that are assumed to be aligned \cite{Battiston2017, Bianconi2018}. They arise in many domains, including ecology \cite{PilosofPorterPascualKefi2017}, neuroscience \cite{Vuksanovic2025}, social systems, and finance \cite{BardosciaBianconiFerrara2019, XieJiaoLi2022}. As a generalization of simple graphs, multiplex networks allow several types of edges between nodes and can be modelled by using decorated graph limits~\cite{DufourOlhede2024,LovaszSzegedy2010}, see also~\cite{ChandnaJansonOlhede2022,SkejaOlhede2024}, using the multivariate Bernoulli distribution~\cite{Teugels1990}. In this paper, we shall couple the theory of exchangeable arrays with the multivariate Bernoulli distribution. 

By representing systems through multiple layers of interactions, multiplex networks provide a natural framework for studying interdependencies across graphs. Characterizing such dependencies has become an important topic and has motivated the development of diverse modeling approaches \cite{ChandnaJansonOlhede2022, Kolaczyk2009, PamfilHowisonPorter2020, SkejaOlhede2024}.
Here, we focus on capturing higher-order dependencies in multiplex exchangeable networks through a subgraph count-based approach.

 Our contributions are as follows. We shall start this manuscript by recalling the parameterization of a multivariate graph limit, starting from the multivariate graphon specification advocated by~\cite{ChandnaJansonOlhede2022,SkejaOlhede2024} and also used by~\cite{Ganguly2025}. 
This paper will be aimed at estimating the multivariate graph limit and determining the properties of statistics derived from the multiplex network. For this reason we want to extend non-parametric estimation methods for single graphs to the estimation of multiplex networks in a way that allows us to estimate the model parameters of the multivariate graph limit model. We shall first approximate the observed multiplex network as a Multiplex Stochastic Block Model (MSBM), and discuss the properties of such estimation.
Another popular estimation method for network models is method of moments~\cite{BickelChenLevina2011}, namely calculating the prevalence of small subgraphs or motifs. We shall discuss the extension of this method to multiplex networks, and introduce new types of network motifs adapted to multiplex networks based on edges of multiple layers of the network, that we will refer to as cross-layer motifs. We shall derive the asymptotic behaviour of these motif counts when collected in a vector, and generalize the assumption of $F$-regularity to this setting~\cite{Bhattacharya2025}. From the asymptotic distribution, we will be able to design hypothesis tests aimed at uncovering multivariate network structure and assessing dependence across layers.

 Throughout this paper, an exchangeable graph refers to a random network whose distribution is invariant under permutations of its nodes \cite{LauritzenRinaldoSadeghi2018, OrbanzRoy2014}. A network motif denotes a subgraph pattern that may reveal structural properties of the network \cite{Alon2007,RibeiroSilvaKaiser2009}.

 The Aldous--Hoover theorem \cite{Aldous1981, Hoover1982, Kallenberg1989} states that any exchangeable random graph can be generated by sampling independent uniform variables \(\{U_i\}\) on \([0,1]\), and connecting each pair of latent nodes \((U_i, U_j)\) independently with probability \(W(U_i, U_j)\). The symmetric measurable function \(W : [0,1]^2 \to [0,1]\), known as a \emph{graphon}, defines the generative mechanism of the network. Graphons also arise as limits of sequences of dense graphs \cite{DiaconisJanson2008, Lovasz2012, Lovasz2006}. In this framework, edges are Bernoulli random variables whose parameter depend on the graphon. Since multiplex networks consists of several graphs defined on the same vertex set, it is natural to extend this framework by combining graph limit theory \cite{Lovasz2012} with the multivariate Bernoulli distribution \cite{FontanaSemeraro2018, Teugels1990}. A multiplex network with \( d \) layers is characterized by \( 2^d - 1 \) parameters, each corresponding to the joint probability of edge formation across different layer combinations. The collection of these parameters defines the multivariate graph limit of the multiplex network.

There has been considerable interest in graphon estimation \cite{GaoLuZhou2015}. A classical approach constructs a network histogram by fitting a Stochastic Block Model (SBM) \cite{OlhedeWolfe2014} (see also Appendix~\ref{append:SBM} for more details), for which convergence rates have been established \cite{GaoLuZhou2015,KloppTsybakovVerzelen2017}.  Other methods include local estimation with covariates \cite{ChandnaOlhedeWolfe2022}, the stochastic shape model \cite{VerdeymeOlhede2024}, and the graph pencil method \cite{Gunderson2024}. 
For multiplex networks, Pamfil et al. \cite{PamfilHowisonPorter2020} fitted a correlated multiplex stochastic block model to estimate the \(2^d - 1\) parameters. 
However, stochastic block models estimate connectivity matrices using only edge counts, corresponding to first-order subgraph densities. From the perspective of graph limit theory, similarity between dense graphs is naturally measured by comparing all finite subgraph densities, as formalized by the subgraph distance of Bollobás and Riordan \cite{BollobasRiordan2009}. In particular, two graphs that have similar motif counts at all scales are close in subgraph distance, and this notion of convergence is equivalent (up to measure-preserving transformations) to convergence in the cut metric (see \cite[Corollary 2.3] {BollobasRiordan2009}). In the multiplex setting, this motivates the introduction of cross-layer motifs, which are subgraph patterns shared across a subset of multiple layers. We therefore extend the framework of \cite{PamfilHowisonPorter2020} by allowing estimation based on arbitrary cross-layer motif counts and integrate the MSBM with the network method of moments introduced in \cite{BickelChenLevina2011}. Within this framework, we introduce alternate bichromatic motifs that allow estimation of off-diagonal connectivity parameters through moment equations.

In the second part of the paper, we study the distribution of motif counts in multiplex networks. Subgraph counts are powerful statistics for testing network models \cite{MaugisOlhedePriebeWolfe2020,MiloShenOrrItzkovitzKashtanChklovskiiAlon2002,ShenOrr2002}. They act as network moments. Counting motifs within individual layers or shared across layers provides insight into the dependence structure of the multiplex network.

The asymptotic behavior of subgraph counts has been extensively studied. In the Erdős--Rényi model, asymptotic normality was established using various techniques \cite{Janson1994,NowickiWierman1988,Nowicki1989,Rucinski1988}. This was later generalized by \cite{BickelChenLevina2011}, who proved asymptotic normality of subgraph counts in exchangeable graphon-based models with normalization by a factor of \( n \). Other results include Poisson approximations for the number of subgraphs in the stochastic block model that are isomorphic to strictly balanced graphs \cite{CoulsonGauntReinert2016}, and joint Gaussian convergence for collections of homomorphism densities \cite{DelmasDhersinSciauveau2021}.

More recently, \cite{BhattacharyaChatterjeeJanson2023} showed that the Gaussian limit in \cite{BickelChenLevina2011} may become degenerate when the graphon satisfies a regularity condition with respect to the motif. In that case, normalization by \(n^{1/2}\) yields a limit consisting of a Gaussian component plus an independent non-Gaussian term given by an infinite weighted sum of centered chi-squared variables.
These results were extended by \cite{Bhattacharya2024} to the joint distribution of counts of multiple motifs within a single exchangeable graphon-based network. They showed that for subgraphs where the graphon is regular, each marginal distribution consists of two independent components: a Gaussian term and a bivariate stochastic integral. In contrast, for subgraphs where the graphon is irregular, the marginals are purely Gaussian and are expressed as linear stochastic integrals. 

In this paper, we extend the framework of~\cite{Bhattacharya2024} to the study of cross-layer motif counts in exchangeable multiplex networks. We derive the joint asymptotic distribution of such counts. We first establish this result for the multiplex Erdős–Rényi model, showing convergence to a multivariate Gaussian distribution whose covariance depends on minimum-based distances between the parameters. We then extend the result to the exchangeable setting under the assumption that the multiplex network converges in the joint cut metric to a multivariate graphon (see~\cite{Bhattacharya2025}).
We show that mixed regimes may arise, where some parameters satisfy regularity conditions with respect to the motif while others do not. For the irregular case, we show that the marginals converge to linear stochastic integrals. In contrast, in the regular case the limiting distribution consists of two independent components: a multivariate Gaussian term, whose covariance structure involves minimum-based distances between graphons, and a bivariate stochastic integral.
From these asymptotic results, we construct confidence intervals for expected cross-layer motif counts and propose two hypothesis tests. First, we develop a test assessing inter-layer structural similarity by comparing expected subgraph densities across layers. Finally, we provide a test of edge-wise independence between the layers of a MSBM by testing block-by-block edge-wise independence.

This paper is organized as follows. In Section~\ref{sec:Notations}, we introduce the necessary notation and briefly review graph limit theory~\cite{Lovasz2012}. In Section~\ref{sec:subgraphcounts}, we define the fundamental objects of cross-layer motifs. Section~\ref{sec:Graphon estimation} presents the methodology for estimating the multivariate graph limit by combining the MSBM with a network method of moments. In Section~\ref{sec:Asymptotics}, we establish the asymptotic distribution results. Finally, Section~\ref{sec:Hypo} is devoted to the hypothesis tests.

\section{Notations and Preliminaries}
\label{sec:Notations}
We begin by introducing the notation and basic concepts that will be used throughout this paper.

\subsection{Basic Notation}

Let $G_n := (V(G_n), E(G_n))$ a simple graph on $n$ vertices, where the vertex set is \(V(G_n) = [n] := \{1,2,\ldots,n\}\) and the edge set is \(E(G_n)\). The number of vertices and edges of \(G_n\) are denoted by \(|V(G_n)| = n\) and \(|E(G_n)|\), respectively. The adjacency matrix of \(G_n\) is written as \(A(G_n) = (a_{ij}(G_n))_{1 \leq i,j \leq n}\).

\noindent
A simple graph is an undirected, unweighted graph without loops or multiple edges. The complete graph on $n$ vertices, denoted by $K_n$, is the graph where every pair of distinct vertices is connected by an edge. For any subgraph $F \subseteq G,$ let $X_F(G_n)$ denote the number of injective copies of $F$ in the graph $G_n$. 

\noindent
We write $\mathrm{Tr}(\cdot)$ for the operator trace on matrices and $\mathrm{Aut}(\cdot)$ for the automorphism group of a graph. 

\subsection{Multivariate Graphon Framework}

We first state the definition of a multiplex network.  

\begin{definition}[Multiplex Network \cite{Ganguly2025}]
\label{def:multiplexnetwork}
For an integer $d \geq 1$, we denote a $d-$layer multiplex network on $n$ nodes by the collection of simple graphs $\mathbf{G} := (G_n^{(1)}, \ldots, G_n^{(d)})$ sharing the same vertex set $V(\mathbf{G}) =[n]$. Each graph $G_s, ~ s\in [d]$ has edge set $E(G_s)$ and vertex $V(\mathbf{G})$ and we denote by $A^{(s)}$ its adjacency matrix. 
\end{definition}

\begin{remark}[Multiplex network as a decorated graph]
\label{rem:decoratedmultiplex}
A multiplex network can equivalently be represented as a single decorated graph, where each edge is endowed with a multivariate label encoding the presence or absence of that edge across the \(d\) layers. This framework, which we briefly recall here, has been developed by \cite{DufourOlhede2024} and \cite{KunszentiKovacsLovaszSzegedy2022}. 

Let \(\mathcal{K}\) be a compact set. A \(\mathcal{K}\)-decorated graph is a simple graph in which each edge \((i,j)\) is assigned (or decorated with) an element \(x_{ij} \in \mathcal{K}\). The set \(\mathcal{K}\) contains a neutral element \(0_\mathcal{K}\), with edges decorated by \(0_\mathcal{K}\) interpreted as absent.
For example, a $2$-layer multiplex network can be modeled as a \(\mathcal{K}\)-decorated graph with
\[
\mathcal{K} = 
\left\{
\begin{pmatrix} 0 \\ 0 \end{pmatrix}, 
\begin{pmatrix} 1 \\ 0 \end{pmatrix}, 
\begin{pmatrix} 0 \\ 1 \end{pmatrix}, 
\begin{pmatrix} 1 \\ 1 \end{pmatrix}
\right\},
\]
where each vector encodes the exclusive presence of an edge in one or both layers. For instance, an edge \((i,j)\) decorated with \(x_{ij} = \begin{pmatrix} 1 \\ 0 \end{pmatrix}\) is present exclusively in layer 1.

More generally, a multiplex network with \(d > 2\) layers can be modeled as a \(\mathcal{K}\)-decorated graph where
$
\mathcal{K} = \{0,1\}^d
$
is the set of all \(2^d\) binary vectors of dimension \(d\). One can associate each non-zero decoration to a distinct color and thus view a $d$-layer multiplex network as an edge-colored graph with \(2^d - 1\) colors.
\end{remark}

Before defining the model that we study in this paper, that is the multiplex exchangeable network, we first define the basic model, that is the multiplex Erd\H{o}s--R\'enyi framework.  

\begin{definition}[Multiplex Erd\H{o}s--R\'enyi network]
\label{def:MER}
Let $d \geq 1$ and $n \geq 1$ be integers. A $d$-layer multiplex network on $n$ nodes, $
\mathbf{G}_n := \big(G_n^{(1)}, \ldots, G_n^{(d)}\big),
$
is called a multiplex Erd\H{o}s--R\'enyi (MER) network if, for every unordered pair 
$1 \leq i < j \leq n$, the vector of edge indicators
$
\underline{A}_{ij} := \big(A_{ij}^{(1)}, \ldots, A_{ij}^{(d)}\big)
$
is independent across pairs $(i,j)$ and satisfies
\begin{equation}
\label{eq:multialdousER}
\underline{A}_{ij} \;\overset{\mathrm{ind}}{\sim}\; 
\mathrm{MultBernoulli}_d(\mathbf{p}),
\end{equation}
where $\mathrm{MultBernoulli}_d$ denotes the $d$-variate Bernoulli distribution (see Appendix~\ref{append:multibern} for more details) specified by the collection of $2^d - 1$ parameters
\begin{equation}
\label{eq:multiparametersER}
\mathbf{p}
=
\Big(
p^{(1)}, \ldots, p^{(d)}, 
p^{(12)}, \ldots, p^{(d-1\,d)}, 
\ldots,
p^{(1\cdots d)}
\Big)^{\!\top}.
\end{equation}
\end{definition}

In this work, we focus on a more general framework, that is, dense exchangeable multiplex networks. To formalize this, we first recall the notion of exchangeability for random graphs and the Aldous-Hoover theorem. 

\begin{definition}[Joint exchangeability \cite{DiaconisJanson2008}]
A random array \((A_{ij})_{i,j \ge 1}\) is \emph{jointly exchangeable} if for any permutation \(\pi\) of \(\mathbb{N}^*\),
\[
(A_{ij}) \;\overset{d}{=}\; (A_{\pi(i)\pi(j)}),
\]
where \(\overset{d}{=}\) denotes equality in distribution.  
A random graph \(G\) is said to be \emph{exchangeable} if its adjacency matrix forms a jointly exchangeable array.
\end{definition}

Joint exchangeability is a powerful statement since it allows us to define an exchangeable graph via the Aldous-Hoover theorem. 

\begin{theorem}[Aldous-Hoover \cite{Aldous1981}]
\label{thm:aldous}
Let \(A = (A_{ij})\) be a jointly exchangeable random array. Then there exists an i.i.d. sequence \(\xi = (\xi_1, \dots, \xi_n)\) with \(\xi_i \sim \mathbf{U}[0,1]\), a random variable \(\gamma \sim \mathbf{U}[0,1]\) independent of \(\xi\), and a measurable function \(W: [0,1]^3 \to [0,1]\) such that
\[
A_{ij} \mid \xi, \gamma \;\overset{\text{ind}}{\sim}\; \text{Bernoulli}\{ W(\xi_i, \xi_j, \gamma) \}.
\]
If \(A\) is the adjacency matrix of a random graph \(G_n\) and \(G_n\) is \emph{disassociated}, then the dependence on \(\gamma\) can be omitted.
\end{theorem}

We are now ready to define an exchangeable graph. 

\begin{definition}[Exchangeable graph \cite{SkejaOlhede2024}]
\label{def:exchangeablegraph}
Let \(G = (V,E)\) be a random graph with \(|V| = n\) and adjacency matrix \(A = (A_{ij})\). In this context, the latent variable \(\xi_i\) is associated with vertex \(i\), and the function \(W(\xi_i, \xi_j)\) from theorem \label{thm:aldous} is the graphon generating the exchangeable graph \(G(n,W)\). In particular, \(W(\xi_i, \xi_j)\) represents the probability that an edge exists between vertices \(i\) and \(j\).
\end{definition}

\begin{remark}
For dense exchangeable graphs, the parameter \(W\) of Definition \eqref{def:exchangeablegraph} is called a graphon and is the limit object of a sequence of dense graphs.  
This graph limit theory, developed by {\cite{Lovasz2012}}, provides a functional representation for sequences of large graphs converging in the sense of homomorphism densities.
\end{remark}

A \(d\)-layer \emph{exchangeable multiplex network} is defined as a collection of \(d\) exchangeable graphs \((G_n^{(1)}, \ldots, G_n^{(d)})\) sharing the same latent vertex variables.  
The joint layer dependence is captured through the joint distribution of the sequence  $(A_{ij}^{(k)})_{k \in [d]}$ of coupled Bernoulli variables, described by the following multivariate extension of the Aldous–Hoover theorem.

\begin{theorem}[Multivariate Aldous-Hoover \cite{SkejaOlhede2024}]
\label{thm:mutlialdous}
For $d$ exchangeable graphs $G_1(n,W^{(1)}), \dots, G_d(n,W^{(d)})$ on $n$ nodes, with adjacency matrices $A^{(1)}, \dots, A^{(d)}$, define
$$
\underline{A}_{ij} = \begin{bmatrix} A_{ij}^{(1)} & \dots & A_{ij}^{(d)} \end{bmatrix}^T
$$
as the vector of observed edges between nodes $i$ and $j$. Then there exists an i.i.d. sequence \(\xi = (\xi_1, \dots, \xi_n)\) with \(\xi_i \sim \mathbf{U}[0,1]\) such that 
\begin{equation}
\label{eq:multialdous}
\underline{A}_{ij} \mid \xi \;\overset{\text{ind}}{\sim}\; \text{MultBernoulli}_d\big(\underline{W}(\xi_i, \xi_j)\big),
\end{equation}
where $\text{MultBernoulli}_d$ denotes the $d$-variate Bernoulli distribution, specified by $2^d - 1$ parameters that capture all cross-dependencies, collected in
\begin{equation}
\label{eq:multiparameters}
\underline{W}(\cdot) = \begin{bmatrix} W^{(1)}(\cdot) & \dots & W^{(d)}(\cdot) & W^{(12)}(\cdot) & \dots & W^{(d-1\,d)}(\cdot) & \dots & W^{(1 \cdots d)}\end{bmatrix}^T.
\end{equation}
\end{theorem}

\begin{remark}
The multiplex Erd\H{o}s--R\'enyi model (see Definition~\ref{def:MER}) 
is a particular case of a multiplex exchangeable network in which 
the vector of cross-layer graphons $\underline{W}(\cdot)$ defined in 
\eqref{eq:multiparameters} is constant.
\end{remark}

A classical approach to approximating a graphon is through the 
Stochastic Block Model (SBM); see Appendix~\ref{append:SBM} for details. 
In the multiplex setting, the multivariate graphon defined in  \eqref{eq:multiparameters} can be approximated by a Multiplex Stochastic 
Block Model (MSBM), defined as follows.

\begin{definition}[Multiplex Stochastic Block Model]
\label{def:MSBM}
Let $d \geq 1$ and $K \in \mathbb{N}$. A $d$-layer multiplex stochastic 
block model (MSBM) consists of graphs 
$\mathbf{G}_n = (G_n^{(1)}, \ldots, G_n^{(d)})$ 
sharing the same vertex set, partitioned into $K$ communities. 
Let $z_i \in \{1, \ldots, K\}$ denote the block 
membership of node $i$. For $1 \leq i < j \leq n$, we have
\begin{equation}
\label{eq:MSBMparameters}
\underline{A}_{ij} \mid z_i, z_j 
\;\overset{\mathrm{ind}}{\sim}\; 
\mathrm{MultBernoulli}_d\big(\boldsymbol{\theta}_{z_i z_j}\big),
\end{equation}
where $\mathrm{MultBernoulli}_d$ denotes the $d$-variate Bernoulli 
distribution and for each pair of $(a,b) \in \{1,\ldots,K\}^2$, the parameter 
vector $\boldsymbol{\theta}_{ab}$ collects the $2^d - 1$ marginal and 
cross-layer interaction probabilities,
\begin{equation}
\label{eq:multiparameters}
\boldsymbol{\theta}_{ab}
=
\Big(
\theta^{(1)}_{ab}, \ldots, \theta^{(d)}_{ab},
\theta^{(12)}_{ab}, \ldots,
\theta^{(d-1\,d)}_{ab},
\ldots,
\theta^{(1\cdots d)}_{ab}
\Big)^{\!\top}.
\end{equation}
\end{definition}

\subsection{Graph Limit Theory}
In Appendix \ref{append:GLT}, we provide elements of graph limit theory developed independently by \cite{Kallenberg1989} and {\cite{Lovasz2012}}.
To derive asymptotic results for subgraph counts in dense multiplex networks, it is necessary to recall definitions from \cite{BhattacharyaChatterjeeJanson2023}. We first state the definition of the regularity of a graphon with respect to a motif. This requires the notion of a conditional homomorphism density.

\begin{definition}[1-Point Conditional Homomorphism Density {\cite{BhattacharyaChatterjeeJanson2023}}]
Let $F$ be a motif, $a \in V(F)$, and $v \in V(G_n)$.  
The \emph{1-point conditional homomorphism density} is defined by
\begin{equation}
t_a(v, F, W) :=
\mathbb{E}\left[
\prod_{(b,c) \in E(F)} W(U_b, U_c)
\ \bigg| \ U_a = v
\right].
\label{eq:1point_cond}
\end{equation}
\end{definition}

\medskip

\begin{definition}[\(F\)-Regularity of a Graphon {\cite{BhattacharyaChatterjeeJanson2023}}]
\label{def:reggraphon}
A graphon \(W\) is said to be \(F\)-regular if
\begin{equation}
\label{eq:F-reg}
\bar{t}(x,F,W) := \frac{1}{|V(F)|} \sum_{a=1}^{|V(F)|} t_a(x,F,W) = t(F,W),
\end{equation}
for almost every \(x \in [0,1]\), where \(t(F,W)\) is the homomorphism density of \(F\) in \(W\).

In other words, the homomorphism density of \(F\) in \(W\) conditioned on marking any one vertex is independent of the marking.
\end{definition}

The asymptotic distribution of cross-layer motif counts depends on the two-point conditional homomorphism density, defined as follows.

\begin{definition}[Two-Point Conditional Motif Kernel {\cite{BhattacharyaChatterjeeJanson2023}}]
Given a graph \(F = (V(F), E(F))\) and a graphon \(W\), define the two-point conditional kernel \(W_F: [0,1]^2 \to \mathbb{R}\) by
\begin{equation}
\label{eq:2point_kernel}
W_F(x,y) := \frac{1}{2|\mathrm{Aut}(F)|} \sum_{\substack{1 \leq a \neq b \leq |V(F)|}} t_{a,b}((x,y), F, W),
\end{equation}
where \(t_{a,b}((x,y),F,W)\) is the two-point conditional homomorphism density of \(F\) in \(W\) given vertices \(a,b\).
\end{definition}

We also introduce the \emph{vertex join operation}, which will be useful when computing the asymptotic variances. 

\begin{definition}[Vertex Join Operation {\cite{Bhattacharya2024}}]
\label{def:vertexjoin}
Fix \(r \geq 1\). Let \(H_1\) and \(H_2\) be graphs on vertex set \(\{1,2,\ldots,r\}\) with edge sets \(E(H_1)\) and \(E(H_2)\).
For \(a,b \in \{1,\ldots,r\}\), the \((a,b)\)-vertex join of \(H_1\) and \(H_2\) is obtained by identifying vertex \(a\) of \(H_1\) with vertex \(b\) of \(H_2\), denoted
\[
    H_1 \bigoplus_{a,b} H_2.
 \]

\end{definition}

Finally, in the univariate case, a sequence of dense graphs converges to a graphon if its cut-norm converges to zero. In the multivariate setting, the joint cut metric (recalled below) characterizes the convergence of a sequence of dense multiplex networks to a multivariate graphon. For a detailed treatment of multiplex graph limit theory, we refer to \cite{Ganguly2025}.

\begin{definition}[Joint Cut-Metric \cite{Bhattacharya2025}]
\label{def:jointcutmetric}

Let $d \in \mathbb{N}$. For a sequence of graphons $\mathbf{W} := (W^{(1)}, \ldots, W^{(d)})$. Then, the joint cut-metric between two sequences of graphons $\mathbf{W}$ and $\mathbf{\tilde{W}}$ is defined as: 
\begin{equation}
\label{eq:jointcut}
\delta_{\square}(\mathbf{W}, \mathbf{\tilde{W}})  := \inf_{\phi} \displaystyle \sum_{j=1}^{d} \|{W^{(j)}}^{\phi} - {\tilde{W}}^{(j)} \|_{\square},
\end{equation}
where $\phi: [0,1] \to [0,1]$ is a measure-preserving bijection, $W^{\phi} := W(\phi(x), \phi(y)),$ for $x,y \in [0,1]^2$, and $\|\cdot \|_{\square}$ is the cut distance defined in Appendix \ref{def:cut-distance}. 

We say that a sequence of $d-$layer multiplex networks $\mathbf{G}_n = (G_n^{(1)}, \ldots, G_n^{(d)})$ converges to the multivariate graphon $\mathbf{W}$ if 
\[
\delta_{\square}(\mathbf{W}^{\mathbf{G_n}}, \mathbf{W}) \xrightarrow[n \to \infty]{} 0,
\]
where $\mathbf{W}^{\mathbf{G_n}} = (W^{G_n^{(1)}}, \ldots  W^{G_n^{(d)}})$ is the sequence of empirical graphons (see Definition in the appendix \ref{def:empiricalgraphon}) associated to the multiplex network.

\end{definition}

\medskip

\section{Subgraph Counts in a Multiplex Network}

Subgraph counts are fundamental statistics for characterizing the structural properties of networks.  
In the case of multiplex networks, counting motifs across layers provides valuable insights into their inter-layer dependence structure.  
Since a multiplex network can be viewed as a decorated graph (see Remark~\ref{rem:decoratedmultiplex}), motifs in a multiplex network correspond naturally to motifs in a decorated graph. Moreover, because a decorated graph can be interpreted as an edge-colored graph, it is natural to refer to multiplex subgraphs as \emph{colored motifs}, although the term ``colored motifs'' has been used in distinct contexts, such as in \cite{RibeiroSilvaKaiser2009} and \cite{Rubert2020}.  

Counting motifs in a multiplex network leads to consider two approaches:  (i) counting a subgraph vector, where different motifs are counted separately in each layer (Section~\ref{subsec:vector}); or  (ii) more interestingly, counting the same motif across all layers (Section~\ref{subsec:cross-layer}).  
In both cases, one must decide whether the subgraphs share the same set of nodes (aligned motifs), or whether they can be embedded on disjoint node sets (independent embeddings). 
The goal of this section is to clarify these notions and to provide a unified framework for understanding multiplex subgraph counts and complete the formulations proposed in \cite{Ganguly2025} and \cite{KunszentiKovacsLovaszSzegedy2022}.

\label{sec:subgraphcounts}

\subsection{Vector-Motif Counts in a Multiplex}
\label{subsec:vector}
We begin by introducing the notion of \emph{vector-motif counts}. 
Let $F_1, \dots, F_d$ be deterministic subgraphs and consider a $d$-layer multiplex network.  
The non-aligned vector-motif count across all $d$ layers
\[
X^{\text{indep}}_{F_1,\dots,F_d}\bigl(G_n^{(1)},\dots,G_n^{(d)}\bigr)
\]
represents the number of all combinations of motif occurrences across the \( d \) layers, where in each layer \( l \in [d] \), one independently chooses an embedding of the motif \( F_l \) (indicated by $\text{indep}$).  
Here, an embedding of a motif \( F_l \) into a layer \( G_n^{(l)} \) is an injective map \( s: V(F_l) \to [n] \) such that each edge \((i,j) \in E(F_l)\) is mapped to an edge \((s(i),s(j))\) in \( G_n^{(l)} \). In other words, it corresponds to one occurrence of the motif within the layer.  

In particular, the embeddings ($s_1, \ldots, s_d$) are chosen independently across layers, so the motifs can occur on entirely different vertex subsets of the multiplex. 
Formally, this quantity can be written as
\begin{equation}
\label{eq:vectormotifcount}
\begin{aligned}
X^{\text{indep}}_{F_1,\dots,F_d}\bigl(G_n^{(1)},\dots,G_n^{(d)}\bigr)
&= \frac{1}{\prod_{l=1}^d |\mathrm{Aut}(F_l)|}
\sum_{s_1\in [n]_{|V(F_1)|}} \cdots \sum_{s_d\in [n]_{|V(F_d)|}}
\prod_{l=1}^d \prod_{(i,j)\in E(F_l)} A^{(l)}_{s_l(i)\,s_l(j)} \\
&= \sum_{F'_1 \subseteq K_n} \cdots \sum_{F'_d \subseteq K_n} 
\prod_{l=1}^d \mathbf{1}_{\{F_l \equiv F'_l\}} \, \mathbf{1}_{\{F'_l \subseteq G_n^{(l)}\}} \\
&= \prod_{l=1}^{d} X_{F_l}\bigl(G_n^{(l)}\bigr),
\end{aligned}
\end{equation}
where $[n]_k$ denotes the set of ordered $k$-tuples of distinct elements of $[n]$, and $A^{(l)}$ is the adjacency matrix of the $l$-th layer $G_n^{(l)}$. 
We recall that $K_n$ denotes the complete graph on vertex set $[n]$, and the notation $F_l \equiv F'_l$ means that the graphs $F_l$ and $F'_l$ are isomorphic.

\vspace{0.3em}

This quantity, however, is not particularly informative for understanding the interaction between layers. 
Indeed, it simply counts all possible combinations of motif instances across layers, without checking whether the motifs occur on the same vertex sets. 
As such, it does not capture any structural dependence or alignment between layers. In fact, if we observe for example a large number of shared triangles in two layers of a multiplex, then this will give information on the similarity of the two layers in the sense that they both display similar structures. Yet, a far more informative measure is the number of aligned triangles (i.e. those that share the same vertices) as this directly reflects the extent of dependency between the layers.

\vspace{0.5em}

Hence, a more meaningful measure of cross-layer dependencies is obtained by considering \emph{aligned embeddings}. 
Let 
\[
k := \max_{l \in [d]} |V(F_l)|
\]
be the maximum motif size across layers. 
The \emph{aligned vector-motif count} across all $d$ layers is then defined as
\begin{equation}
\label{eq:alignedvectormotif}
X_{F_1, \ldots, F_d}^{\text{aligned}}\bigl(G_n^{(1)}, \ldots, G_n^{(d)}\bigr)
:= \frac{1}{\prod_{l=1}^d |\mathrm{Aut}(F_l)|}
\sum_{s\in [n]_{k}} 
\prod_{l=1}^d \prod_{(i,j)\in E(F_l)} A^{(l)}_{s(i)\,s(j)}.
\end{equation}

In this definition, each layer is evaluated on the same $k$-tuple of vertices $s=(s(1),\dots,s(k))$. 
Hence, the aligned count measures the number of aligned $k$-tuples of vertices such that the motif $F_1$ is realized in layer~1 on some subset of these vertices, the motif $F_2$ is realized in layer~2 on (possibly another) subset of the same vertices, and so on. For example, in a two-layer multiplex network with $F_1 = K_3$ and $F_2 = K_2$, the quantity $X_{F_1, F_2}^{\text{aligned}}(G_n^{(1)}, G_n^{(2)})$ counts the number of triplets $(v_1, v_2, v_3) \in [n]^3$ such that these three vertices form a triangle in layer~1, and at least one of the three possible vertex pairs among $\{v_1, v_2, v_3\}$ forms an edge in layer~2.  
In contrast to~\eqref{eq:vectormotifcount}, the quantity \eqref{eq:alignedvectormotif} captures the extent to which motifs are spatially aligned across layers, and hence reflect true cross-layer structural dependencies in the multiplex network. The word dependency is here as of yet quite non-specific and will depend on the edge variables involved.

\bigskip 
\noindent
In this paper, we focus on the special case where all layers are tested against the same motif, i.e.\ $F_1=\cdots=F_d=F$.  
This setting is natural when studying dependencies between the layers of a multiplex network: instead of simply counting different motifs within each layer, one is often more interested in how frequently a given motif $F$ occurs simultaneously across different layers, as this directly reflects inter-layer similarity. As discussed above, we consider aligned embeddings, meaning that the same set of vertices is required to form the motif $F$ in all layers. 
Under this assumption, the corresponding \emph{cross-layer motif count} across all $d$ layers is defined as
\begin{equation}
\label{eq:alignedmotifcount}
X_F\bigl(G_n^{(1)},\ldots,G_n^{(d)}\bigr) 
= \frac{1}{|\mathrm{Aut}(F)|} 
  \sum_{s \in [n]_{|V(F)|}} 
    \prod_{(i,j)\in E(F)} \prod_{l=1}^d A^{(l)}_{s_i s_j},
\end{equation}
where $A^{(l)}$ denotes the adjacency matrix of layer $l$. 
Throughout the paper, we shall use the shorthand notation
$
X_F(G_n^{(1)}, \ldots, G_n^{(d)}) := X_F^{\text{aligned}}(G_n^{(1)}, \ldots, G_n^{(d)}).
$

\vspace{0.3em}

In the next section, we focus on this quantity and compute its expectation under two models: the exchangeable (graphon-based) model (see Theorem~\ref{thm:mutlialdous}) and the Erdős–Rényi model.

\subsection{Cross-layer Motif Counts in a Multiplex}
\label{subsec:cross-layer}
Let \(F = (V(F), E(F))\) be a deterministic graph. 
For each \(i \in [d]\), we denote by \(X_F(G_n^{(i)})\) the number of subgraphs of \(G_n^{(i)}\) that are isomorphic to \(F\).
For $l \in \mathbb{N}$  and for distinct indices \(i_1, \ldots, i_l \in [d]\), we denote by
\[
X_F\bigl(G_n^{(i_1)}, G_n^{(i_2)}, \ldots, G_n^{(i_l)}\bigr)
\]
the cross-layer motif count given by \eqref{eq:alignedmotifcount}, i.e. the number of subgraphs that are isomorphic to \(F\) and simultaneously present across the layers \((G_n^{(i_k)})_{k \in [l]}\), in the sense that \(F\) appears on the same vertex set in each of these layers. 

To systematically describe all combinations of layers, we introduce the set
\begin{equation}
\Lambda_d = \{1, 2, \dots, d, (12), \dots, (d-1\,d), \dots, (12\ldots d)\},
\label{eq:lambda_set}
\end{equation}
which contains \(2^d - 1\) elements. 
Each element of \(\Lambda_d\) corresponds to a subset of layers: for each \(k \in [d]\), there are \(\binom{d}{k}\) elements of $\Lambda_d$ that refer to exactly \(k\) layers.  For example, \((12)\) denotes the subset containing the first two layers of the multiplex, while \((247)\) refers to the second, fourth, and seventh layers.
We introduce the following map
\begin{equation}
\label{eq:mapS}
S : \Lambda_d \to \mathcal{P}([d]),
\end{equation}
which assigns to each element of \(\Lambda_d\) the set of indices corresponding to its associated layers. For instance, \(S((12)) = \{1,2\}\) and \(S((123)) = \{1,2,3\}\).  
We also define the complement map
\begin{equation}
\label{eq:barS}
\bar{S} : \Lambda_d \to \mathcal{P}([d]),
\end{equation}
which assigns to each element of \(\Lambda_d\) the complement of its associated layer indices in \([d]\). For example, \(\bar{S}((12)) = \{3,4,\dots,d\}\) and \(\bar{S}((134)) = \{2,5,\dots,d\}\).  

\noindent 
We denote the collection of general motif counts in a $d$-layer multiplex network by
\begin{equation}
    \label{eq:generalmotifsdef}
    \mathbf{X}_F := \bigl( X_F(G_n^{(k)}) \bigr)_{k \in \Lambda_d},
\end{equation}
and for each $k \in \Lambda_d,$ the number of cross-layer motifs $X_F(G_n^{(k)})$ is given by 
\begin{equation}
\label{eq:generalmotif}
X_F(G_n^{(k)}) := \frac{1}{|\mathrm{Aut}(F)|} \sum_{s \in [n]_{|V(F)|}} \prod_{(a,b) \in E(F)} A^{(k)}_{s_a s_b}(G_n),
\end{equation}
where \(A^{(k)}(G_n) = \bigcirc_{i \in S(k)} A^{(i)}\) denotes the adjacency matrix of the intersection graph \(G_n^{(k)}\), $\bigcirc_{i \in S(k)}$ denotes the Hadamard (entry-wise) product of matrices $A^{(i)}$, and \([n]_{|V(F)|}\) is the set of all \( |V(F)| \)-tuples of distinct elements from \([n]\). 

\begin{figure}[H]
    \centering \begin{tikzpicture}[scale=0.7, every node/.style={circle,draw,minimum size=4mm}]

\foreach \i/\x/\y in {
  1/0/0,
  2/2/0,
  3/1/1.5,
  4/3/1.5,
  5/2/3,
  6/0/4.5,
  7/2/4.5,
  8/4/4.5,
  9/5/3,
  10/1/6,
  11/3/6
} {
  \node (n\i) at (\x,\y) {\i};
}

\draw[red,thick] (n1) -- (n2);
\draw[green,thick] (n2) -- (n3);
\draw[red,thick] (n1) -- (n3);
\draw[green,thick] (n3) -- (n5);
\draw[green,thick] (n3) -- (n4);
\draw[red,thick] (n4) -- (n5);
\draw[red,thick] (n5) -- (n11);
\draw[red,thick] (n5) -- (n7);
\draw[red,thick] (n7) -- (n11);
\draw[green,thick] (n6) -- (n7);
\draw[green,thick] (n6) -- (n10);
\draw[green,thick] (n7) -- (n10);
\draw[green,thick] (n11) -- (n8);
\draw[green,thick] (n8) -- (n9);

\node[draw=none, fill=none] at (1.6,-0.8) {\textbf{Layer (1)}};

\begin{scope}[xshift=7cm]

\foreach \i/\x/\y in {
  1/0/0,
  2/2/0,
  3/1/1.5,
  4/3/1.5,
  5/2/3,
  6/0/4.5,
  7/2/4.5,
  8/4/4.5,
  9/5/3,
  10/1/6,
  11/3/6
} {
  \node (m\i) at (\x,\y) {\i};
}

\draw[blue,thick] (m1) -- (m4);
\draw[blue,thick] (m2) -- (m4);
\draw[green,thick] (m3) -- (m4);
\draw[green,thick] (m2) -- (m3);
\draw[green,thick] (m3) -- (m5);
\draw[blue,thick] (m5) -- (m6);
\draw[green,thick] (m11) -- (m8);
\draw[green,thick] (m8) -- (m9);
\draw[blue,thick] (m7) -- (m8);
\draw[blue,thick] (m7) -- (m9);
\draw[green,thick] (m6) -- (m7);
\draw[green,thick] (m7) -- (m10);
\draw[green,thick] (m6) -- (m10);

\node[draw=none, fill=none] at (2,-0.8) {\textbf{Layer (2)}};

\end{scope}

\end{tikzpicture}
    \caption{A $2$-layer multiplex network with $11$ vertices. Edges of layer 1 are colored in red, edges of layer 2 colored in blue and edges present in both layers are colored in green.}
    \label{fig:placeholder}
\end{figure}
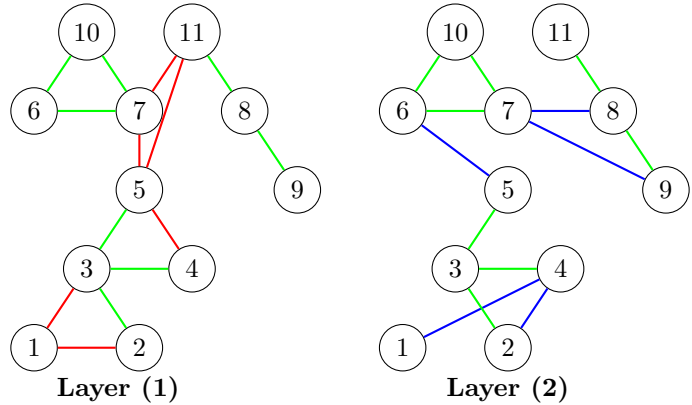

\begin{example}
We provide an example of cross-layer motif counts in a $2$-layer multiplex network in Figure~\ref{fig:placeholder}. Several types of cross-layer motifs can be counted. For instance, the number of shared triangles (i.e., triangles present in both layers) is $1$: the triangle with edge set $\{(6,7), (7,10), (6,10)\}$, whose edges are all colored green. 

The number of triangles that are entirely exclusive to layer 1 (i.e., present only in layer 1 and not in layer 2) is $3$, namely 
$\{(1,2),(2,3),(3,1)\}$, $\{(3,4),(4,5),(3,5)\}$, and $\{(5,7),(7,11),(5,11)\}$. Among these, only one triangle is edge-exclusive to layer 1 (i.e., all three of its edges appear solely in layer~1 and are colored red): $\{(5,7), (7,11), (5,11)\}$.
\end{example}

The objective of the next section (Section \ref{sec:Graphon estimation}) is to construct consistent estimators for the collection of graphons \((W^{(k)})_{k \in \Lambda_d}\) using a motif-counting approach, namely the vector of cross-layer motif counts \(\mathbf{X}_F\) defined in~\eqref{eq:generalmotifsdef}. In particular, when a parameter \(W^{(k)}\) corresponds to a subset of layers \(S(k)\) containing a single element (\(|S(k)| = 1\)), its estimation is based on motif counts in the corresponding single graph \(G^{(i)}\) with \(i \in S(k)\). On the other hand, if \(S(k)\) contains multiple layers (\(|S(k)| > 1\)), the estimation relies on cross-layer motif counts across all graphs \((G^{(i)})_{i \in S(k)}\).

\vspace{0.3em}
Consider a dense exchangeable graph $G_n$ generated by a graphon $W$ according to the Aldous–Hoover theorem \eqref{thm:aldous}. It is easy to see that the expected number of motifs $F$ in $G_n$ is given by
\begin{align} \mathbb{E}[X_F(G_n)]  &= \frac{1}{|\mathrm{Aut}(F)|} \displaystyle \sum_{s \in [n]_{|V(F)|}} \mathbb{E}_{\xi} \left[ \mathbb{E} \left[ \displaystyle \prod_{\{i,j\} \in E(F)} A_{s_i s_j} \biggm| \xi\right] \right] \nonumber \\ &= \frac{1}{|\mathrm{Aut}(F)|} \displaystyle \sum_{s \in [n]_{|V(F)|}} \mathbb{E}_{\xi} \left[ \displaystyle \prod_{\{i,j\} \in E(F)} W(\xi_{s_i}, \xi_{{s_j}})\right] =\frac{\binom{n}{v} v!}{|\mathrm{Aut}(F)|} t(F,W), \label{equ:A} \end{align} 
where $t(F,W)$ denotes the homomorphism density of $F$ in $W$, as defined in Appendix \eqref{eq:homodensdef}.
From Equation \eqref{equ:A}, one can directly use the method of moments in order to estimate the density $t(F,W).$

For a $2$-layer multiplex network, $G_n^{(1)}$ and $G_n^{(2)}$, with parameters $W^{(1)}$,  $W^{(2)}$ and $W^{(12)}$, the expected number of cross-layer motifs shared between both layers (which is equal to the number of motifs in the intersection graph $G_n^{(12)}$  whose expression is given in \eqref{eq:generalmotif}) is
\begin{align} \mathbb{E}[\underbrace{X_F(G_n^{(1)}, G_n^{(2)})}_{=X_F(G_n^{(12)})}]  &= \frac{1}{|\mathrm{Aut}(F)|} \displaystyle \sum_{s \in [n]_{|V(F)|}} \mathbb{E}_{\xi} \left[ \mathbb{E} \left[ \displaystyle \prod_{\{i,j\} \in E(F)} A^{(1)}_{s_is_j}A^{(2)}_{s_is_j} \biggm| \xi\right] \right] = \frac{\binom{n}{v} v!}{|\mathrm{Aut}(F)|} t(F,W^{(12)}), \label{equ:A2} \end{align}
which immediately leads to a straightforward moment-based estimator for $t(F,W^{(12)})$.

We now consider the general case of a multiplex network with $d>2$ layers, characterized by $2^d - 1$ parameters.
Let $k \in \Lambda_d$, and let $S(k) \subseteq [d]$ denote the subset of layers associated with $k$. 
Estimating $W^{(k)}(\cdot)$ requires counting cross-layer motifs that are shared across all graphs $(G_n^{(i)})_{i \in S(k)}$. Under the assumption of aligned nodes in the multiplex framework, the number of such shared motifs coincides with the number of motifs in the intersection graph $G_n^{(S(k))}$, and its expectation is given by
\begin{equation}
\mathbb{E}[X_F(G_n^{(S(k))})] 
= \frac{\binom{n}{|V(F)|} |V(F)|!}{|\mathrm{Aut}(F)|} \, t(F, W^{(k)}),
\end{equation}
where $t(F, W^{(k)})$ denotes the homomorphism density of $F$ in the graphon $W^{(k)}$. Here again, one uses the method of moment to obtain an estimator of $t(F,W^{(k)}).$

\begin{remark}
\label{rem:combination}
Note that while moment-based estimators can be obtained for homomorphism densities, one cannot directly construct a moment-based estimator for the underlying graphons $W^{(k)}$, since $t(F,W^{(k)})$ involves a multiple integral in $W^{(k)}$. To address this, we combine the method of moments with the multiplex stochastic block model framework, as developed in Section~\ref{sec:Graphon estimation}.  
\end{remark}

\begin{remark}

Consider a $d$-layer multiplex Erdős--Rényi model ($d>2$). For $k \in \Lambda_d$, a direct calculation yields
\begin{equation}
\mathbb{E}\Bigl[X_F\big(G_n^{(S(k))}\big)\Bigr]
= \frac{\binom{n}{|V(F)|} \, |V(F)|!}{|\mathrm{Aut}(F)|} \, p_{(k)}^{|E(F)|}.
\label{eq:ERd_expected_shared}
\end{equation}
Hence, in the Erdős--Rényi setting, the method of moments can be applied directly to estimate $p_{(k)}$. This observation will form the backbone of our methodology in Section~\ref{sec:Graphon estimation}, where the multiplex stochastic block model is viewed as a combination of multiplex Erdős--Rényi models.
\end{remark}

\section{Multivariate Graphon Estimation}
\label{sec:Graphon estimation}

In the classical Stochastic Block Model approach (see Appendix \ref{append:SBM}), the estimation of the connectivity matrix entries in \eqref{eq:connect-matrix} and \eqref{eq:connect-matrix-2} relies on counting the number of edges connecting nodes from community $a$ to nodes from community $b$, and dividing by the total number of possible edges between these two communities. Our objective is to extend this edge-counting approach to more complex motifs (such as triangles) spanning multiple communities, thereby combining the multiplex Stochastic Block Model (see Definition \ref{def:MSBM}) with a method-of-moments framework. By allowing connectivity matrices coefficients to be estimated using arbitrary motif counts, the proposed approach provides a more expressive moment-based representation of network structure than edge counts alone. This is not because the underlying graphon is intrinsically sensitive to particular motifs, but rather because collections of motif densities jointly characterize graphons through the equivalence between the subgraph distance and the cut metric \cite{BollobasRiordan2009}.
We do not provide a comparison of estimation quality (in terms of mean squared error rates or minimax upper bounds) between our moment-based approach and the standard SBM (for which such rates have been established in \cite{GaoLuZhou2015}). Nor do we claim that our method outperforms the classical SBM. Instead, our contribution lies in proposing a general framework that incorporates arbitrary motif counts into a moment-based estimation procedure. 

In the following, we present the methodology first for a single-layer network, then for a two-layer multiplex network, and finally for the general case of a 
$d$-layer multiplex network.

\subsection{Case $d =1$}

Let $(G_n)$ be an exchangeable random graph generated by a graphon $W$. Recall the expression for the expected number of motifs in a single network given in \eqref{equ:A}. As noted in Remark \ref{rem:combination}, although the homomorphism density $t(F,W)$ can be estimated using the method of moments \cite{BickelChenLevina2011}, directly estimating the graphon $W$ is considerably more challenging, since $t(F,W)$ involves a multiple integral of $W$ (see Appendix \eqref{eq:homodensdef}).
A natural strategy to bypass this integral is to consider the special case of an Erdős–Rényi graph. Indeed, by \eqref{eq:ERd_expected_shared}, the parameter $p$ can be directly estimated via the method of moments. Building on this idea, one can estimate a general graphon $W$ by leveraging the stochastic block model (SBM) together with the method of moments for Erdős–Rényi graphs. Specifically, an SBM with $K$ communities can be viewed as a mixture of $K$ Erdős–Rényi graphs, each with parameter corresponding to the connectivity matrix entry $\theta_{kk}$ for $k \in {1, \dots, K}$.
Thus, rather than estimating $\theta_{kk}$ simply as the proportion of edges as in Appendix \eqref{eq:connect-matrix}, we treat $\theta_{kk}$ as the parameter of an Erdős–Rényi graph and estimate it using motif counts, as in \eqref{eq:ERd_expected_shared}. 
This approach works well for estimating the diagonal elements of the connectivity matrix $\theta$, that is, $\theta_{kk}$ for $k \in \{1, \dots, K\}$. However, a challenge arises when estimating the off-diagonal elements $\theta_{kl}$ for $k \neq l$, which represent the probability that an edge connects a node from community $k$ to a node from community $l$. Estimating $\theta_{kl}$ requires counting motifs that link nodes from two different Erdős–Rényi graphs. Indeed, $\theta_{kl}$ cannot be viewed as the parameter of a standard Erdős–Rényi graph, but rather as the parameter of a \emph{bipartite Erdős–Rényi graph}, consisting of two disjoint vertex sets corresponding to the two communities. We cannot directly estimate $\theta_{kl}$ using a method of moments based on the same motif $F$ employed for estimating the diagonal elements $\theta_{ii}$. For instance, if $F$ is a triangle, then counting triangles spanning both communities $k$ and $l$ introduces an overlap: the count necessarily includes edges connecting two nodes within the same community, which biases the estimation. By contrast, if $F$ is a $2$-star (i.e., $K_{1,2}$), no such overlap occurs: a $2$-star spanning both communities cannot contain an intra-community edge, due to its bipartite structure. These considerations motivate the introduction of a reduced class of motifs for the estimation of off-diagonal elements of the connectivity matrix. We refer to these as \emph{bichromatic alternate motifs} (see Definition~\ref{def:alternate}).
Note that the estimation of $\theta_{kl}$ is simple in the special case where $F = K_2$, as it reduces to counting the edges connecting nodes in community $k$ to nodes in community $l$, as already done in the classical SBM (see Appendix \eqref{eq:connect-matrix}). To generalize this to arbitrary motifs, we proceed as follows: after partitioning the graph into $K$ communities, we construct $\binom{K}{2}$ graphs $\mathcal{G}_{kl}$, which are node-bicolored, for all $k \neq l \in [K]$. If $n$ is a multiple of $K$, each node-bicolored graph $\mathcal{G}_{kl}$ will contain $h$ nodes from community $k$ colored in red and $h$ nodes from community $l$ colored in blue. Otherwise, the graph will contain $h$ nodes from community $k$ and $h + r$ nodes from community $l$.
The adjacency matrix of $\mathcal{G}_{kl}$ is a block matrix, provided that we reorder the nodes so that the first $h$ nodes are from community $k$ and the remaining nodes are from community $l$. It has the form
\begin{equation}
\tilde{A}^{kl} = \begin{bmatrix} 
\Theta^{kk} & (\Theta^{kl})^T \\ 
\Theta^{kl} & \Theta^{ll} 
\end{bmatrix},
\end{equation}
where 
\begin{align*}
\mathbb{E}[(\Theta^{kk})_{ij}] &= \theta_{kk}, & \Theta^{kk} & \quad \text{is the adjacency matrix of the Erdős–Rényi graph of community $k$},\\
\mathbb{E}[(\Theta^{kl})_{ij}] &= \theta_{kl}, & \Theta^{kl} & \quad \text{is the adjacency matrix of the bipartite Erdős–Rényi graph of communities $k$ and $l$},\\
\mathbb{E}[(\Theta^{ll})_{ij}] &= \theta_{ll}, & \Theta^{ll} & \quad \text{is the adjacency matrix of the Erdős–Rényi graph of community $l$}.
\end{align*}

Let $\hat{z}$ be the MLE estimator of the SBM community membership vector, as defined in Appendix \eqref{eq:MLE}. We can then estimate $\theta_{kk}$ using $\mathbb{E}[X_{F^{\text{red}}}(\mathcal{G}_{kl})]$, $\theta_{ll}$ using $\mathbb{E}[X_{F^{\text{blue}}}(\mathcal{G}_{kl})]$, and $\theta_{kl}$ using $\mathbb{E}[X_{F^{\text{rb alternate}}}(\mathcal{G}_{kl})]$, where 
\begin{itemize}
\item $X_{F^{\text{red}}}(\mathcal{G}_{kl})$ denotes the number of red-node motifs in $\mathcal{G}_{kl}$, that is, the number of motifs $F$ where the nodes belong only to community $k$. It is given by
\begin{equation}
\label{eq:red}
\begin{aligned}
X_{F^{\text{red}}}(\mathcal{G}_{kl}) &= \frac{1}{|\mathrm{Aut}(F)|} \sum_{s \in [h]_{|V(F)|}} \prod_{\{i,j\} \in E(F)} \underbrace{\left[{\tilde{A}}^{kl}_{s_i s_j}\right]_{11}}_{= \Theta^{kk}_{s_i s_j}} \\
&= \frac{1}{|\mathrm{Aut}(F)|} \displaystyle \sum_{s \in [n]_{|V(F)|}} \prod_{(i,j) \in E(F)} A_{s_is_j} \prod_{t \in V(F)} \mathbf{1}\{\hat{z}_{s_t} =k\}.
\end{aligned}
\end{equation}
\item $X_{F^{\text{blue}}}(\mathcal{G}_{kl})$ denotes the number of blue-node motifs in $\mathcal{G}_{kl}$, that is, the number of motifs $F$ where the nodes belong only to community $l$, and is given by
\begin{equation}
\label{eq:blue}
\begin{aligned}
X_{F^{\text{blue}}}(\mathcal{G}_{kl}) & = \frac{1}{|\mathrm{Aut}(F)|} \sum_{s \in [h]_{|V(F)|}} \prod_{\{i,j\} \in E(F)} \underbrace{\left[{\tilde{A}}^{kl}_{s_i s_j}\right]_{22}}_{= \Theta^{ll}_{s_i s_j}}\\
&= \frac{1}{|\mathrm{Aut}(F)|} \displaystyle \sum_{s \in [n]_{|V(F)|}} \prod_{(i,j) \in E(F)} A_{s_is_j} \prod_{t \in V(F)} \mathbf{1}\{\hat{z}_{s_t} =l\}.
\end{aligned}
\end{equation}
\item $X_{F^{\text{rb alternate}}}(\mathcal{G}_{kl})$ denotes the number of alternate bichromatic motifs in $\mathcal{G}_{kl}$ associated to $F$. Alternate motifs are subgraphs in which the nodes belong to both communities $k$ and $l$, but in which there are no adjacent vertices of the same color (see Figure \ref{fig:1} and definition \ref{def:alternate}) . This quantity is given by
\begin{equation}
\label{eq:alternate}
\begin{aligned}
X_{F^{\text{rb alternate}}}(\mathcal{G}_{kl}) &= \frac{1}{|\mathrm{Aut}(F^{\text{alt}})|} \sum_{s \in [h]_{|V(F^{\text{alt}})|}^{\star}} \prod_{\{i,j\} \in E(F^{\text{alt}})} \underbrace{\left[{\tilde{A}}^{kl}_{s_i s_j}\right]_{12}}_{= \Theta^{kl}_{s_i s_j}}\\
&=  \frac{1}{|\mathrm{Aut}(F^{\text{alt}})|} \displaystyle \sum_{s \in [n]_{|V(F^{\text{alt}})|}} \prod_{(i,j) \in E(F^{\text{alt}})} A_{s_is_j} \left( \sum_{c \in \mathcal{C}_{alt}(F^{\text{alt}})} \prod_{t =1}^{|V(F^{\text{alt}})|} \mathbf{1}\{\hat{z}_{s_t} =c_t\}\right),
\end{aligned}
\end{equation}

where 
\begin{itemize}
    \item[-]$F^{\text{alt}}$ is the alternate motif associated to $F$.
    \item[-] $[h]_{|V(F^{\text{alt}})|}^{\star}$ is the set of $|V(F^{\text{alt}})|$-tuples $(s_1, \ldots, s_{|V(F^{\text{alt}})|}) \in [h]^{|V(F^{\text{alt}})|}$, not necessarily with distinct indices,
    \item[-] $\mathcal{C}_{alt}(F^{\text{alt}}) := \{c \in \{k,l\}^{|V(F^{\text{alt}})|} : \forall (i,j) \in E(F^{\text{alt}}), c_i \neq c_j\}.$
\end{itemize}
\end{itemize}

\begin{figure}[ht]
\centering
\caption{Examples of bichromatic alternate motifs.}
\label{fig:1}
\begin{minipage}{0.45\textwidth}
\centering
\begin{tikzpicture}[scale=0.3]
    \node[draw, fill=blue!30, circle] (A) at (0,0) {};
    \node[draw, fill=red!30, circle] (B) at (2,0) {};
    \node[draw, fill=red!30, circle] (C) at (0,2) {};
    \node[draw, fill=blue!30, circle] (D) at (2,2) {};
    \draw[-] (A) -- (B);
    \draw[-] (B) -- (D);
    \draw[-] (D) -- (C);
    \draw[-] (C) -- (A);
\end{tikzpicture}
\end{minipage}
\hfill
\begin{minipage}{0.45\textwidth}
\centering
\begin{tikzpicture}[scale=0.3]
    \node[draw, fill=blue!30, circle] (A) at (0,0) {};
    \node[draw, fill=red!30, circle] (B) at (1,2) {};
    \node[draw, fill=blue!30, circle] (C) at (2,0) {};
    \draw[-] (A) -- (B);
    \draw[-] (B) -- (C);
\end{tikzpicture}
\end{minipage}
\end{figure}

To proceed rigorously, we provide a formal definition of an alternate motif, along with the conditions it must satisfy.  

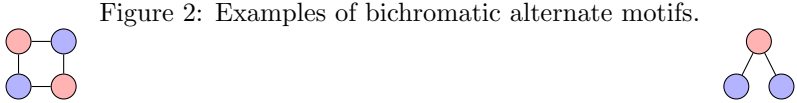
\begin{definition}[Bichromatic/alternate motif]
\label{def:alternate}
A bichromatic alternate motif $F^{\text{alternate}}$ is a motif $F^{\text{alt}}$ in a node-bicolored graph with colors $x$ and $y$, satisfying the following conditions:
\begin{enumerate}
    \item A subset of nodes $(v_i)_{i \in I} \subseteq V(F^{\text{alt}})$ are colored with color $x$.
    \item The remaining nodes $(v_i)_{i \in V(F^{\text{alt}}) \setminus I}$ are colored with color $y$.
    \item No two nodes of the same color are adjacent, i.e., $v_i \bcancel{\sim} v_j$ for all $i, j \in I$. Equivalently, edges can only exist between nodes of different colors.
\end{enumerate}

Equivalently to the preceding conditions, a motif $F^{\text{alt}}$ is alternate if the set $\mathcal{C}_{alt} (F^{\text{alt}})$, as introduced in \eqref{eq:alternate} is non-empty. 

Moreover, given any motif $F$ in a node-bicolored graph $\mathcal{G}_{kl}$, one can associate a corresponding bichromatic alternate motif $F^{\mathrm{alt}}$ by successively removing all edges connecting vertices of the same color until the above conditions are satisfied. Thus, the alternate motif $F^{\mathrm{alt}}$ associated with a motif $F$ is defined as the largest (in the sense of inclusion) subgraph of $F$ in which no two adjacent vertices belong to the same community. 
\end{definition}

Let us take an example to illustrate the concept of alternate motif. 
\begin{example}
Consider $K = 10$ communities and community size $h = 4$. For two communities $k$ and $l$, the bicolored graph $\mathcal{G}_{kl}$ is illustrated below:

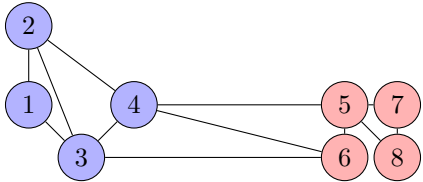
\begin{figure}[ht]
\centering
\begin{tikzpicture}[scale=0.7]
\node[draw, fill=blue!30, circle] (1) at (0,1) {1};
\node[draw, fill=blue!30, circle] (2) at (0,2.5) {2};
\node[draw, fill=blue!30, circle] (3) at (1,0) {3};
\node[draw, fill=blue!30, circle] (4) at (2,1) {4};
\node[draw, fill=red!30, circle] (5) at (6,1) {5};
\node[draw, fill=red!30, circle] (6) at (6,0) {6};
\node[draw, fill=red!30, circle] (7) at (7,1) {7};
\node[draw, fill=red!30, circle] (8) at (7,0) {8};

\draw[-] (1) -- (2);
\draw[-] (2) -- (3);
\draw[-] (1) -- (3);
\draw[-] (2) -- (4);
\draw[-] (3) -- (4);
\draw[-] (4) -- (5);
\draw[-] (4) -- (6);
\draw[-] (5) -- (6);
\draw[-] (5) -- (7);
\draw[-] (7) -- (8);
\draw[-] (5) -- (8);
\draw[-] (3) -- (6);
\end{tikzpicture}
\end{figure}

If the chosen motif is $F = K_3$ (a triangle), then the corresponding alternate motif is the $2$-star $K_{1,2}$. In this case, we have
\[
X_{F^{\text{rb alternate}}}(\mathcal{G}_{kl}) = 2, \quad X_{F^{\text{red}}}(\mathcal{G}_{kl}) = 1, \quad X_{F^{\text{blue}}}(\mathcal{G}_{kl}) = 2.
\]
\vspace{0.1mm}
\begin{remark}
Observe that bichromatic alternate motifs correspond precisely to the motifs in the bipartite subgraph of $\mathcal{G}_{kl}$ induced by the disjoint vertex sets of communities $k$ and $l$.
\end{remark}

\begin{remark}
From the equivalence between the subgraph distance and the cut norm, we know in general that the collection of all subgraph densities characterizes a graphon. Using alternate motifs to estimate the off-diagonal entries of the connectivity matrix does have some limitations, as it does not retain all the information in the subgraph densities. Still, it captures more than just edge counts: it extracts the part of the subgraph information that can propagate across communities and thus remains useful for estimating off-diagonal entries.
\end{remark}

\end{example}

\medskip
\noindent
Having defined the necessary tools and approach, we now present the methodology employed to estimate the graphon $W$ using motif counts. The procedure begins by estimating the community assignment vector $\widehat{z}$ via the maximum likelihood estimator (MLE), as in Appendix \eqref{eq:MLE}. The connectivity matrix $\theta$ is estimated from motif counts as follows:  

\begin{enumerate}  
    \item Construct the $\binom{K}{2}$ bicolored graphs $\mathcal{G}_{kl}$, where nodes from community $k$ are colored red and nodes from community $l$ are colored blue.  
    \item For each bicolored graph $\mathcal{G}_{kl}$, we apply the method of moments to obtain  
    \begin{equation}
    \label{eq:estimates1}
    \widehat{\theta}_{kk} = f^{\text{red}}\left(X_{F^{\text{red}}}(\mathcal{G}_{kl})\right), \quad
    \widehat{\theta}_{ll} = f^{\text{blue}}\left(X_{F^{\text{blue}}}(\mathcal{G}_{kl})\right), \quad
    \widehat{\theta}_{kl} = f^{\text{alt}}(X_{F^{\text{rb alt}}}(\mathcal{G}_{kl})),
    \end{equation}  
    where $X_{F^{\text{red}}}(\mathcal{G}_{kl})$, $X_{F^{\text{blue}}}(\mathcal{G}_{kl})$, and $X_{F^{\text{rb alt}}}(\mathcal{G}_{kl})$ are given in \eqref{eq:red}, \eqref{eq:blue}, and \eqref{eq:alternate}, respectively and $f^{\text{red}}, f^{\text{blue}}$, $f^{\text{alt}}$  are the functions given by the method of moments : for example, $f^{\text{red}}$ is given by inverting $\mathbb{E}\left[X_{F^{\text{red}}}(\mathcal{G}_{kl})\right] = \frac{|[h]_{|V(F)|}|}{|\mathrm{Aut}(F)|} \theta_{kk}^{|E(F)|}$.
    \item This yields estimates for all entries of $\theta$. Note that estimating $\theta_{kk}$ does not require bicolored graphs, since it can be obtained directly from motif counts within each community. Bicolored graphs are only needed for estimating the off-diagonal entries $\theta_{kl}$ with $k \neq l$.  
\end{enumerate}  

Finally, the graphon $W$ is estimated using the network histogram estimator as in Appendix \eqref{eq:networkhistogram}
\begin{equation}  
    \widehat{W}(x, y; h) = \widehat{\theta}_{\min\{\lceil nx/h \rceil, K\}, \, \min\{\lceil ny/h \rceil, K\}},  
    \qquad 0 < x, y < 1.  
\end{equation}

\subsection{Case $d=2$}

Let $(G_n^{(1)}, G_n^{(2)})$ be a $2$-layer multiplex network characterized by graphons $W^{(1)}$, $W^{(2)}$, and $W^{(12)}$. To estimate these graphons, we must estimate three $K \times K$ connectivity matrices: $\theta^{(1)}$, $\theta^{(2)}$, and $\theta^{(12)}$. The two-layer MSBM can be viewed as a mixture of $K$ $2$-layer multiplex Erdős–Rényi networks (recall Definition \ref{def:MER}). Recall from \eqref{eq:ERd_expected_shared} that for two  Erdős–Rényi graphs $G_n^{(1)}$ and $G_n^{(2)}$ with adjacency matrices $A^{(1)}$ and $A^{(2)}$, we have  
\[
\mathbb{E}[X_F(G_n^{(1)}, G_n^{(2)})] 
=  \frac{\binom{n}{|V(F)|} |V(F)|!}{|\mathrm{Aut}(F)|} \, p_{(12)}^{|E(F)|},
\]
where $X_F(G_n^{(1)}, G_n^{(2)})$ denotes the number of cross-layer motifs shared across both layers, and 
$
p_{(12)} = \mathbb{E}(A^{(1)}_{ij} \cdot A^{(2)}_{ij})
$
is the cross-moment of the edge variables in the two graphs.

The estimation strategy extends the single-layer approach developed in the previous subsection. After partitioning the two-layer network into $K$ aligned communities, we form, for each layer, the $\binom{K}{2}$ bicolored graphs $\mathcal{G}_{kl}^{(1)}$ (from layer 1) and $\mathcal{G}_{kl}^{(2)}$ (from layer 2), where each graph contains $h$ nodes from community $k$ (colored red) and $h$ nodes from community $l$ (colored blue).

For each pair $(\mathcal{G}_{kl}^{(1)}, \mathcal{G}_{kl}^{(2)})$ of bicolored graphs across the two layers, we define three block matrices:  
\begin{equation}
{\tilde{A}^{kl}}{}^{(i)} = 
\begin{bmatrix}
    {\Theta^{kk}}^{(i)} & {({\Theta^{kl}}^{(i)})}^T \\
    {\Theta^{kl}}^{(i)} & {\Theta^{ll}}^{(i)}
\end{bmatrix}, 
\qquad i \in \{1,2\},
\end{equation}
and  
\begin{equation}
\begin{aligned}
{\tilde{A}^{kl}}{}^{(12)} 
&= \begin{bmatrix}
    {\Theta^{kk}}^{(12)} & {({\Theta^{kl}}^{(12)})}^T \\
    {\Theta^{kl}}^{(12)} & {\Theta^{ll}}^{(12)}
\end{bmatrix} \\
&= \begin{bmatrix}
    {\Theta^{kk}}^{(1)} \circ {\Theta^{kk}}^{(2)} 
    & {({\Theta^{kl}}^{(1)})}^T \circ {({\Theta^{kl}}^{(2)})}^T  \\
    {\Theta^{kl}}^{(1)} \circ {\Theta^{kl}}^{(2)}  
    & {\Theta^{ll}}^{(1)} \circ {\Theta^{ll}}^{(2)},
\end{bmatrix} 
\end{aligned}
\end{equation}
where $\circ$ denotes the Hadamard product of matrices. The block matrix ${\tilde{A}^{kl}}{}^{(12)}$ encodes cross-layer correlations between the two bicolored graphs. In particular, $
\mathbb{E}\!\left(({\Theta^{kl}}^{(12)})_{ij}\right) = \theta^{(12)}_{kl}.
$

We estimate $\theta^{(i)}$ for $i \in \{1,2\}$ following the procedure described in the single-layer case. For the cross-layer connectivity matrix $\theta^{(12)}$, the estimation proceeds as follows: for each pair $(\mathcal{G}_{kl}^{(1)}, \mathcal{G}_{kl}^{(2)})$ and using the method of moments, we estimate $\theta^{(12)}_{kk}$ using $X_{F^{\text{red}}}(\mathcal{G}^{(1)}_{kl}, \mathcal{G}^{(2)}_{kl})$, the number of cross-layer red-node motifs; $\theta^{(12)}_{ll}$ using $X_{F^{\text{blue}}}(\mathcal{G}^{(1)}_{kl}, \mathcal{G}^{(2)}_{kl})$, the  number of cross-layer blue-node motifs; and $\theta^{(12)}_{kl}$ using $X_{F^{\text{rb alternate}}}(\mathcal{G}^{(1)}_{kl}, \mathcal{G}^{(2)}_{kl})$, the number of cross-layer alternate bichromatic motifs associated to $F$, where 
\begin{align}
X_{F^{\text{red}}}(\mathcal{G}^{(1)}_{kl}, \mathcal{G}^{(2)}_{kl}) &= \frac{1}{|\mathrm{Aut}(F)|} \sum_{s \in [h]_{|V(F)|}} \prod_{\{i,j\} \in E(F)} \underbrace{\left[{\tilde{A}^{kl}}{}^{(12)}_{s_is_j}\right]_{11}}_{= {\Theta^{kk}}^{(12)}_{s_i s_j}}, \\
X_{F^{\text{blue}}}(\mathcal{G}^{(1)}_{kl}, \mathcal{G}^{(2)}_{kl}) &= \frac{1}{|\mathrm{Aut}(F)|} \sum_{s \in [h]_{|V(F)|}} \prod_{\{i,j\} \in E(F)} \underbrace{\left[{\tilde{A}^{kl}}{}^{(12)}_{s_is_j}\right]_{22}}_{= {\Theta^{ll}}^{(12)}_{s_i s_j}}, \\
X_{F^{\text{rb alternate}}}(\mathcal{G}^{(1)}_{kl}, \mathcal{G}^{(2)}_{kl}) &= \frac{1}{|\mathrm{Aut}(F^{\text{alt}})|} \sum_{s \in [h]_{|V(F^{\text{alt}})|}^{\star}} \prod_{\{i,j\} \in E(F^{\text{alt}})} \underbrace{\left[{\tilde{A}^{kl}}{}^{(12)}_{s_is_j}\right]_{12}}_{= {\Theta^{kl}}^{(12)}_{s_i s_j}}.
\end{align}

As in the single-layer setting, estimating the diagonal entries $\theta_{kk}^{(1)}$, $\theta_{kk}^{(2)}$, and $\theta_{kk}^{(12)}$ for $1 \leq k \leq K$ does not require constructing bicolored graphs. In particular, for a given $k \in \{1,\dots,K\}$, $\theta_{kk}^{(12)}$ can be estimated directly from the number of cross-layer motifs $F$ between community $k$ in layer 1 and community $k$ in layer 2, regardless of node colors. 

\medskip
\noindent 
To illustrate how to count cross-layer red, blue and alternate motifs, we present a concrete example.

\begin{example}
Consider a $2$-layer multiplex network with $10$ communities and community size $h=4$.

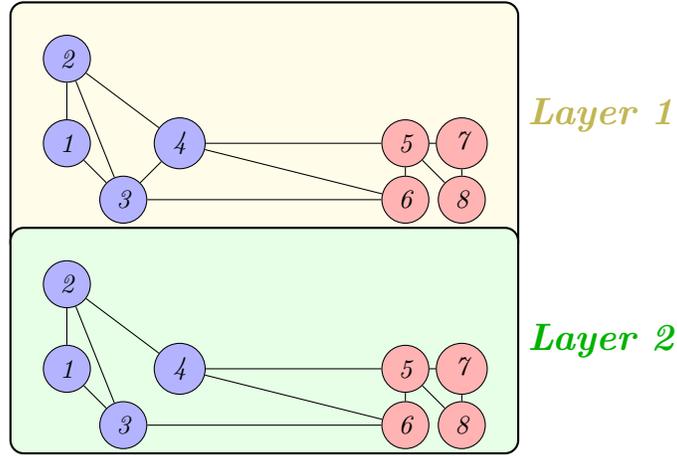
\begin{center}
\begin{tikzpicture}[scale=0.75]
    \fill[yellow!10, rounded corners=5pt, draw=black, thick]
        (-1,-1) rectangle (8,3.5);
    \fill[green!10, rounded corners=5pt, draw=black, thick]
        (-1,-4.5) rectangle (8,-0.5);

    \node[anchor=west, text=yellow!70!black, font=\Large\bfseries] at (8,1.5) {Layer 1};
    \node[anchor=west, text=green!70!black, font=\Large\bfseries] at (8,-2.5) {Layer 2};

    \begin{scope}
        \node[draw, fill=blue!30, circle] (1) at (0,1) {1};
        \node[draw, fill=blue!30, circle] (2) at (0,2.5) {2};
        \node[draw, fill=blue!30, circle] (3) at (1,0) {3};
        \node[draw, fill=blue!30, circle] (4) at (2,1) {4};
        \node[draw, fill=red!30, circle] (5) at (6,1) {5};
        \node[draw, fill=red!30, circle] (6) at (6,0) {6};
        \node[draw, fill=red!30, circle] (7) at (7,1) {7};
        \node[draw, fill=red!30, circle] (8) at (7,0) {8};

        \draw[-] (1) -- (2);
        \draw[-] (2) -- (3);
        \draw[-] (1) -- (3);
        \draw[-] (2) -- (4);
        \draw[-] (3) -- (4);
        \draw[-] (4) -- (5);
        \draw[-] (4) -- (6);
        \draw[-] (5) -- (6);
        \draw[-] (5) -- (7);
        \draw[-] (7) -- (8);
        \draw[-] (5) -- (8);
        \draw[-] (3) -- (6);
    \end{scope}

    \begin{scope}[yshift=-4cm]
        \node[draw, fill=blue!30, circle] (1) at (0,1) {1};
        \node[draw, fill=blue!30, circle] (2) at (0,2.5) {2};
        \node[draw, fill=blue!30, circle] (3) at (1,0) {3};
        \node[draw, fill=blue!30, circle] (4) at (2,1) {4};
        \node[draw, fill=red!30, circle] (5) at (6,1) {5};
        \node[draw, fill=red!30, circle] (6) at (6,0) {6};
        \node[draw, fill=red!30, circle] (7) at (7,1) {7};
        \node[draw, fill=red!30, circle] (8) at (7,0) {8};

        \draw[-] (1) -- (2);
        \draw[-] (2) -- (3);
        \draw[-] (1) -- (3);
        \draw[-] (2) -- (4);
        \draw[-] (4) -- (5);
        \draw[-] (4) -- (6);
        \draw[-] (5) -- (6);
        \draw[-] (5) -- (7);
        \draw[-] (7) -- (8);
        \draw[-] (5) -- (8);
        \draw[-] (3) -- (6);
    \end{scope}

\end{tikzpicture}
\end{center}

Let $F = K_3$, i.e., the triangle motif. The corresponding alternate motif associated with $F$ is the $2$-star $K_{1,2}$, whose two leaves are either both red or both blue. Then we have
\[
X_{F^{\text{rb alternate}}}(\mathcal{G}^{(1)}_{kl}, \mathcal{G}^{(2)}_{kl}) = 2, \quad
X_{F^{\text{red}}}(\mathcal{G}^{(1)}_{kl}, \mathcal{G}^{(2)}_{kl}) = 1, \quad
X_{F^{\text{blue}}}(\mathcal{G}^{(1)}_{kl}, \mathcal{G}^{(2)}_{kl}) = 1.
\]
\end{example}

Having defined the necessary tools and approach, we now present the methodology employed to estimate the graphons $W^{(i)}, i \in \{1,2,(12)\}$ using motif counts. 

\begin{enumerate}
    \item Estimate the community assignment vector $\widehat{z}$ using the profile MLE, as in Appendix \eqref{eq:profileMLE}.  
    \item Estimate the connectivity matrices $\theta^{(i)}$, for $i \in \{1,2,(12)\}$, as follows:  
    \begin{enumerate}
        \item Form $\binom{K}{2}$ pairs of bicolored graphs $(\mathcal{G}^{(1)}_{kl}, \mathcal{G}^{(2)}_{kl})$, where nodes of community $k$ are red and nodes of $l$ are blue.
        \item For each pair, provide moment estimators of :
        \begin{itemize}
            \item $\theta^{(i)}_{kk}, \theta^{(i)}_{kl}, \theta^{(i)}_{ll}$ for $i \in \{1,2\}$ with the same procedure as in the case $d=1$,
            \item $\theta^{(12)}_{kk}$ using $X_{F^{\text{red}}}(\mathcal{G}^{(1)}_{kl}, \mathcal{G}^{(2)}_{kl})$, 
            \item $\theta^{(12)}_{ll}$ using $X_{F^{\text{blue}}}(\mathcal{G}^{(1)}_{kl}, \mathcal{G}^{(2)}_{kl})$,
            \item $\theta^{(12)}_{kl}$ using $X_{F^{\text{rb alternate}}}(\mathcal{G}^{(1)}_{kl}, \mathcal{G}^{(2)}_{kl})$.
        \end{itemize}
    \end{enumerate}
    \item Finally, this provides estimations for all entries of $\theta^{(1)}, \theta^{(2)}, \theta^{(12)}$, and the corresponding graphons are estimated using the network histogram 
    \begin{equation}
\widehat{W}^{(i)}(x,y;h) = \widehat{\theta}^{(i)}_{\min\{\lceil nx/h \rceil, K\}, \min\{\lceil ny/h \rceil, K\}}, \quad 0 < x,y < 1,
\end{equation}
for $i \in \{1,2,(12)\}$.

\end{enumerate}

\subsection{Case $d>2$}

We now extend the methodology to the case of a $d$-layer multiplex network. First, construct $\binom{K}{2}$ collections of bicolored graphs across the $d$ layers.  
For each pair of communities $(k,l)$ with $k \neq l$, the corresponding collection is denoted by $(\mathcal{G}^{(i)}_{kl})_{i \in [d]}$, where nodes of community $k$ are colored red and nodes of community $l$ are colored blue.  

For each $u \in \Lambda_d$, consider the set of bicolored graphs $(\mathcal{G}^{(S(u))}_{kl})$, where $S$ is the map defined in \eqref{eq:mapS}. We then estimate:  
\begin{itemize}
    \item $\theta^{(i)}_{kk}, \; \theta^{(i)}_{kl}, \; \theta^{(i)}_{ll}$ for $i \in [d]$, using the same procedure as in the single-layer case,  
    \item $\theta^{(u)}_{kk}$ via $X_{F^{\text{red}}}\!\left((\mathcal{G}_{kl}^{(S(u))})\right)$, the expected number of cross-layer red-node motifs shared across the $S(u)$ layers,
    \item $\theta^{(u)}_{ll}$ via $X_{F^{\text{blue}}}\!\left((\mathcal{G}_{kl}^{(S(u))})\right)$, the expected number of cross-layer blue-node motifs shared across the $S(u)$ layers,
    \item $\theta^{(u)}_{kl}$ via $X_{F^{\text{rb alt}}}\!\left((\mathcal{G}_{kl}^{(S(u))})\right)$, the expected number of cross-layer alternate bichromatic motifs associated to $F$ that are shared across the $S(u)$ layers. 
\end{itemize}

This yields estimates for all entries of the connectivity matrices $\theta^{(u)}$, for every $u \in \Lambda_d$. Finally, the corresponding network histogram estimator is given by  \begin{equation}
\label{eq:MOMestimator}
\widehat{W}^{(u)}(x,y;h) 
= \widehat{\theta}^{(u)}_{\min\{\lceil nx/h \rceil, K\}, \, \min\{\lceil ny/h \rceil, K\}}, 
\quad 0 < x,y < 1,
\end{equation}
for $u \in \Lambda_d$.

\section{Asymptotic Joint Fluctuation of Multiplex Moments}
\label{sec:Asymptotics}

In this section, we establish the joint asymptotic distribution of motif counts in multiplex networks. Existing results primarily concern the case of a single exchangeable network generated by a graphon $W$. Indeed, \cite{BickelChenLevina2011} proved that subgraph counts have an asymptotically Gaussian distribution. Subsequently, \cite{BhattacharyaChatterjeeJanson2023} refined this result, showing that the Gaussian limit, which they expressed as a linear stochastic integral, may be degenerate when the graphon $W$ is regular with respect to the subgraph (see Definition~\ref{def:reggraphon}). In such cases, the normalized motif count $X_F(G_n)$, with normalization by $n^{1/2}$ rather than $n$, converges in distribution to the sum of a Gaussian component and an independent, non-Gaussian component given by an infinite weighted sum of centered chi-squared variables.  
After that, \cite{Bhattacharya2024} extended these results to the joint distribution of  
$
\left(X_{F_1}(G_n),\ X_{F_2}(G_n),\ \ldots,\ X_{F_r}(G_n)\right),$ 
that is, the counts of multiple motifs within a single graph generated by a graphon $W$. They showed that for motifs where $W$ is regular, each marginal distribution decomposes into two independent components: a Gaussian term and a bivariate stochastic integral. In contrast, for motifs where $W$ is irregular, the marginals remain purely Gaussian, represented as linear stochastic integrals.  

Here, we extend these results to the multiplex setting. We begin by deriving the joint asymptotic distribution of cross-layer motif counts in a $d$- layer multiplex Erdős--Rényi model (see Definition \ref{def:MER}). We then generalize this result to the case of an exchangeable $d$-layer multiplex network. 

Throughout, we recall that the set  
\[
\Lambda_d = \{1,\,2,\,(12),\,3,\,\ldots,\,(12\ldots d)\}
\]  
defined in \eqref{eq:lambda_set} represents the $2^d-1$ indices corresponding to groups of layers in the multiplex. We also recall the definition of $F-$regularity of a graphon given in Definition \ref{def:reggraphon}.

\subsection{Asymptotic Distribution of Multiplex Erdős–Rényi Moments}

Let $d>2$. Assume \((G_n^{(1)}, G_n^{(2)}, \ldots, G_n^{(d)})\) is a sequence of $d-$layer Erdős–Rényi multiplexes with parameters $\mathbf{p} = (p_{(k)})_{k \in \Lambda_d}$. We define the normalized cross-layer motif count vector \(\mathbf{Z}_F\) by
\begin{equation}
\mathbf{Z}_F := \left[Z_F(G_n^{(k)}) \right]_{k \in \Lambda_d}
\label{eq:zf-def}
\end{equation}
where, for each \(k \in \Lambda_d\),
\[
Z_F(G_n^{(k)}) = \frac{X_F(G_n^{(k)}) - \mathbb{E}(X_F(G_n^{(k)}))}{n^{|V(F)| -1}}.
\]
Then we have the following result. 

\begin{theorem}[\textbf{Asymptotic distribution of multiplex Erdős–Rényi moments}]
\label{thm:main}

The vector \(\mathbf{Z}_F\) defined in Equation~\eqref{eq:zf-def} converges in distribution as \(n \to \infty\):
\[
\mathbf{Z}_F\xrightarrow{D} (N_1, N_2, N_{(12)}, \ldots, N_{(12\ldots d)}),
\]
where \(N = (N_k)_{k \in \Lambda_d} \sim \mathcal{N}(\mathbf{0}, \Sigma)\), and the covariance matrix \(\Sigma = (\sigma_{ij})_{i,j \in \Lambda_d}\) is given by
\begin{equation}
\label{eq:sigmaER}
\sigma_{ij} = \frac{(p_{(i)} p_{(j)})^{|E(F)| - 1}}{2|\Aut(F)|^2} \left[p_{(i)} \wedge p_{(j)} - p_{(i)} p_{(j)}\right].
\end{equation}
\end{theorem}

The proof of Theorem \eqref{thm:main} is given in Appendix \ref{proof:main}. 

\subsection{Asymptotic Distribution of Exchangeable Multiplex Moments}

We now turn to the exchangeable framework. Let $d>2$. Consider a sequence of exchangeable $d$-layer multiplex networks  
$
(G_n^{(1)}, G_n^{(2)}, \ldots, G_n^{(d)})
$
that converges in the joint cut metric in the multivariate graphon (see Definition \ref{def:jointcutmetric}). 
Define the vector of centered and scaled cross-layer subgraph counts by 
\begin{equation}
\mathbf{Z}_F:= \left[Z_F(G_n^{(1)}), Z_F(G_n^{(2)}), \underbrace{Z_F(G_n^{(1)}, G_n^{(2)})}_{=Z_F(G_n^{(12)})}, \ldots, \underbrace{Z_F(G_n^{(1)}, G_n^{(2)}, \ldots, G_n^{(d)})}_{=Z_F(G_n^{(12\ldots d)})} \right], \label{eq:ZFvector}
\end{equation}
which we also denote more compactly as
\begin{equation}
\mathbf{Z}_F= \left[Z_F(G_n^{(k)}) \right]_{k \in \Lambda_d}. \label{eq:ZFcompact}
\end{equation}
We denote by $\Lambda_d'$ the subset of $\Lambda_d$ for which the graphons associated to the indices of $\Lambda'_d$ are $F$-irregular. For each $k \in \Lambda_d$, the standardized cross-layer subgraph count is defined by
\begin{equation}
Z_F(G_n^{(k)}) = 
\begin{cases}
\displaystyle \frac{X_F(G_n^{(k)}) - \mathbb{E}(X_F(G_n^{(k)}))}{n^{|V(F)| - 1}} & \text{if $k \notin \Lambda_d'$}, \\[1em]
\displaystyle \frac{X_F(G_n^{(k)}) - \mathbb{E}(X_F(G_n^{(k)}))}{n^{|V(F)| - \frac{1}{2}}} & \text{if if $k \in \Lambda_d'$}.
\end{cases}
\label{eq:ZFdefinition}
\end{equation}

\noindent
Let \(\mathcal{I}_{F, \{1,2\}}\) be the set of subgraphs $F'$ of \(K_{|V(F)|}\) isomorphic to \(F\), such that  $(1,2) \in E(F')$. Define the covariance matrix $\Sigma := (\sigma_{ij})_{i,j \in \Lambda_d \setminus \Lambda_d'}$ by : 
\begin{equation}
\sigma_{ij} = \frac{1}{2(|V(F)| - 2)!^2} \, \Omega_{ij}, \label{equ:5}
\end{equation}
where
\begin{equation}
\begin{aligned}
\Omega_{ij} &= \sum_{\substack{F'_1 \in \mathcal{I}_{F,\{1,2\}} \\ F'_2 \in \mathcal{I}_{F,\{1,2\}}}} \mathbb{E} \left[ C_{ij}(U_1,U_2) \cdot \phi\left(W_{(i)}(U_1,U_2), W_{(j)}(U_1,U_2) \right) \right] \\
&= \sum_{\substack{F'_1 \in \mathcal{I}_{F,\{1,2\}} \\ F'_2 \in \mathcal{I}_{F,\{1,2\}}}} \displaystyle \int_{[0,1]^2}  C_{ij}(x,y) \cdot \phi\left(W_{(i)}(x,y), W_{(j)}(x,y) \right) \, \mathrm{d}x \,\mathrm{d}y\label{eq:Omega}
\end{aligned}
\end{equation}
The auxiliary terms used in this definition are as follows. For any $F'_1, F'_2 \in \mathcal{I}_{F,\{1,2\}}$ and graphon $W_{(k)}$, define
\begin{equation}
t^{-}_{1,2}(U_1,U_2,F'_1,W) := \mathbb{E} \left[ \prod_{(a,b) \in E(F'_1) \setminus \{(1,2)\}} W(U_a, U_b) \,\middle|\, U_1, U_2 \right], \label{eq:tminus}
\end{equation}
\begin{equation}
C_{ij}(U_1,U_2) := t^{-}_{1,2}(U_1,U_2,F'_1,W_{(i)}) \cdot t^{-}_{1,2}(U_1,U_2,F'_2,W_{(j)}), \label{eq:Cij}
\end{equation}
\begin{equation}
\phi(a,b) := \min(a,b) - ab, \quad \text{for } a,b \in [0,1]. \label{eq:phi}
\end{equation}

We are now ready to state the following result. 

\begin{theorem}[\textbf{Asymptotic distribution of General Multiplex Moments}]
\label{thm:main2}

Let $\mathbf{Z}_F$ be defined as in equations \eqref{eq:ZFvector} and \eqref{eq:ZFcompact}. Then we have the following convergence in distribution:
\begin{equation}
\mathbf{Z}_F\xrightarrow{\text{D}} \mathbf{Z} := (Z_1, Z_2, Z_{(12)}, \ldots , Z_{(12\ldots d)}). \label{eq:Zlimit}
\end{equation}
Each limit variable $Z_k$, indexed by $k \in \Lambda_d$, depends on whether the corresponding graphons are $F$-regular or $F$-irregular.

If $k \in \Lambda_d'$ (i.e., the $W_{(k)}$ are $F$-irregular), then the limiting random variable $Z_k$ is given by:
\begin{equation}
Z_k := \int_0^1 \left\{ \frac{1}{|\mathrm{Aut}(F)|} \sum_{a=1}^{|V(F)|} t_a(x, F, W_{(k)}) - \frac{|V(F)|}{|\mathrm{Aut}(F)|} t(F, W_{(k)}) \right\} \, dB_x. \label{equ:6}
\end{equation}

On the other hand, if $k \notin \Lambda_d'$ (i.e. the $W_{(k)}$ are $F$-regular), then the limiting random variable $Z_k$ is given by :
\begin{equation}
Z_k := N_k + \int_0^1 \int_0^1 \left\{ W_F^{(k)}(x,y) - \frac{|V(F)|(|V(F)|-1)}{2|\mathrm{Aut}(F)|} t(F,W_{(k)}) \right\} \, dB_x \, dB_y, \label{equ:7}
\end{equation}
where $\mathbf{N} := (N_k)_{k \in \Lambda_d \setminus \Lambda_d'}$ is a jointly Gaussian vector with zero mean and covariance matrix $\Sigma$ as defined in equation \eqref{equ:5}, and $\mathbf{N}$ is independent of the Brownian $\{B_{t}\}_{t \in [0,1]}$.
\end{theorem}

The proof of Theorem \eqref{thm:main2} is given in Appendix \ref{append:B}.

\begin{remark}[Erdős–Rényi is a particular case of $W-$F regularity]

Note that the multiplex Erdős–Rényi is a particular case of the general multiplex network where all the parameters $(W^{(k)})_{k \in \Lambda_d}$ satisfy the $F-$regularity condition, for all motifs $F$, since the parameters are in this case constants and not functions.   
\end{remark}
\begin{remark}[Notes on regularity]
\label{rem:jointregularity}

Consider a $2$-layer multiplex network. It is straightforward to verify that the $F$-regularity of the individual graphons $W^{(1)}$ and $W^{(2)}$ does not necessarily imply the $F$-regularity of the joint graphon $W^{(12)}$. 
We illustrate this with a counterexample. Consider a $2$-layer multiplex network where both layers are Erdős–Rényi graphs with parameters 
$
W^{(1)} = W^{(2)} = \frac{1}{2},
$
and the joint graphon is given by 
$
W^{(12)}(x,y) = \frac{1}{2} \, \mathbf{1}\{x < \tfrac{1}{2},\, y < \tfrac{1}{2}\}.
$
Note that this is not a multiplex Erdős–Rényi model, since the joint graphon is not constant.  
Since $W^{(1)}$ and $W^{(2)}$ are constant, they are both $F$-regular.  
Taking $F = K_2$, we have
$
t(K_2, W^{(12)}) = \frac{1}{8}, ~\text{and}~
t_1(x, K_2, W^{(12)}) = t_2(x, K_2, W^{(12)}) = \frac{1}{4}\,\mathbf{1}\{x < \tfrac{1}{2}\}.
$
Hence,
$
\bar{t}(x, K_2, W^{(12)}) \neq t(K_2, W^{(12)}),
$
which shows that $W^{(12)}$ is not $K_2$-regular.  
More concretely, suppose that layer~1 represents friendship relations, while layer~2 represents professional relations. Then, within the subgroup corresponding to $(x,y) \in [0,\tfrac{1}{2})^2$ (say for example at EPFL), individuals have a probability of $0.5$ of being friends and colleagues, whereas outside this subgroup the probability is zero.  
This example illustrates that the $F$-regularity of the joint graphon may differ from that of the marginal graphons. Consequently, when studying the asymptotic distribution of $\bigl(Z_F(G_n^{(k)})\bigr)_{k \in \Lambda_2}$, one may, without loss of generality, assume that some of the graphons are $F$-regular while others are not.

\end{remark}

\begin{remark}[Limiting covariance in mixed regimes]
Consider the limiting covariance between $Z_F(G_n^{(i)})$ and $Z_F(G_n^{(j)})$ in the case where $i \in \Lambda_d'$ and $j \in \Lambda_d \setminus \Lambda_d'$, that is, when the two cross-layer motif counts belong to different regimes. It is straightforward to see that, even if we consider the same normalization, the limiting covariance is equal to $0$. Indeed, the limiting Gaussian component $N_j$ of $Z_F(G_n^{(j)})$ is independent of the Brownian motion and the covariance between multiple stochastic integrals of different orders is zero (see \cite[Chapter~1, Section~1.1.2]{Nualart2006}). Thus, the limiting covariance in mixed regimes always vanishes, and therefore does not require further consideration.
\end{remark}

\section{Asymptotic Hypothesis Tests}
\label{sec:Hypo}

Having derived the asymptotic fluctuations of cross-layer motif counts, we are now ready to design hypothesis tests. Fix an integer $d \in \mathbb{N}$. We derive a first test \eqref{eq:test1} for structural similarity in the exchangeable multiplex model, and a second test \eqref{eq:test4} for conditional edge-wise independence in the multiplex Stochastic Block Model (MSBM) (see Definition \ref{def:MSBM}). For simplicity and interpretability, we first consider the case $d=2$ layers. These tests can be extended to the general case with $d>2$ layers (see Remarks \ref{rem:similarsub} and \ref{rem:indepsubset}).
We introduce the following notations, which are relevant for the second test \eqref{eq:test4}. Let $K \geq 2$ denote the number of blocks in the MSBM. For $a,b \in [K]$ and $i \in \{1,2\}$, let $\mathcal{G}_{ab}^{(i)}$ denote the bipartite Erdős–Rényi graph of layer $i$ with nodes from community $a$ and nodes from community $b$. Denote by $n_{ab}$ its number of vertices. When $a=b$, $\mathcal{G}_{ab}^{(i)}$ is no longer bipartite.  
Furthermore, let $(\theta_{ab}^{(i)})_{i \in \Lambda_2}$ be the parameters of the multiplex Erdős–Rényi (MER) network $(\mathcal{G}_{ab}^{(1)}, \mathcal{G}_{ab}^{(2)})$ (recall Definition~\ref{def:MER} for the definition of the MER). 
Finally, note that when $a \neq b$, the subgraph count $X_F(\mathcal{G}_{ab}^{(i)})$ equals zero if the motif $F$ is not bichromatic alternate (see Definition \ref{def:alternate}). Consequently, for the first test \eqref{eq:test1}, any motif structure $F$ can be used, while the second test \eqref{eq:test4} requires the use of a bichromatic alternate motif.

We introduce the following two hypothesis tests (of course, one can imagine all possible tests using cross-layer motif counts) :

\begin{align}
& \underline{\text{Similarity}} ~ H_0 : \mathbb{E}[X_F(G_n^{(1)})] = \mathbb{E}[X_F(G_n^{(2)})] \quad \text{vs} \quad H_1 : \mathbb{E}[X_F(G_n^{(1)})] \neq \mathbb{E}[X_F(G_n^{(2)})], \label{eq:test1}\\[6pt]
& \underline{\text{Edge-wise independence}} ~ H_0 : P \equiv \Bigl( \forall (a,b) \in [K]^2, ~ \theta_{ab}^{(12)} = \theta_{ab}^{(1)} \cdot \theta_{ab}^{(2)} \Bigr)  
\quad \text{vs} \quad 
H_1 : \neg P. \label{eq:test4} 
\end{align}

\subsection{Testing structural similarity}

The first test \eqref{eq:test1} determines whether the expected number of motifs in layer 1 equals that in layer 2 in the multiplex exchangeable framework. This can be interpreted as the the network structure of layer $1$ is the same as that of layer $2$. An example would be when the layers vis-a-vis each other satisfies some distributional invariance. We mention as possible examples that of stationarity (see \cite{SuvegesOlhede2023}) or exchangeability (see \cite{ChandnaMaugis2026}).

\medskip
In order to test structural similarity \eqref{eq:test1} in the multiplex echangeable model, one needs to construct a joint confidence interval for the vector entries of
\begin{equation}
\mathbf{t}_F := \bigl( t(F, W_{(k)}) \bigr)_{k \in \Lambda_d} = \left( \frac{|\mathrm{Aut}(F)|}{\binom{n}{|V(F)|}|V(F)|!} \mathbb{E} \left( X_F(G_n^{(k)})\right) \right)_{k \in \Lambda_d} .
\label{eq:homdens}
\end{equation}

We follow the approach of \cite{Bhattacharya2024} and extended it to the multiplex setting.
To construct joint confidence intervals for the expected number of cross-layer motifs, it is necessary to approximate the quantiles of their joint asymptotic distribution using a sampling procedure, since the asymptotic distribution depends on multivariate graphon $(W_{(k)})_{k \in \Lambda_d}$ through the 1-point conditional homomorphism density
\(
t_a(x, F, W_{(k)})
\), the homomorphism density $t(F,W_{(k)})$ and
the 2-point conditional motif kernel
\(
{W_{(k)}}_{F}(x, y)
\) as established in Theorem \ref{thm:main2}.
In practice, the multivariate graphon is unknown, and we only observe a finite sample of networks
$
\mathbf{G_n} := \left( G_n^{(k)}\right)_{k \in \Lambda_d}.
$
More precisely, we observe only the finite collection \((G_n^{(1)}, \ldots, G_n^{(d)})\) and for a given \(k \in \Lambda_d\) such that \(|S(k)| \geq 2\), one constructs the graph \(G_n^{(k)}\) as the intersection of the graphs \(\{G_n^{(i)}\}, i \in S(k)\); that is, its adjacency matrix is given by the Hadamard product of the corresponding matrices \(\{A_n^{(i)}\}, i \in S(k)\). 

\noindent
For every $k \in \Lambda_d$, the empirical counterparts of $t_a(x, F, W_{(k)})$ and ${W_{(k)}}_{F}(x, y)$ based on the observed sample $\mathbf{G}_n$ are defined as follows. 
The estimator $\tilde{t}_a(v, F, G_n^{(k)})$ (see Definition~\ref{def:1pt} in Appendix \ref{append:E}) is given by the number of injective 1-point conditional homomorphisms of $F$ in the intersection graph $G_n^{(k)}$. 
Similarly, $\widehat{W}_{F}((u,v), G_n^{(k)})$ (see Definition~\ref{def:2pt} in Appendix \ref{append:E}) is defined as an average of the number of injective 2-point conditional homomorphisms of $F$ in the intersection graph $G_n^{(k)}$. 
These estimators are obtained by adapting the constructions of \cite{Bhattacharya2024} to the multiplex framework.

\noindent
Based on these quantities, we construct a sampling statistic vector $\widehat{\mathbf{Z}}_F := \left( \widehat{Z}_F(k) \right)_{k \in \Lambda_d}$, as defined in \eqref{eq:bootstrap-stat} in Appendix \ref{append:E}. By Proposition~\ref{thm:multiplier-bootstrap} in Appendix \ref{append:E}, this vector converges in distribution to the vector $\mathbf{Z} := \left(Z_k \right)_{k \in \Lambda_d}$ of the asymptotic distributions of cross-layer motif counts. The full details of the sampling procedure are provided in Appendix~\ref{append:E}.

\medskip
Now the sampling statistic $\widehat{\mathbf{Z}}_F$ defined in \eqref{eq:bootstrap-stat} depends on whether the $F$-regularity of $W_{(k)}$ holds, i.e. $k \notin \Lambda_d'$, or fails, i.e. $k \in \Lambda_d'$. In practice, it is not known a priori for which indices $k \in \Lambda_d$ the $F$-regularity of $W_{(k)}$ is satisfied. Hence, we cannot directly rely on the quantiles of the sampling vector $\widehat{\mathbf{Z}}_F$. We need then to test $W_{(k)}-F$-regularity for every $k \in \Lambda_d$. We follow and extend the approach of \cite{Bhattacharya2024}. 
Given the observed sequence of networks $\mathbf{G}_n$ defining the multiplex, we compute a set $S$ of indices $k \in \Lambda_d$ for which the $W_{(k)}-F$-regularity condition is rejected. This set acts as the decision rule of $2^d - 1$ hypothesis tests, each testing whether a given $k$ belongs to $\Lambda_d'$ or not. Once this set of indices $S$ is identified from the observed networks $\mathbf{G}_n$, corresponding to the cases where $F$-irregularity is rejected, we refine the sampling vector in \eqref{eq:bootstrap-stat} by replacing the condition $k \in \Lambda_d'$ with $k \in S$.

\smallskip
\noindent
To proceed, we remind the reader of the definition of $W^{(k)}-F$-regularity (Definition~\ref{def:reggraphon}). While Definition~\ref{def:reggraphon} provides a sufficient condition for $F$-regularity of $W_{(k)}$, a necessary and sufficient condition is that the variance of $Z_k$, as defined in \eqref{equ:6} for $k \in \Lambda_d'$, vanishes. Indeed, in the case of a single network, \cite{BhattacharyaChatterjeeJanson2023} showed that the Gaussian limit of subgraph counts---expressed as a linear stochastic integral---may degenerate whenever the underlying graphon $W$ is regular with respect to the subgraph. The same reasoning applies to cross-layer motif counts, so that the $F$-regularity condition of $W_{(k)}$ is equivalent to the degeneracy of the Gaussian limit $Z_k$ for $k \in \Lambda_d'$.  We have the following lemma: 

\begin{lemma}
\label{prop:SIGMA}
The variance of $Z_k$, as defined in \eqref{equ:6}, is denoted by $\sigma^2(k)$ and is given by: 
\begin{equation}
\sigma^2(k) = \frac{1}{|\mathrm{Aut}(F)|^2}  \left[ \sum_{1 \leq a,b \leq |V(F)|} t\bigl(F \bigoplus_{a,b} F, W_{(k)}\bigr) - |V(F)|^2 \cdot t(F,W_{(k)})^2 \right].
\label{equ:vargen}
\end{equation}
where we recall that
\begin{align*}
t\bigl(F \bigoplus_{a,b} F, W_{(k)} \bigr) &= \mathbb{E}\left[ t_a(U_a,F,W_{(k)}) \cdot t_b(U_b,F,W_{(k)}) \right],
\end{align*}
with $(U_a)$ and $(U_b)$ independent uniform random variables on $[0,1]$.
\end{lemma}

The proof of Lemma \ref{prop:SIGMA} is given in Appendix \ref{append:F}.

\smallskip
\noindent
Hence, for a given $k \in \Lambda_d$, the $W_{(k)}-F$-regularity condition is equivalent to $\sigma^2(k) = 0$, where $\sigma^2(k)$ is defined in \eqref{equ:vargen}. 
We perform a system of $2^d - 1$ hypothesis tests for $W_{(k)}$-$F$-regularity.  
For a given $k \in \Lambda_d$, the $W_{(k)}$-$F$-regularity testing problem can be formulated as
\begin{equation}
H_0 : k \in \Lambda_d \setminus \Lambda'_d \quad (\equiv \sigma^2(k) = 0) 
\quad \text{versus} \quad 
H_1 : k \in \Lambda'_d \quad (\equiv \sigma^2(k) \neq 0).
\label{eq:hypothesis-test}
\end{equation}
For each $k \in \Lambda_d$, the variance $\sigma^2(k)$ in \eqref{equ:vargen} can be consistently estimated from the observed multiplex  
\begin{equation}
\mathbf{G}_n = \big( G_n^{(k)} \big)_{k \in \Lambda_d},
\label{eq:graph-collection}
\end{equation}
by its empirical analogue:

\begin{equation}
\begin{aligned}
\sigma^2(k, \mathbf{G}_n) 
&:= \frac{1}{|\mathrm{Aut}(F)|^2}  
\Bigg[ \sum_{1 \leq a,b \leq |V(F)|} 
\widehat{t}\!\left(F \bigoplus_{a,b} F, G_n^{(k)}\right) \\
&\hspace{3.5cm} - |V(F)|^2 \, \widehat{t}(F, G_n^{(k)}) ^2 \Bigg],
\end{aligned}
    \label{eq:empirical-varicance-gen}   
    \end{equation}
where for $k \in \Lambda_d$, 
\begin{equation}
\begin{aligned}
\widehat{t}\!\left(F \bigoplus_{a,b} F, G_n^{(k)}\right) 
&= \frac{1}{n} \sum_{v=1}^n 
\tilde{t}_a(v, F, G_n^{(k)}) \cdot \tilde{t}_b(v, F, G_n^{(k)}) = \widehat{t}\left(H, G_n^{(k)}\right)
\end{aligned}
\label{eq:empirical-ta-tb}
\end{equation}
with $\tilde{t}_a(v,F,G_n^{(k)})$ defined in \eqref{eq:empirical_1point_sum}, $H = F \displaystyle\bigoplus_{a,b} F$ is the vertex join graph (see Definition \ref{def:vertexjoin}) and for every subgraph $F$, the estimator $\widehat{t}\left(F, G_n^{(k)}\right)$ is defined as the number of injective homomorphisms of $F$ in the intersection graph $G_n^{(k)}$ i.e. :
\begin{equation}
\widehat{t}(F, G_n^{(k)}) 
= \frac{1}{|[n]_{|V(F)|}|} 
\sum_{s \in [n]_{|V(F)|}} 
\prod_{(a,b) \in E(F)} A^{(k)}_{s_a s_b}.
\label{eq:empirical-t}
\end{equation}

\begin{remark}
    Note that instead of using the estimator $\widehat{t}(F, G_n^{(k)})$ as given in \eqref{eq:empirical-t}, one can use the plug-in estimator $t(F, \widehat{W})$, where $\widehat{W}$ is estimated using the MSBM method-of-moment approach as definded in \eqref{eq:MOMestimator} in Section \ref{sec:Graphon estimation}. 
\end{remark}

Having defined the empirical counterpart of the variance $\sigma^2(k)$, we define the set (full details are provided in Appendix \ref{append:G}) : 
\begin{equation}
    S(F,\mathbf{G}_n) := \big\{ k \in \Lambda_d: \sqrt{n}\,\sigma^2(k, \mathbf{G}_n) > 1 \big\}
    \label{eq:rejection-set-gen1}
\end{equation}
of indices for which the hypothesis of $W_{(k)}-F$
-regularity is rejected. 

\medskip
\noindent
Then, the sampling vector, given by \eqref{eq:bootstrap-stat} is refined as
    \begin{equation}
\widehat{Z}_F(k) =
\begin{cases}
\frac{1}{\sqrt{n}} \displaystyle \sum_{v=1}^n \big(\widehat{t}(v,F,G_n^{(k)}) - \bar{t}(F,G_n^{(k)})\big) Z_v, & \text{if } k \in S(F,\mathbf{G}_n), \\[1em]
\frac{1}{n} \displaystyle \sum_{1 \leq u,v \leq n} \big(\widehat{W}_{F}((u,v), G_n^{(k)}) - \overline{W}_F(G_n^{(k)}\big)\big(Z_u Z_v - \delta_{u,v}\big), & \text{if } k \notin S(F,\mathbf{G}_n).
\end{cases}
\label{eq:Zhatgen}
\end{equation}
Let $\hat{q}_{1-\alpha, F, \mathbf{G}_n}$ denote the $(1-\alpha)$-quantile of the distribution of $\|\widehat{\mathbf{Z}}_F\|_2 \,\big|\, \mathbf{G}_n$, where $\widehat{\mathbf{Z}}_F:= \left(\widehat{Z}_F(k) \right)_{k \in \Lambda_d}.$

\medskip 

For $k \in \Lambda_d,$ define 
\begin{equation}
Z_F(G_n^{(k)}) =
\begin{cases}
\displaystyle \frac{X_F(G_n^{(k)}) - \mathbb{E}\big(X_F(G_n^{(k)})\big)}{n^{|V(F)| - \frac{1}{2}}}, & \text{if } k \in S(F,\mathbf{G}_n), \\[1em]
\displaystyle \frac{X_F(G_n^{(k)}) - \mathbb{E}\big(X_F(G_n^{(k)})\big)}{n^{|V(F)| - 1}}, & \text{if } k \notin S(F,\mathbf{G}_n).
\end{cases}
\label{eq:ZFgen}
\end{equation}

Let \begin{equation}
\mathbf{Z}_F := \left( Z_F(G_n^{(k)})\right)_{k \in \Lambda_d}.
\end{equation}

Then, the joint confidence set for $\mathbf{t}_F$, as defined in \eqref{eq:homdens}, is given by 
\begin{equation}
\mathcal{C}(F,\mathbf{G}_n) := \big\{\mathbf{t}_F : \|\mathbf{Z}_F\|_2 \leq \hat{q}_{1-\alpha, F, \mathbf{G}_n} \big\}.
\label{eq:confsetgen}
\end{equation}

\medskip
In practice, test \eqref{eq:test1} is implemented as follows. From the observed multiplex network $\mathbf{G}_n$, and for each $k \in \Lambda_d$, compute the empirical variance $\sigma^2(k, \mathbf{G}_n)$ defined defined in \eqref{eq:empirical-varicance-gen}. Then form the index set
$
S(F, \mathbf{G}_n) ,
$ as in \eqref{eq:rejection-set-gen1}. 

\noindent
Next, compute the sampling vector
\[ 
\widehat{\mathbf{Z}}_F = \bigl(\widehat{Z}_F(k)\bigr)_{k \in \Lambda_d},
\]
according to \eqref{eq:Zhatgen}. Then calculate its Euclidean norm $\|\widehat{\mathbf{Z}}_F\|_2$. Repeating this procedure $B$ times yields the samples
\[
\bigl\{ \|\widehat{\mathbf{Z}}_F^{(b)}\|_2 : b = 1, \ldots, B \bigr\}, 
\]
which in turn allow us to compute the empirical $(1-\alpha)$-quantile $\hat{q}_{1-\alpha, F, \mathbf{G}_n}$.

\noindent
Finally, for each $k \in \Lambda_d$, define the scaling factors
\begin{equation}
r_k = 
\begin{cases}
\displaystyle \frac{1}{n^{|V(F)| - \tfrac{1}{2}}}, & k \in S(F, \mathbf{G}_n), \\[10pt]
\displaystyle \frac{1}{n^{|V(F)| - 1}}, & k \notin S(F, \mathbf{G}_n),
\end{cases}
\label{eq:scaling_factors}
\end{equation}
and set
\begin{equation}
\label{eq:Zgen1}
\mathbf{Z}_F = \bigl(Z_F(G_n^{(1)}),\, Z_F(G_n^{(2)}),\, Z_F(G_n^{(12)})\bigr), 
\end{equation}
where  
\begin{align*}
Z_F(G_n^{(1)}) &= r_1\bigl(X_F(G_n^{(1)}) - t_1\bigr), \\  
Z_F(G_n^{(2)}) &= r_2\bigl(X_F(G_n^{(2)}) - t_2\bigr), \\  
Z_F(G_n^{(1)}, G_n^{(2)}) &= r_{12}\bigl(X_F(G_n^{(1)}, G_n^{(2)}) - t_{12}\bigr).
\end{align*}

\noindent
We now have all the ingredients required to compute in practice the joint confidence set $\mathcal{C}(F,\mathbf{G}_n)$ defined in \eqref{eq:confsetgen}. We are thus ready to test structural similarity (test \eqref{eq:test1}). Consider the hyperplane
\begin{equation}
\label{equ:hyper1}
\mathcal{H}_0 := \{ t \in \mathbf{t}_F : t_1 -t_2 = 0 \}.
\end{equation}
We have the following:

\begin{lemma}
\label{lem:firsttest}
Let $\mathbf{Z}_F$ be as defined in \eqref{eq:Zgen1}. The distance from $\mathbf{Z}_F$ to the hyperplane $\mathcal{H}_0$ defined in \eqref{equ:hyper1} is 
    \begin{equation}
    \label{equ:dist1}
    \mathcal{D} = \frac{r_1 r_2 |X_F(G_n^{(1)}) - X_F(G_n^{(2)})|}{\sqrt{r_1^2 + r_2^2}}.
    \end{equation} 
\end{lemma}

The proof of Lemma \ref{lem:firsttest} is given in Appendix \ref{append:hypo1}. 

\medskip
Let $\mathcal{D}$ be as defined in \eqref{equ:dist1}. Then, the decision rule of test \eqref{eq:test1}, based on the observed values $X_F(G_n^{(1)})$ and $X_F(G_n^{(2)})$ is
\begin{equation}
\begin{cases}
\text{Fail to reject } H_0 & \text{if } \mathcal{D} \leq \hat{q}_{1-\alpha, F, \mathbf{G}_n} \big|_{1,2}, \\[6pt]
\text{Reject } H_0 & \text{if } \mathcal{D} > \hat{q}_{1-\alpha, F, \mathbf{G}_n} \big|_{1,2},
\end{cases}
\label{eq:decision_rule_test1}
\end{equation}
where $\hat{q}_{1-\alpha, F, \mathbf{G}_n} \big|_{1,2}$ denotes the restricted $(1-\alpha)$-quantile of the distribution of 
\[
\left\| \big(\widehat{Z}_F(1), \widehat{Z}_F(2) \big) \right\|_2 \bigg| \mathbf{G}_n.
\]

\begin{remark}
\label{rem:similarsub}
For a $d$-layer multiplex exchangeable network, suppose we want to test structural similarity between the subset of layers $(G_n^{(i)})_{i \in I}$ and the subset  of layers $(G_n^{(j)})_{j \in J}$, where 
$
I = \{i_1,\ldots,i_n\} \subseteq [d], 
J = \{j_1,\ldots,j_m\} \subseteq [d],  I \cap J = \varnothing.
$
Then the null hypothesis can be formulated as
\begin{align}
\label{eq:subset-test}H_0 : P \equiv \Bigl( \mathbb{E}\left[X_F(G_n^{(i_1)}, \ldots, G_n^{(i_n)} )\right] =  \mathbb{E}\left[X_F(G_n^{(j_1)}, \ldots, G_n^{(j_m)} )\right]\Bigr) \nonumber 
\quad \text{vs} \quad
H_1 : \neg P.
\end{align}
\end{remark}

\subsection{Testing edge-wise independence}
The purpose of the second test \eqref{eq:test4} is to determine whether layers 1 and 2 of a MSBM with $K$ blocks are conditionally edge-wise independent. The most direct approach is to verify whether the entries of the adjacency matrix \(A^{(1)}\) are conditionally independent (given the latent variables) from those of \(A^{(2)}\).

One may wonder why we investigate conditional edge-wise independence within the MSBM framework rather than in the more general multivariate graphon setting. In the multivariate graphon model, conditional edge-wise independence, that is 
$A_{ij}^{(1)} \perp A_{ij}^{(2)} \mid \xi \quad \text{for every } (i,j),$ implies that $t(F, W^{(12)}) = t(F, W^{(1)} \cdot W^{(2)}).$
A natural way to test this is to check whether 
\[
\mathbb{E}\left[X_F(G_n^{(1)}, G_n^{(2)})\right] = t(F, \widehat{W}^{(1)} \cdot \widehat{W}^{(2)}),
\]
where \(\widehat{W}^{(i)}\) denotes the method-of-moments estimator defined in \eqref{eq:MOMestimator} in Section \ref{sec:Graphon estimation}. However, in practice, determining a value for \(t(F, \widehat{W}^{(1)} \cdot \widehat{W}^{(2)})\) is technically challenging, as it requires computing motif densities in the product of two estimated graphons. 
By contrast, the MSBM admits a much simpler structure. It can be viewed as a collection of 
$K^2$ multiplex Erdős–Rényi networks (MER). In particular, there are $K$ MER, each corresponding to one community, and $K^2 - K$ bipartite MER, each corresponding to a pair of distinct communities. Hence, testing egde-wise independence at the graphon-level is reduced to testing edge-wise independence in the multiplex Erdős–Rényi model, which is simpler because the joint parameter $p_{(12)}$ factorizes directly into the product $p_{(1)}\cdot p_{(2)}$ unlike the homomorphism density $t(F, W^{(1)} \cdot W^{(2)})$ which does not admit such a simple decomposition. Furthermore, this approach allows us to use the collections of multiplex Erdős–Rényi moments whose joint asymptotic distribution is obtained in Theorem \ref{thm:main}. In this way, conditional edge-wise independence between the layers in the MSBM can be determined through block-by-block edge-wise independence. 

\medskip
The following equivalence holds in the MSBM model: for every $(i,j)$,
\begin{equation}
\label{eq:can}
A_{ij}^{(1)} \perp A_{ij}^{(2)} \mid \xi 
\quad \Longleftrightarrow \quad 
\forall (a,b) \in [K]^2,~\theta_{ab}^{(12)} = \theta_{ab}^{(1)} \cdot \theta_{ab}^{(2)}.
\end{equation}
If the right-hand side of \eqref{eq:can} holds, then all edge pairs \(A_{ij}^{(1)}\) and \(A_{ij}^{(2)}\) are conditionally independent. Rejecting the null hypothesis in test \eqref{eq:test4} corresponds to rejecting such independence, meaning that there exists (conditional) dependence between at least one pair of edges \(A_{ij}^{(1)}\) and \(A_{ij}^{(2)}\) for some \(i,j \in [N]\). 
 
\medskip
In order to test conditional edge-wise independence (test \eqref{eq:test4}), one needs to construct, for every $a,b \in [K]^2$, a joint confidence interval
$\mathcal{C}^{\mathrm{ER}}_{ab}(F, \mathbf{\mathcal{G}_{ab}})$ for the vector entries of
\begin{equation}
\label{eq:tFER}
\mathbf{t}_F^{ab} := \big(\theta_{ab}^{(k)} \big)_{k \in \Lambda_d} = \left(\left(\frac{|\mathrm{Aut}(F)|}{\binom{n_{ab}}{|V(F)|}|V(F)|!}\mathbb{E}[X_F(\mathcal{G}_{ab}^{(k)})]\right)^{\frac{1}{|E(F)|}}\right)_{k \in \Lambda_d}.
\end{equation}
We use the same procedure as in the exchangeable model, with the key difference that we no longer need to worry about testing $F$-regularity, since the Erdős–Rényi model corresponds to an $F$-regular setting for all parameter values.
For given $a,b \in [K]$, consider the vector of normalized motif counts
$\mathbf{Z}_F^{ab} := 
\Big(Z_F(\mathcal{G}_{ab}^{(1)}),
 Z_F(\mathcal{G}_{ab}^{(2)}),\allowbreak
 Z_F(\mathcal{G}_{ab}^{(1)}, \mathcal{G}_{ab}^{(2)}) \Big)$
with the same normalization as in \eqref{eq:zf-def}. As established in Theorem \ref{thm:main}, the vector $\mathbf{Z}_F^{ab}$ converges in distribution to a multivariate Gaussian with covariance matrix $\Sigma_{ab}$, given by \eqref{eq:sigmaER} in terms of the parameters $\theta_{ab}^{(i)}$, for $i \in \Lambda_2$.
Since the parameters $\theta_{ab}^{(i)}$ are unknown and in practice we only observe a finite sample of networks
$
\mathbf{\mathcal{G}_{ab}} := (\mathcal{G}_{ab}^{(k)})_{k \in \Lambda_2},
$
where $\mathcal{G}_{ab}^{(12)}$ is constructed as the intersection graph of $\mathcal{G}_{ab}^{(1)}$ and $\mathcal{G}_{ab}^{(2)}$, it is necessary to provide an empirical estimator of the covariance matrix $\Sigma_{ab}$. 
We consider the plug-in estimator $\widehat{\Sigma}_{ab}$ obtained by replacing the parameters $\theta_{ab}^{(i)}$ with their method-of-moments estimators $\widehat{\theta}_{ab}$ defined in \eqref{eq:estimates1} in Section \ref{sec:Graphon estimation}.
It is straightforward to see that $\widehat{\Sigma}_{ab}$ is a consistent estimator of $\Sigma_{ab}$, since the method-of-moments estimators $\widehat{\theta}_{ab}$ are consistent (by the law of large numbers, the normalized cross-layer motif counts converge almost surely to their expectation). Thus, by Slutsky's theorem, we obtain that $\mathbf{Z}_F^{ab}$ converges in distribution to a multivariate Gaussian $\mathbf{\widehat{N}}_{ab}$ with covariance $\widehat{\Sigma}_{ab} = (\widehat{\sigma}_{ab}(i,j))_{i,j \in \Lambda_2}$, where 
\[
\widehat{\sigma}_{ab}(i,j) = \frac{(\widehat{\theta}^{(i)}_{ab} \widehat{\theta}^{(j)}_{ab})^{|E(F)| - 1}}{2|\Aut(F)|^2} \left[\widehat{\theta}^{(i)}_{ab} \wedge \widehat{\theta}^{(j)}_{ab} - \widehat{\theta}^{(i)}_{ab}\widehat{\theta}^{(j)}_{ab}\right].
\]
Let $\hat{q}_{{{1-\alpha}, F, \mathbf{\mathcal{G}_{ab}}}}
$ denote the $(1-\alpha)$-empirical quantile
of the conditional distribution of $\|\mathbf{\widehat{N}}_{ab}\|_2 \big| \mathbf{\mathcal{G}_{ab}}$. 
Then, the joint confidence set for $\mathbf{t}_F^{ab}$, as defined in \eqref{eq:tFER}, is given by
\begin{equation}
\label{eq:CIER}
\mathcal{C}^{\mathrm{ER}}_{ab}(F, \mathbf{\mathcal{G}_{ab}}) := \{ \mathbf{t}_F^{ab}: \|\mathbf{Z}_F^{ab} \|_2 \big | \leq \hat{q}_{{1 -\alpha}, F, \mathbf{\mathcal{G}_{ab}}}\}.
\end{equation}

\medskip
\noindent
We are now ready to test edgewise independence (test \eqref{eq:test4}). Consider the hyperplane
\begin{equation}
\mathcal{H}_0 := \bigcap_{(a,b) \in [K]^2} \{t \in \mathbf{t}_F^{ab} : {\theta}^{(12)}_{ab} = {\theta}^{(1)}_{ab}{\theta}^{(2)}_{ab}\}.
\label{equ:hyper4}
\end{equation}
We have the following:
\begin{lemma}
\label{lem:fourthtest}
For every $a,b \in [K]^2$, set $c_{ab} := =
\frac{\binom{n_{ab}}{|V(F)|}|V(F)|!}{|\mathrm{Aut}(F)|}.$ The distance from the collection $\{\mathbf{Z}^{ab}_F : (a,b)\in[K]^2\}$ to the set $\mathcal{H}_0$ defined in \eqref{equ:hyper4} is 
\begin{equation}
\label{equ:dist4}
\mathcal D
=
\sqrt{
\sum_{(a,b)\in[K]^2}
f(\theta_{ab}^{(1)*}, \theta_{ab}^{(2)*})
},
\end{equation}
\end{lemma}
where $\theta_{ab}^{(1)*}, \theta_{ab}^{(2)*}$ are the solutions of the non-linear system 
\begin{equation}
\label{eq:system}
\begin{cases}
-\dfrac{2c_{ab}|E(F)|}{n_{ab}^{2(|V(F)|-1)}}(\theta_{ab}^{(1)*})^{|E(F)|-1}
\Big(X_F(\mathcal G_{ab}^{(1)})-c_{ab}(\theta_{ab}^{(1)*})^{|E(F)|}\Big)
\\ \qquad
-2c_{ab}|E(F)|(\theta_{ab}^{(1)*}\theta_{ab}^{(2)*})^{|E(F)|-1}\theta_{ab}^{(2)*}
\Big(X_F(\mathcal G_{ab}^{(1)},\mathcal G_{ab}^{(2)})
- c_{ab}(\theta_{ab}^{(1)*}\theta_{ab}^{(2)*})^{|E(F)|}\Big)
=0,
\\[10pt]
-2c_{ab}|E(F)|(\theta_{ab}^{(2)*})^{|E(F)|-1}
\Big(X_F(\mathcal G_{ab}^{(2)})-c_{ab}(\theta_{ab}^{(2)*})^{|E(F)|}\Big)
\\ \qquad
-2c_{ab}|E(F)|(\theta_{ab}^{(1)*}\theta_{ab}^{(2)*})^{|E(F)|-1}\theta_{ab}^{(1)*}
\Big(X_F(\mathcal G_{ab}^{(1)},\mathcal G_{ab}^{(2)})
- c_{ab}(\theta_{ab}^{(1)*}\theta_{ab}^{(2)*})^{|E(F)|}\Big)
=0,
\end{cases}
\end{equation}
and $f$ is the function defined by 
\begin{equation}
\begin{aligned}
f(x,y)
=
\frac{1}{n_{ab}^{2(|V(F)|-1)}} \Big[
\big(&X_F(\mathcal G_{ab}^{(1)})-c_{ab}(x)^{|E(F)|}\big)^2
+
\big(X_F(\mathcal G_{ab}^{(2)})-c_{ab}(y)^{|E(F)|}\big)^2
\\& +
\big(X_F(\mathcal G_{ab}^{(1)},\mathcal G_{ab}^{(2)})
- c_{ab}(xy)^{|E(F)|}\big)^2\Big].
\end{aligned}
\end{equation}

The proof of Lemma \ref{lem:fourthtest} is given in Appendix \ref{append:hypo2}.

\medskip
Let $\mathcal{D}$ be as defined in \eqref{equ:dist4}. Then, the decision rule of test \eqref{eq:test4}, based on the observed values $X_F(\mathcal{G}_{ab}^{(1)})$, $X_F(\mathcal{G}_{ab}^{(2)})$ and $X_F(\mathcal{G}_{ab}^{(1)},\mathcal{G}_{ab}^{(2)})$ is
\begin{equation}
\label{eq:composite_decision_rule}
\begin{cases}
\text{Fail to reject } H_0 & \text{if } \mathcal{D} \leq \hat{q}_{{{1-\alpha}, F, \mathbf{\mathcal{G}_{ab}}}}, \\
\text{Reject } H_0 & \text{otherwise}.
\end{cases}
\end{equation}

\begin{remark}
\label{rem:indepsubset}
For a $d$-layer MSBM, suppose we want to test whether the subset of layers $(G_n^{(i)})_{i \in I}$ is jointly  edge-wise independent of the subset $(G_n^{(j)})_{j \in J}$ conditionally on the latent variables, where 
$
I = \{i_1,\ldots,i_n\} \subseteq [d], 
J = \{j_1,\ldots,j_m\} \subseteq [d],  I \cap J = \varnothing.
$
Then the null hypothesis can be formulated as
\begin{align}
\label{eq:subset-test}H_0 : P \equiv \Bigl( \forall (a,b) \in [K]^2, ~ \theta_{ab}^{(i_1, \ldots, j_m)} = \theta_{ab}^{(i_1, \ldots, i_n)} \theta_{ab}^{(j_1, \ldots, j_m)} \Bigr) \nonumber 
\quad \text{vs} \quad
H_1 : \neg P.
\end{align}

\end{remark}

\section{Discussion}

In this paper, we developed a framework to study higher-order dependencies in multiplex networks through the introduction of cross-layer motifs. Our approach extends classical motif-based statistics to account for the interactions between layers, providing a characterization of the moments of a multiplex network. By capturing how subgraph patterns are shared across multiple layers, these statistics enable a more nuanced understanding of inter-layer dependence and allow for the estimation of the multivariate graph limit describing a $d$-layer network.

Several difficulties arise when moving from the unilayer to the multiplex exchangeable setting, both for estimation and for inference. On the estimation side, we introduced the multiplex stochastic block model as an approximation of the multivariate graph limit. This allows one to go beyond classical approaches based only on edge densities, and instead use arbitrary cross-layer motif counts. This required setting up a moment-based methodology to express the connectivity matrix coefficients in terms of these higher-order statistics. 

On an other hand, deriving the joint distribution of cross-layer motif counts is also more involved. In particular, defining a multivariate version of the $F$-regularity condition is nontrivial, and we have shown that multiplex networks may exhibit mixed regimes where some parameters satisfy a regularity condition while others do not. This leads to different types of limiting behavior coexisting within the same model. Another issue is testing edge-wise conditional independence across layers. Unlike in the Erdős–Rényi case, homomorphism densities in the exchangeable model do not factorize in a way that makes such tests straightforward. We bypassed this difficulty by working within the multiplex stochastic block model and performing block-by-block tests based on multiplex Erdős–Rényi moments. Overall, this paper addresses the fundamental question of how to define, estimate, and interpret meaningful statistics for multiplex networks that capture interactions across layers.

\appendix

\newpage

\section{Proof of Theorem \ref{thm:main}}
\label{proof:main}

\begin{proof}
Let $k \in \Lambda_d$. We begin by expressing the number of isomorphic copies of the graph \(F\) in \(G_n^{(k)}\) as a generalized U-statistic. Let \( \{Y_{ij} : 1 \leq i < j \leq n\} \) be an i.i.d. sequence of \( U[0,1] \) random variables.  Then for all $k \in \Lambda_d$,
\begin{equation}
\begin{aligned}
\label{eq:u-stat}
X_F(G_n^{(k)}) 
&= \sum_{1 \leq i_1 < \ldots < i_{|V(F)|} \leq n } f^{(k)}\left(Y_{i_1i_2}, \ldots, Y_{i_{|V(F)|-1}i_{|V(F)|}} \right) \\
&:= S_{n,|V(F)|}(f^{(k)}),
\end{aligned}
\end{equation}

where \(f^{(k)}\) is a symmetric kernel function depending on the edge indicators of the graph.

More precisely, the kernel \(f^{(k)}\) is defined as
\begin{equation}
f^{(k)}\left(Y_{12}, \ldots, Y_{|V(F)|-1\,|V(F)|} \right) = \sum_{F' \in \mathcal{I}_F} \prod_{(a,b) \in E(F')} \mathbf{1}\{Y_{ab} \leq p_{(k)}\},
\label{eq:kernel}
\end{equation}
where \(\mathcal{I}_F\) denotes the set of injective homomorphic copies of \(F\), and \(\mathbf{1}\{\cdot\}\) is the indicator function.

Next, we introduce an orthogonal series expansion of the kernel \(f^{(k)}\), which will in turn induce a corresponding expansion for the U-statistic \(S_{n,|V(F)|}(f^{(k)})\). This expansion allows us to isolate the leading term that governs the asymptotic distribution.

In the remainder of the proof, we will show that \(S_{n,|V(F)|}(f^{(k)})\) has the same limiting distribution as the first non-vanishing term in its orthogonal decomposition. We refer the reader to the framework of the orthogonal decomposition of generalized U-statistics, presented in detail in the supplementary material \ref{append:Ortho}.

\medskip
\noindent
To begin, we first recall a lemma from \cite{JansonNowicki1991} that connects the asymptotic behavior of the full U-statistic with its projection onto the component of minimal degree.

\begin{lemma}[\cite{JansonNowicki1991} ]\label{lem:principal-term}
Let $k \in \Lambda_d$. Let \( l \) be the principal degree of \( f^{(k)} \). Then
\begin{equation}
n^{l/2 - |V(F)|} S_{n,|V(F)|}(f^{(k)}) \approx n^{l/2 - |V(F)|} S_{n,|V(F)|}(f^{(k)}_{(l)}),
\label{eq:approximation}
\end{equation}
where the approximation is in the sense of convergence in distribution. Consequently, if
\begin{equation}
n^{l/2 - |V(F)|} S_{n,|V(F)|}(f^{(k)}_{(l)}) \xrightarrow{D} \xi,
\label{eq:limit-principal}
\end{equation}
for some non-degenerate random variable \(\xi\), then
\begin{equation}
n^{l/2 - |V(F)|} S_{n,|V(F)|}(f^{(k)}) \xrightarrow{D} \xi.
\label{eq:limit-full}
\end{equation}
\end{lemma}

We now state the main asymptotic result for the vector of centered and normalized U-statistics, which will allow us to complete the proof of Theorem~\ref{thm:main}.

\begin{proposition}\label{prop:joint-clt}
Let $k \in \Lambda_d$. Suppose that \( f^{(k)} \in L^2(K_{|V(F)|}) \) is symmetric and has principal degree \( l \). Let \(\Gamma_l\) denote the collection of non-isomorphic graphs with \(l\) vertices. Define
\begin{equation}
\mathbf{S} := \left[\underbrace{n^{l/2 - |V(F)|} \left(S_{n,|V(F)|}(f^{(k)}) - \binom{n}{|V(F)|} \mathbb{E}(f^{(k)})\right)}_{=: S_k}\right]_{k \in \Lambda_d}.
\label{eq:Sk-def}
\end{equation}
Then, as \( n \to \infty \),
\begin{equation}
\mathbf{S} \xrightarrow{D} \mathcal{N}(\mathbf{0}, \Sigma),
\label{eq:Sk-convergence}
\end{equation}
where the covariance matrix \(\Sigma = (\sigma_{ij})_{i,j \in \Lambda_d}\) is given by
\begin{equation}
\sigma_{ij} = \frac{1}{[(|V(F)| - l)!]^2} \sum_{G \in \Gamma_l} \frac{1}{|\mathrm{Aut}(G)|} \mathbb{E}\left(f_G^{(i)} \cdot f_G^{(j)}\right).
\label{eq:covariance-sigma}
\end{equation}
\end{proposition}

\begin{proof}[Proof of Proposition~\ref{prop:joint-clt}]
Let \(f^{(k)} \in L^2(K_{|V(F)|})\) be a symmetric kernel. We first define a fully symmetrized version of the U-statistic:
\begin{equation}
S_{n,|V(F)|}^{*}(f^{(k)}) := \sum_{1 \leq i_1 \ne i_2 \ne \ldots \ne i_{|V(F)|} \leq n} f^{(k)}\left(Y_{i_1i_2}, \ldots, Y_{i_{|V(F)|-1}i_{|V(F)|}}\right).
\label{eq:symmetric-U}
\end{equation}
Due to the symmetry of \(f^{(k)}\) and from (\ref{eq:u-decomp}), we can relate this to the standard U-statistic as follows:
\begin{equation}
S_{n,|V(F)|}(f^{(k)}) = \frac{1}{|V(F)|!} S_{n,|V(F)|}^{*}(f^{(k)}) = \sum_{G \subseteq K_{|V(F)|}} \frac{1}{|V(F)|!} S_{n,|V(F)|}^{*}(f_G^{(k)}).
\label{eq:decomp-symmetric}
\end{equation}

The symmetry of \(f^{(k)}\) implies that if \(G_1\) and \(G_2\) are isomorphic subgraphs of \(K_{|V(F)|}\), then their contributions to the symmetric U-statistic are equal:
\[
S^{*}_{n,|V(F)|}(f^{(k)}_{G_1}) = S^{*}_{n,|V(F)|}(f^{(k)}_{G_2}).
\] Moreover, each graph \(G \in \Gamma_t\) is isomorphic to exactly \(\frac{|V(F)|!}{(|V(F)| - t)! |\mathrm{Aut}(G)|}\) subgraphs of \(K_{|V(F)|}\). Using this, we deduce
\begin{equation}
S_{n,|V(F)|}(f^{(k)}) = \sum_{t=0}^{|V(F)|} \sum_{G \in \Gamma_t} \frac{1}{(|V(F)| - t)! |\mathrm{Aut}(G)|} S_{n,|V(F)|}^{*}(f_G^{(k)}).
\label{eq:sum-over-Gamma}
\end{equation}

Now, restricting to the term corresponding to the principal degree \(l\), we get
\begin{equation}
S_{n,|V(F)|}(f_{(l)}^{(k)}) = \sum_{G \in \Gamma_l} \frac{1}{(|V(F)| - l)! |\mathrm{Aut}(G)|} S_{n,|V(F)|}^{*}(f_G^{(k)}).
\label{eq:principal-component}
\end{equation}

Applying Lemma~\ref{lem:principal-term} together with Equation~\eqref{eq:principal-component}, we obtain the following approximation for \(S_k\) defined in Equation~\eqref{eq:Sk-def}:
\begin{equation}
S_k \approx n^{l/2 - |V(F)|} \sum_{G \in \Gamma_l} \frac{1}{(|V(F)| - l)! |\mathrm{Aut}(G)|} S_{n,|V(F)|}^{*}(f_G^{(k)}).
\label{eq:Sk-approximation}
\end{equation}

To study the asymptotic behavior of the vector \(\mathbf{S} = [S_k]_{k \in \Lambda_d}\), we apply the Cramér-Wold device. It suffices to analyze the limiting distribution of linear combinations of the form
\begin{equation}
\sum_{k \in \Lambda_d} \alpha_k S_k,
\label{eq:cw-linear}
\end{equation}
for arbitrary constants \(\{\alpha_k\}_{k \in \Lambda_d}\).

Using Equation~\eqref{eq:Sk-approximation}, we approximate this linear combination as follows:
\begin{align}
\sum_{k \in \Lambda_d} \alpha_k S_k 
&\approx \sum_{k \in \Lambda_d} \alpha_k n^{l/2 - |V(F)|} \sum_{G \in \Gamma_l} \frac{1}{(|V(F)| - l)! \cdot |\mathrm{Aut}(G)|} S^{*}_{n,|V(F)|}(f^{(k)}_G) \notag \\
&= n^{l/2 - |V(F)|} \sum_{G \in \Gamma_l} \frac{1}{(|V(F)| - l)! \cdot |\mathrm{Aut}(G)|} \sum_{k \in \Lambda_d} \alpha_k S^{*}_{n,|V(F)|}(f^{(k)}_G) \notag \\
&= n^{l/2 - |V(F)|} \sum_{G \in \Gamma_l} \frac{1}{(|V(F)| - l)! \cdot |\mathrm{Aut}(G)|} S^{*}_{n,|V(F)|} \left( \sum_{k \in \Lambda_d} \alpha_k f^{(k)}_G \right).
\label{eq:lin-comb-expansion}
\end{align}

To conclude the convergence of this expression, we use the following lemma (see \cite{JansonNowicki1991}):

\begin{lemma} [\cite{JansonNowicki1991}]\label{lem:asymptotic-normal}
Suppose each \(f^{(k)}_G\) is invariant under permutations of the vertex labels, i.e. $f^{(k)} = f^{(k)} \circ \pi$ for any permutation $\pi$ of $\{1,2,\ldots, |V(F)|\}$. Then the family
\[
n^{l/2 - |V(F)|} S^{*}_{n,|V(F)|}(f^{(k)}_G), \quad k \in \Lambda_d,
\]
converges jointly in distribution to a family \(\{\xi_k\}_{k \in \Lambda_d}\) of independent Gaussian random variables with
\[
\xi_k \sim \mathcal{N}\left(0, |\mathrm{Aut}(G)| \cdot \mathbb{E}[(f^{(k)}_G)^2] \right).
\]
\end{lemma}

By linearity and independence, and applying Lemma~\ref{lem:asymptotic-normal} to Equation~\eqref{eq:lin-comb-expansion}, we conclude that
\begin{equation}
\sum_{k \in \Lambda_d} \alpha_k S_k \xrightarrow{\mathcal{D}} \mathcal{N} \left(0, \sum_{i,j \in \Lambda_d} \alpha_i \alpha_j \sigma_{ij} \right),
\end{equation}
where the covariance matrix \((\sigma_{ij})\) is given by
\[
\sigma_{ij} = \frac{1}{[(|V(F)| - l)!]^2} \sum_{G \in \Gamma_l} \frac{1}{|\mathrm{Aut}(G)|} \mathbb{E}[f^{(i)}_G \cdot f^{(j)}_G].
\]

This completes the proof of Proposition~\ref{prop:joint-clt}.

\end{proof}

We now prove that the principal degree of \( f^{(k)} \) is equal to 2 in the Erdős–Rényi case. Note that the set \( \Gamma_2 \) of isomorphism types of graphs on 2 vertices consists of two elements:
\[
\Gamma_2 = \left\{ E_{\{1,2\}},\ K_{\{1,2\}} \right\},
\]
where \( E_{\{1,2\}} \) is the empty graph with two isolated vertices \( \{1\} \) and \( \{2\} \), and \( K_{\{1,2\}} \) is the complete graph on two vertices.

We first show that the first-degree component \( f^{(k)}_{(1)} \) vanishes for all \( k \in \Lambda_d \), while the second-degree component \( f^{(k)}_{(2)} \) does not. From the definition of $f^{(k)}_{(s)}$ in (\ref{eq:principal-degree}), we have that:
\begin{equation}
f^{(k)}_{(1)} = \sum_{a=1}^{|V(F)|} f^{(k)}_{K_{\{a\}}}, \quad \text{where } f^{(k)}_{K_{\{a\}}} = P_{M_{K_{\{a\}}}} f^{(k)}.
\end{equation}
Now, using equations~\eqref{eq:cond-expectation} and~\eqref{eq:proj-sum}, we have:
\begin{align}
f^{(k)}_{K_{\{a\}}} &= P_{M_{K_{\{a\}}}} f^{(k)} = P_{L^2(K_{\{a\}})} f^{(k)} - P_{M_{\emptyset}} f^{(k)} \\
&= \underbrace{\mathbb{E}[f^{(k)} \mid U_a]}_{=\mathbb{E}\left[f^{(k)}\right]} - \mathbb{E}[f^{(k)}] = 0,
\end{align}
since in an Erdős–Rényi (ER) model, the function \( f^{(k)} \) depends only on the \( Y_{ij} \), and not on the latent variables \( U_i \). Therefore,
\begin{equation}
f^{(k)}_{(1)} = 0, \quad \text{for all } k \in \Lambda_d.
\end{equation}

We now verify that the second-degree component \( f^{(k)}_{(2)} \) is non-zero. This component includes contributions from both \( f^{(k)}_{E_{\{a,b\}}} \) and \( f^{(k)}_{K_{\{a,b\}}} \) for \( 1 \leq a < b \leq |V(F)| \). For the empty graph \( E_{\{a,b\}} \), we have:
\begin{align}
f^{(k)}_{E_{\{a,b\}}} &= P_{L^2(E_{\{a,b\}})} f^{(k)} - P_{M_{K_{\{a\}}}} f^{(k)} - P_{M_{K_{\{b\}}}} f^{(k)} - P_{M_{\emptyset}} f^{(k)} \\
&= \mathbb{E}(f^{(k)} \mid U_a, U_b) - \mathbb{E}(f^{(k)} \mid U_a) - \mathbb{E}(f^{(k)} \mid U_b) + \mathbb{E}(f^{(k)})
\\&= \mathbb{E}(f^{(k)}) - \mathbb{E}(f^{(k)}) -\mathbb{E}(f^{(k)}) + \mathbb{E}(f^{(k)})\\
    &=0,
\end{align}
again due to the structure in the ER model. Hence,
\begin{equation}
\mathbb{E}\left[f^{(i)}_{E_{\{1,2\}}} \cdot f^{(j)}_{E_{\{1,2\}}}\right] = 0, \quad \text{for all } i,j \in \Lambda.
\end{equation}

Now consider the component \( f^{(k)}_{K_{\{a,b\}}} \), associated to the complete graph on two vertices. We compute:
\begin{align}
f^{(k)}_{K_{\{a,b\}}} &= P_{L^2(K_{\{a,b\}})} f^{(k)} - P_{M_{E_{\{a,b\}}}} f^{(k)} - P_{M_{K_{\{a\}}}} f^{(k)} - P_{M_{K_{\{b\}}}} f^{(k)} - P_{M_{\emptyset}} f^{(k)} \\
    &= P_{L^2({K_{\{a, b\}}})}f^{(k)} - P_{L^2(E_{\{a,b\}})}f^{(k)}\\
    &= \mathbb{E}(f^{(k)} \mid U_a, U_b, Y_{ab}) - \mathbb{E}(f^{(k)} \mid U_a, U_b)\\
    &= \mathbb{E}(f^{(k)} \mid Y_{ab}) - \mathbb{E}(f^{(k)})\\
    & \neq 0,
\end{align}
so we conclude that:
\begin{equation}
f^{(k)}_{(2)} \neq 0, \quad \text{for all } k \in \Lambda_d.
\end{equation}
Thus, the principal degree \( l \) of \( f^{(k)} \) is equal to 2 in the Erdős–Rényi case.

Next, we derive an explicit expression for \( f^{(k)}_{K_{\{1,2\}}} \). From the computation above, we have:
\begin{align}
f^{(k)}_{K_{\{1,2\}}} &= \mathbb{E}(f^{(k)} \mid Y_{12}) - \mathbb{E}(f^{(k)}) \\
&= \sum_{F' \in \mathcal{I}_F} p_{(k)}^{|E(F)| - 1} \mathbf{1}\{Y_{12} \leq p_{(k)}\} \mathbf{1}\{(1,2) \in E(F')\} - \sum_{F' \in \mathcal{I}_F} p_{(k)}^{|E(F)|} \\
&= \sum_{F' \in \mathcal{I}_{F,\{1,2\}}} p_{(k)}^{|E(F)| - 1} \left(\mathbf{1}\{Y_{12} \leq p_{(k)}\} - p_{(k)}\right),
\end{align}
where \( \mathcal{I}_{F,\{1,2\}} := \{F' \in \mathcal{I}_F : (1,2) \in E(F')\} \).

We then compute the covariance. For all $i,j \in \Lambda_d$ :
\begin{align}
\mathbb{E}\left[f^{(i)}_{K_{\{1,2\}}} \cdot f^{(j)}_{K_{\{1,2\}}}\right] 
&= \sum_{F_1', F_2' \in \mathcal{I}_{F,\{1,2\}}} (p_{(i)}p_{(j)})^{|E(F)| - 1} \cdot \mathbb{E}\left[ \left(\mathbf{1}\{Y_{12} \leq p_{(i)}\} - p_{(i)}\right)\left(\mathbf{1}\{Y_{12} \leq p_{(j)}\} - p_{(j)}\right) \right] \\
&= \sum_{F_1', F_2' \in \mathcal{I}_{F,\{1,2\}}} (p_{(i)}p_{(j)})^{|E(F)| - 1} \cdot (p_{(i)} \wedge p_{(j)} - p_{(i)}p_{(j)}).
\end{align}
Now observe that \( |\mathcal{I}_{F,\{1,2\}}| = \frac{(|V(F)| - 2)!}{|\mathrm{Aut}(F)|} \), so we can write:
\begin{equation}
\mathbb{E}\left[f^{(i)}_{K_{\{1,2\}}} \cdot f^{(j)}_{K_{\{1,2\}}}\right]
= \frac{(|V(F)| - 2)!^2}{|\mathrm{Aut}(F)|^2} \cdot (p_{(i)}p_{(j)})^{|E(F)| - 1} \cdot (p_{(i)} \wedge p_{(j)} - p_{(i)}p_{(j)}).
\end{equation}

Finally, applying Proposition~\ref{thm:main}, and using that \( |\mathrm{Aut}(K_{\{1,2\}})| = 2 \), we obtain:
\begin{equation}
\sigma_{ij} = \frac{1}{2|\mathrm{Aut}(F)|^2} \cdot (p_{(i)}p_{(j)})^{|E(F)| - 1} \cdot (p_{(i)} \wedge p_{(j)} - p_{(i)}p_{(j)}).
\end{equation}

Noting that the vector \( \mathbf{S} \) from Proposition~\ref{thm:main} coincides with \( \mathbf{Z}_F\) when the principal degree \( l = 2 \), we conclude that the asymptotic distribution of the MER moments is established.

\end{proof}

\section{Proof of Theorem \ref{thm:main2}}
\label{append:B}
The proof of Theorem~\ref{thm:main2} builds upon the works of 
\cite{BhattacharyaChatterjeeJanson2023}, \cite{Bhattacharya2024}, and \cite{JansonNowicki1991}. 
A substantial part of the argument can be adapted directly from \cite{Bhattacharya2024} 
by incorporating indices on the graphon, that is, by considering the indexed family 
\((W_{(k)})_{k \in \Lambda_d}\) instead of a single graphon \(W\). 
More precisely, for each \( k \in \Lambda_d \), we define
\begin{align}
X_F(G_n^{(k)}) - \mathbb{E}\!\left[X_F(G_n^{(k)})\right] 
&= \sum_{1 \leq i_1 < \cdots < i_{|V(F)|} \leq n} 
f^{(k)}\!\left(
\begin{array}{l}
U_{i_1}, \ldots, U_{i_{|V(F)|}}, \\[0.3em]
Y_{i_1i_2}, \ldots, Y_{i_{|V(F)|-1}i_{|V(F)|}}
\end{array}
\right)  \\
&=: S_{n,|V(F)|}\!\left(f^{(k)}\right), \label{eq:Ustat_form}
\end{align}
where the kernel \( f^{(k)} \) is given by
\begin{align}
f^{(k)}\big(&U_1, \ldots, U_{|V(F)|}, Y_{12}, \ldots, Y_{|V(F)|-1\, |V(F)|} \big) \nonumber \\
&= \sum_{F' \in \mathcal{I}_F} \prod_{(a,b) \in E(F')} 
\mathbf{1}\!\left\{Y_{ab} \leq W_{(k)}(U_a, U_b)\right\} 
- |\mathcal{I}_F| \cdot t(F, W_{(k)}). \label{eq:f_kernel}
\end{align}

Hence, the convergence towards the linear stochastic integral in the case 
\(k \in \Lambda_d'\) (i.e., when the graphons \(W_{(k)}\) are \(F\)-irregular), 
as well as the convergence towards the bivariate stochastic integral component 
for \(k \in \Lambda_d \setminus \Lambda_d'\), need not be repeated here, 
since they follow directly from \cite[Theorem~2.1]{Bhattacharya2024}. 
The only component requiring a separate treatment is the Gaussian term 
arising when \(k \in \Lambda_d \setminus \Lambda_d'\). 
Using the same notations as in the proof of Theorem~2.1 in \cite{Bhattacharya2024}, 
we directly state the following propositions.

\begin{proposition} \label{prop:2}
Let $\{\xi_s\}_s$ and $\{\tilde{\xi}_s\}_s$ be independent collections of $\mathcal{N}(0,1)$ and $\mathcal{N}(0,2)$ random variables, respectively. For all $k \in \Lambda_d$ and $Z_F(G_n^{(k)})$ as defined in \eqref{eq:ZFdefinition} , we have the convergence in distribution:
\begin{equation}
\mathbf{Z}_F := (Z_F(G_n^{(k)}))_{k \in \Lambda_d} \xrightarrow{\text{D}} \mathbf{Z}:= (Z_k)_{k \in \Lambda_d}. \label{eq:TF_dist_conv}
\end{equation}
The limiting random variables $Z_k$ are given by:
\begin{itemize}
    \item If $k \in \Lambda_d'$, i.e. the $W_{(k)}$ are $F$-irregular:
    \begin{equation}
    Z_k = \frac{1}{(|V(F)| - 1)!} \sum_{s \geq 1} \mathbb{E}\left[f^{(k)} \cdot {\phi_{K_{\{1\}}}}_s \right] \xi_s := Q_{(k)}. \label{eq:Q_k}
    \end{equation}

    \item If $k \in \Lambda_d \setminus \Lambda_d'$, i.e. the $W_{(k)}$ are $F$-regular:
    \begin{align}
    Z_k = \frac{1}{2(|V(F)| - 2)!} \bigg\{&
    \sum_{s \geq 1} \mathbb{E}\left[f^{(k)} \cdot ({\phi_{K_{\{1\}}}}_s \otimes {\phi_{K_{\{1\}}}}_s) \right] (\xi_s^2 - 1) \nonumber \\
    &+ 2 \sum_{s < t} \mathbb{E}\left[f^{(k)} \cdot ({\phi_{K_{\{1\}}}}_s \otimes {\phi_{K_{\{1\}}}}_t) \right] \xi_s \xi_t \nonumber \\
    &+ \sum_{s \geq 1} \mathbb{E}\left[f^{(k)} \cdot {\phi_{K_{\{1,2\}}}}_s \right] \tilde{\xi}_s \bigg\} := R_{(k)} + \tilde{R}_{(k)}, \label{eq:R_k}
    \end{align}
    where $R_{(k)}$ is the sum of the first two terms of $Z_k$ and $\tilde{R}_{(k)}$ is the last term of $Z_k.$
\end{itemize}
\end{proposition}

The proof of proposition \ref{prop:2} follows directly from the proof of Theorem~2.1 in \cite{Bhattacharya2024} (Proposition~A.2).

\begin{proposition} \label{prop:3}
Let $k \in \Lambda_d\setminus \Lambda_d'$. Let $\tilde{R}_{(k)}$ be defined as in Equation~\eqref{eq:R_k}. Then the vector
\[
\tilde{\mathbf{R}} := (\tilde{R}_{(k)})_{k \in \Lambda_d\setminus \Lambda_d'}
\]
converges in distribution to a multivariate normal distribution:
\[
\tilde{\mathbf{R}} \sim \mathcal{N}(\mathbf{0}, \Sigma),
\]
where $\Sigma$ is defined as in ~\eqref{equ:5}.
\end{proposition}

\begin{proof}[Proof of proposition \ref{prop:3}]

The variables $\{\tilde{\xi}_s\}_s$ are i.i.d. $\mathcal{N}(0,2)$. Hence, for each $k \in \Lambda_d \setminus \Lambda_d'$, the linear combination
\[
\tilde{R}_k = \frac{1}{2(|V(F)| - 2)!} \sum_{s \geq 1} \mathbb{E}[f^{(k)} \cdot {\phi_{K_{\{1,2\}}}}_s] \tilde{\xi}_s
\]
is Gaussian, and any linear combination of the $\tilde{R}_k$ is also Gaussian. Therefore, $\tilde{\mathbf{R}}$ is a Gaussian vector.

We compute the covariance matrix $\Gamma = (\gamma_{ij})$ of $\tilde{\mathbf{R}}$. We have:
\begin{align}
\gamma_{ij} =
\begin{cases}
\displaystyle \frac{1}{2(|V(F)| - 2)!^2} \sum_{s \geq 1} \mathbb{E}^2\left[f^{(i)} \cdot {\phi_{K_{\{1,2\}}}}_s\right] & \text{if } i = j, \\
\displaystyle \frac{1}{2(|V(F)| - 2)!^2} \sum_{s \geq 1} \mathbb{E}\left[f^{(i)} \cdot {\phi_{K_{\{1,2\}}}}_s\right] \mathbb{E}\left[f^{(j)} \cdot {\phi_{K_{\{1,2\}}}}_s \right] & \text{if } i \neq j.
\end{cases} \label{eq:gamma_ij}
\end{align}

First, recall that the scalar product between \( f^{(i)}_{K_{\{1,2\}}} \) and \( f^{(j)}_{K_{\{1,2\}}} \) can be expressed as
\begin{equation}
\langle f^{(i)}_{K_{\{1,2\}}}, f^{(j)}_{K_{\{1,2\}}} \rangle = \sum_{s \geq 1} \mathbb{E}\left[f^{(i)} \phi^{(s)}_{K_{\{1,2\}}} \right] \mathbb{E}\left[f^{(j)} \phi^{(s)}_{K_{\{1,2\}}} \right], \label{eq:proj_scalar}
\end{equation}
and the \( L^2 \)-norm satisfies
\begin{equation}
\| f^{(i)}_{K_{\{1,2\}}} \|_2^2 = \sum_{s \geq 1} \mathbb{E}^2\left[f^{(i)} \phi^{(s)}_{K_{\{1,2\}}} \right]. \label{eq:proj_norm}
\end{equation}

From \eqref{eq:proj_scalar} and \eqref{eq:proj_norm}, we deduce that the covariance coefficient \(\gamma_{ij}\) can be written as
\begin{equation}
\gamma_{ij} = \begin{cases}
        \frac{1}{2(|V(F)|-2)!^2} \mathbb{E}\left[(f^{(i)}_{K_{\{1,2\}}})^2\right] & \text{if } i=j,\\[5pt]
        \frac{1}{2(|V(F)|-2)!^2} \mathbb{E}\left[f^{(i)}_{K_{\{1,2\}}} f^{(j)}_{K_{\{1,2\}}}\right] & \text{if } i \neq j.
    \end{cases} \label{eq:gamma_def}
\end{equation}

Recall also that for each \( i \in \Lambda_d \setminus \Lambda_d' \), the projection is given by
\begin{equation}
f^{(i)}_{K_{\{1,2\}}} = \mathbb{E}\left[f^{(i)} \mid U_1, U_2, Y_{12}\right] - \mathbb{E}\left[f^{(i)} \mid U_1, U_2\right]. \label{eq:proj_def}
\end{equation}
By definition of \( f^{(i)} \) in \eqref{eq:f_kernel}, we obtain that
\begin{equation}
f^{(i)}_{K_{\{1,2\}}} = \sum_{F' \in \mathcal{I}_{F, \{1,2\}}} t^{-}_{1,2}(U_1,U_2,F',W_{(i)}) \cdot \left[\mathbf{1}\{Y_{12} \leq W_{(i)}(U_1,U_2)\} - W_{(i)}(U_1,U_2)\right], \label{eq:proj_rep}
\end{equation}
where
\begin{equation}
t^{-}_{1,2}(U_1,U_2,F',W_{(i)}) := \mathbb{E} \left[ \prod_{(a,b) \in E(F')\setminus \{(1,2)\}} W_{(i)}(U_a,U_b) \, \Big| \, U_1, U_2\right]. \label{eq:t_minus}
\end{equation}

Let us now compute the expectation
\begin{align}
&\mathbb{E}_{Y_{12}} \Big[\left(\mathbf{1}_{\{Y_{12} \leq W_{(i)}(U_1,U_2)\}} - W_{(i)}(U_1,U_2)\right)
\left(\mathbf{1}_{\{Y_{12} \leq W_{(j)}(U_1,U_2)\}} - W_{(j)}(U_1,U_2)\right) \Big] \nonumber \\
&= \min\left(W_{(i)}(U_1,U_2), W_{(j)}(U_1,U_2)\right) - W_{(i)}(U_1,U_2) W_{(j)}(U_1,U_2). \label{eq:Y12_cov}
\end{align}

Plugging equations \eqref{eq:proj_rep} and \eqref{eq:Y12_cov} into the computation of the covariance, we obtain
\begin{align}
\mathbb{E}\left[f^{(i)}_{K_{\{1,2\}}} f^{(j)}_{K_{\{1,2\}}}\right] &= \sum_{\substack{F'_1 \in \mathcal{I}_{F, \{1,2\}}\\ F'_2 \in \mathcal{I}_{F, \{1,2\}}}}  
\mathbb{E} \bigg[ t^{-}_{1,2}(U_1,U_2,F'_1,W_{(i)}) \cdot t^{-}_{1,2}(U_1,U_2,F'_2,W_{(j)}) \nonumber \\ 
&\quad \times \big(\mathbf{1}\{Y_{12} \leq W_{(i)}(U_1,U_2)\} - W_{(i)}(U_1,U_2) \big) \nonumber \\
&\quad \times \big(\mathbf{1}\{Y_{12} \leq W_{(j)}(U_1,U_2)\} - W_{(j)}(U_1,U_2) \big)  
\bigg] \nonumber\\
&=\sum_{\substack{F'_1, F'_2 \in \mathcal{I}_{F, \{1,2\}}}} \mathbb{E} \Big[ t^{-}_{1,2}(U_1,U_2,F'_1,W_{(i)}) \cdot t^{-}_{1,2}(U_1,U_2,F'_2,W_{(j)}) \nonumber \\
&\qquad \times \left(\min(W_{(i)}(U_1,U_2), W_{(j)}(U_1,U_2)) - W_{(i)}(U_1,U_2) W_{(j)}(U_1,U_2)\right) \Big]. \label{eq:cov_proj}
\end{align}

Finally, observe that
\begin{equation}
\mathbb{E}\left[t^{-}_{1,2}(U_1,U_2,F',W_{(i)}) \cdot W_{(i)}(U_1,U_2) \right] = \mathbb{E}\left[t_{1,2}(U_1,U_2,F',W_{(i)})\right], \label{eq:tminus_t}
\end{equation}
which follows from the tower property of conditional expectation. Hence, combining equations \eqref{eq:cov_proj} and \eqref{eq:tminus_t} with the scaling in \eqref{eq:gamma_def}, we conclude that
\begin{equation}
\frac{1}{2(|V(F)| - 2)!^2} \mathbb{E}\left[f^{(i)}_{K_{\{1,2\}}} f^{(j)}_{K_{\{1,2\}}}\right] = \Omega_{ij}, \label{eq:Omega_def}
\end{equation}
where \(\Omega_{ij}\) is the entry of the covariance matrix \(\Sigma\) defined in \eqref{eq:Omega}. This completes the proof of Proposition \ref{prop:3}.

\end{proof}

\section{Sampling Statistic of the Asymptotic Distribution}
\label{append:E}
In this section, we give the details of the construction of the sampling statistic, which allows us to estimate the limiting joint distribution of the multiplex moments established in Theorem \ref{thm:main2}. The approach extends that of \cite{Bhattacharya2024}.
We first define the empirical counterpart of the 1-point  conditional density. 

\begin{definition}[Empirical 1-Point Conditional Density]
\label{def:1pt}
The empirical version of \eqref{eq:1point_cond} is
\begin{align}
\tilde{t}_a(v, F, G_n^{(k)}) &:=
\frac{
\left|
\left\{ \varphi : V(F) \to V(G_n^{(k)}) \
\middle| \
\varphi(a) = v, \ \varphi \ \text{injective hom.}
\right\}
\right|
}{n^{|V(F)| - 1}},
\label{eq:empirical_1point} \\
& =\frac{1}{n^{|V(F)|-1}}
\sum_{\mathbf{s}_{\{a\}^c}} 
\left( \prod_{j \in N_F(a)} A^{(k)}_{v, s_j} \right)
\left( \prod_{(i,j) \in E(F \setminus \{a\})} A^{(k)}_{s_i, s_j} \right),
\label{eq:empirical_1point_sum}
\end{align}
where $\mathbf{s}_{\{a\}^c}$ is the set of all $|V(F)|-1$-tuples ${(s_i)}_{i \in V(F) \setminus \{a\}} \in ([n] \setminus v)^{|V(F)|-1}$  with distinct indices and $N_F(a)$ is the set of neighbors of $a$ in $F$.
In other words, $\tilde{t}_a(v, F, G_n^{(k)})$ is the fraction of injective homomorphisms from $F$ to $G_n^{(k)}$ mapping $a$ to $v$.

Furthermore, the following average
\begin{equation}
\widehat{t}(v,F,G_n^{(k)}) :=
\frac{1}{|\mathrm{Aut}(F)| }
\sum_{a=1}^{|V(F)|} \tilde{t}_a(v, F, G_n^{(k)}),
\label{eq:avg_t}
\end{equation}
is an empirical estimator of
\(
\frac{1}{|\mathrm{Aut}(F)|} \displaystyle \sum_{a=1}^{|V(F)|} t_a(v,F,W_{(k)}).
\)

Finally, recalling that $
t(F,W_{(k)}) = \mathbb{E}\left(t_a(v,F,W_{(k)}) \right),
$, 
we observe that 
\begin{equation} \label{eq:empirical_bar_t}
\bar{t}(F,G_n^{(k)}) := \frac{1}{n} \sum_{v=1}^{n} \widehat{t}(v,F,G_n^{(k)}),
\end{equation}
is an empirical estimator of
$
\frac{|V(F)|}{|\mathrm{Aut}(F)|} \, t(F,W_{(k)}).
$
\end{definition}

We now provide empirical counterparts for the $2-$point conditional motif kernel ${W_{(k)}}_F(x,y)$.

\begin{definition}[Empirical 2-Point Conditional Motif Kernel]
\label{def:2pt}
For $u,v \in V(G_n^{(k)})$, we define the empirical analogue of \eqref{eq:2point_kernel} by

\begin{equation}
\widehat{W}_{F}((u,v), G_n^{(k)}) :=
\frac{1}{2|\mathrm{Aut}(F)|}
\sum_{1 \leq a \neq b \leq |V(F)|}
\tilde{t}_{a,b}((u,v), F, G_n^{(k)}),
\label{eq:W_hat}
\end{equation}
where 
\begin{align}
\tilde{t}_{a,b}((u,v), F, G_n^{(k)}) &:=
\frac{
\left| \left\{ \varphi : V(F) \to V(G_n^{(k)}) \
\middle| \
\varphi(a) = u, \ \varphi(b) = v, \ \varphi \ \text{injective hom.}
\right\} \right|
}{n^{|V(F)| - 2}} 
\label{eq:empirical_2point}\\
&=
\frac{1}{n^{|V(F)|-2}}
A^{(k)+}_{u,v}
\sum_{\mathbf{s}_{\{a,b\}^c}} 
\left( \prod_{j \in N_F(a) \setminus \{b\}} A^{(k)}_{u, s_j} \right)
\left( \prod_{j \in N_F(b) \setminus \{a\}} A^{(k)}_{v, s_j} \right)
\notag \\
&\quad \times
\left( \prod_{(i,j) \in E(F \setminus \{a,b\})} A^{(k)}_{s_i, s_j} \right),
\label{eq:empirical_2point_sum}
\end{align}
and  
\begin{equation}
A^{(k)+}_{u,v} :=
\begin{cases}
A^{(k)}_{u,v}, & \text{if } (u,v) \in E(F), \\
1, & \text{otherwise}.
\end{cases}
\label{eq:A_plus}
\end{equation}
\end{definition}

\begin{definition}[2-Point empirical homomorphism density]
With $\widehat{W}_{F}((u,v), G_n^{(k)})$ as in \eqref{eq:W_hat}, we define 
\begin{equation}
\overline{W}_F(G_n^{(k)}) :=
\frac{1}{n^2} \sum_{u=1}^n \sum_{v=1}^n
\widehat{W}_{F}((u,v), G_n^{(k)})
\label{eq:W_bar}
\end{equation}
which is an empirical estimator of 
\(
\frac{|V(F)|(|V(F)|-1)}{2|\mathrm{Aut}(F)|} t(F,W_{(k)}).
\)
\end{definition}

\medskip
We are now ready to define the sampling statistic along with the convergence result. 

\begin{proposition}[\textbf{Convergence of the Sampling Statistic}]
\label{thm:multiplier-bootstrap}
Let $Z_1, \ldots, Z_n$ be i.i.d. standard normal random variables, i.e., $Z_i \sim \mathcal{N}(0,1)$ for all $1 \leq i \leq n$. Let $\Lambda'_d \subseteq \Lambda_d$ denote the set of indices $k \in \Lambda_d$ for which we have $F$-irregularity. 

Define the sampling statistics $\widehat{Z}_F(k)$ as follows:
\begin{align}
\widehat{Z}_F(k) := 
\begin{cases}
\dfrac{1}{\sqrt{n}} \sum\limits_{v=1}^n \left(\widehat{t}(v,F,G_n^{(k)}) - \bar{t}(F,G_n^{(k)})\right) Z_v & \text{if } k \in \Lambda'_d, \\[1.5ex]
\dfrac{1}{n} \sum\limits_{u,v=1}^n \left(\widehat{W}_F((u,v), G_n^{(k)}) - \overline{W}_F(G_n^{(k)})\right)(Z_u Z_v - \delta_{u,v}) & \text{if } k \in \Lambda_d \setminus \Lambda'_d,
\end{cases}
\label{eq:bootstrap-stat}
\end{align}

Suppose $G_n^{(k)}$ is a realization from the random graph model $G(n, W_{(k)})$ for each $k \in \Lambda_d$ and that the sequence $(G_n^{(k)})_{k \in \Lambda_d}$ converges in the joint-cut metric. Then, conditionally on the observed graphs $\mathbf{G}_n := \{G_n^{(k)}\}_{k \in \Lambda_d}$, we have the following convergence in distribution:
\begin{align}
\widehat{\mathbf{Z}}_F := \left( \widehat{Z}_F(k) \right)_{k \in \Lambda_d} \,\Big|\, \mathbf{G}_n \xrightarrow{D} \mathbf{Z} := \left( Z_k \right)_{k \in \Lambda_d},
\label{eq:bootstrap-convergence}
\end{align}
where $Z_k$ denotes the asymptotic distribution of the cross-layer motif counts as given in Theorem~\ref{thm:main2}.
\end{proposition}

\begin{proof}[Proof of Proposition \ref{thm:multiplier-bootstrap}]
The proof builds upon the approach of Theorem 4.1 in \cite{Bhattacharya2024}, extending it to the setting of cross-layer motif counts. Although the core reasoning and main steps remain similar, the extension to the multiplex case introduces additional subtleties. For clarity and completeness, we therefore reproduce and adapt some of these steps for the reader.

Let $k \in \Lambda_d$. We begin by expressing the bootstrapped statistic $\widehat{Z}_F(k)$ in terms of stochastic integrals.

Let $I_s = \left[\frac{s-1}{n}, \frac{s}{n}\right[$ for $1 \leq s \leq n$, so that the collection $(I_s)_{s=1}^n$ forms a partition of $[0,1]$ into intervals of equal length $1/n$. Let $(B_t)_{t \in [0,1]}$ be a standard Brownian motion independent of the observed graphs $G_n^{(k)}$. Define
\begin{equation}
\label{eq:eta}
\eta_s := \int_{I_s} dB_s.
\end{equation}

Then the collection $\{\eta_1, \ldots, \eta_n\}$ is an i.i.d. sample from the normal distribution $\mathcal{N}(0,1/n)$.

We now define the following alternative representation of $\widehat{Z}_F(\text{Col}(k))$ using these Gaussian increments:
\begin{align}
\widehat{Z}'_F(k) := 
\begin{cases}
\displaystyle \sum\limits_{v=1}^n \left( \widehat{t}(v,F,G_n^{(k)}) - \bar{t}(F,G_n^{(k)}) \right) \eta_v & \text{if } k \in \Lambda'_d, \\[1.5ex]
\displaystyle \sum\limits_{u,v=1}^n \left( \widehat{W}_{F}((u,v), G_n^{(k)}) - \overline{W}_F(G_n^{(k)}) \right)\left( \eta_u\eta_v - \frac{\delta_{u,v}}{n} \right) & \text{otherwise},
\end{cases}
\label{eq:Zprime-def}
\end{align}
where $\widehat{t}(v,F,G_n^{(k)})$ and $\widehat{W}_{F}((u,v), G_n^{(k)})$ are as defined in equations \eqref{eq:avg_t} and \eqref{eq:W_hat}, respectively. 

It is clear that, by distributional equivalence of $\eta_s$ and $Z_s / \sqrt{n}$, we have
\begin{equation}
\widehat{Z}'_F(k) \overset{\text{Law}}{=} \widehat{Z}_F(k) \quad \text{for all } k \in \Lambda_d.
\end{equation}

Next, for $x \in [0,1]$, define the piecewise extension of the 1-point empirical homomorphism density:
\begin{equation}
\widehat{t}(x,F,G_n^{(k)}) := \widehat{t}(\lceil nx \rceil, F, G_n^{(k)}).
\end{equation}
Then, for every $v \in V(G_n^{(k)}),$ we have 
\begin{align}
\sum_{v=1}^n \int_{I_v} \widehat{t}(x,F,G_n^{(k)}) \, dB_v = \int_0^1 \widehat{t}(x,F,G_n^{(k)}) \, dB_x  \quad \text{and} \quad  \int_0^1 \widehat{t}(x,F,G_n^{(k)}) \, dx = \bar{t}(F,G_n^{(k)}).
\end{align}

Similarly, for $(x,y) \in [0,1]^2$, define
\begin{equation}
\widehat{W}_F((x,y), G_n^{(k)}) := \widehat{W}_F((\lceil nx \rceil, \lceil ny \rceil), G_n^{(k)}),
\end{equation}
so that, for every $u,v \in V(G_n^{(k)})$, we have 
\begin{align}
\sum_{u,v=1}^n \int_{I_u} \int_{I_v} \widehat{W}_F((x,y), G_n^{(k)}) \, dB_u \, dB_v = \int_{[0,1]^2} \widehat{W}_F((x,y),G_n^{(k)}) \, dB_x \, dB_y  
\end{align}
and 
\begin{align}
\int_{[0,1]^2} \widehat{W}_F((x,y), G_n^{(k)}) \, dx \, dy= \overline{W}_F(G_n^{(k)}).   
\end{align}

Using the Wick formula for stochastic integrals, we have
\begin{equation}
\eta_u \eta_v - \frac{\delta_{u,v}}{n} = \int_{I_u} \int_{I_v} dB_u \, dB_v.
\end{equation}
Substituting into \eqref{eq:Zprime-def}, we obtain continuous representations of the bootstrapped statistics :

\begin{itemize}
    \item If $k \in \Lambda'_d$, then
    \begin{equation}
    \widehat{Z}'_F(k) = \int_0^1 \left( \widehat{t}(x,F,G_n^{(k)}) - \int_0^1 \widehat{t}(x,F,G_n^{(k)}) \, dx \right) dB_x.
    \end{equation}
    
    \item If $k \notin \Lambda'_d$, then
    \begin{equation}
    \widehat{Z}'_F(k) = \int_{[0,1]^2} \left( \widehat{W}_F((x,y), G_n^{(k)}) - \int_{[0,1]^2} \widehat{W}_F((x,y), G_n^{(k)}) \, dx \, dy \right) dB_x \, dB_y.
    \end{equation}
\end{itemize}

For $k \in \Lambda_d$, denote by $W^{G_n^{(k)}}$ the empirical graphon associated to graph $G_n^{(i)}$ (see Definition \ref{def:empiricalgraphon}). We state the following convergence result.

\begin{lemma}
\label{lem:bootstrap-convergence}
For all $k \in \Lambda_d$, as $n \to \infty$,
\begin{equation}
\mathbb{E}\left[ \left| \widehat{Z}'_F(k) - Y_F(k) \right|^2 \,\Big|\, \mathbf{G}_n \right] \xrightarrow{\text{a.s.}} 0,
\end{equation}
where the limit random variable $Y_F(k)$ is given by:
\begin{itemize}
    \item If $k \in \Lambda'_d$,
    \begin{equation}
    Y_F(k) := \int_0^1 \left\{ \frac{1}{|\mathrm{Aut}(F)|} \sum_{a=1}^{|V(F)|} t_a(x,F,W^{G_n^{(k)}}) - \frac{|V(F)|}{|\mathrm{Aut}(F)|} t(F,W^{G_n^{(k)}}) \right\} dB_x.
    \end{equation}
    
    \item If $k \notin \Lambda'_d$,
    \begin{equation}
    Y_F(k) := \int_0^1 \int_0^1 \left\{ W_F^{G_n^{(k)}}(x,y) - \frac{|V(F)|(|V(F)|-1)}{2|\mathrm{Aut}(F)|} t(F,W^{G_n^{(k)}}) \right\} dB_x \, dB_y.
    \end{equation}
\end{itemize}
\end{lemma}

\begin{proof}[Proof of Lemma \ref{lem:bootstrap-convergence}]

We start the proof by showing the convergence for $k \in \Lambda'_d$, that is, when we have $W_{(k)}$-$F$-irregularity.

We first bound the conditional mean squared error:
\begin{align}
\mathbb{E}\left[ \left| \widehat{Z'}_F(k) - Y_F(k) \right|^2 \mid \mathbf{G_n}\right] 
&\leq 2 \underbrace{\mathbb{E}\left[  I_1^2 \left( \widehat{t}(x,F,G_n^{(k)}) - \frac{1}{|\mathrm{Aut}(F)|} \sum_{a=1}^{|V(F)|} t_a(x,F,W^{G_n^{(k)}})\right) \Big| \mathbf{G_n}\right]}_{=(A)}  \label{eq:A}\\
&\quad + 2 \underbrace{\mathbb{E}\left[ I_1^2\left( \int_0^1 \widehat{t}(x,F,G_n^{(k)}) \,dx - \frac{|V(F)|}{|\mathrm{Aut}(F)|}t(F,W^{G_n^{(k)}}) \right)  \Big| \mathbf{G_n}\right]}_{=(B)} \label{eq:B}
\end{align}

We start with term $(A)$:
\begin{align}
(A) &\leq \int_0^1 \left(\widehat{t}(x,F,G_n^{(k)}) - \frac{1}{|\mathrm{Aut}(F)|} \sum_{a=1}^{|V(F)|} t_a(x,F,W^{G_n^{(k)}})\right)^2 \,dx \\
&= \sum_{v=1}^n \int_{I_v} \left(\widehat{t}(x,F,G_n^{(k)}) - \frac{1}{|\mathrm{Aut}(F)|} \sum_{a=1}^{|V(F)|} t_a(x,F,W^{G_n^{(k)}})\right)^2 \,dx \label{equ:9.1}
\end{align}
where the first inequality comes from the boundedness property of stochastic integrals, that is $\mathbb{E}\left[I_d(f)^2\right] \leq d! \cdot ||f||^2$, and the second equality follows from the definition of $\widehat{t}(\cdot, F, G_n^{(k)})$ given in \eqref{eq:avg_t}.

For $1 \leq v \leq n$, recall:
\begin{align}
\widehat{t}(v,F, G_n^{(k)}) &= \frac{1}{|\mathrm{Aut}(F)|} \sum_{a=1}^{|V(F)|} \tilde{t}_a(v,F,G_n^{(k)}) \nonumber \\
&= \frac{1}{|\mathrm{Aut}(F)| n^{|V(F)|-1}} \sum_{a=1}^{|V(F)|} X_a(v,F,G_n^{(k)}),
\end{align}
where for $k \in \Lambda_d$
\begin{equation}
X_a(v,F,G_n^{(k)}) = \sum_{\mathbf{s}_{\{a\}^c}} 
\left( \prod_{j \in N_F(a)} A^{(k)}_{v, s_j} \right)
\left( \prod_{(k,j) \in E(F \setminus \{a\})} A^{(k)}_{s_k, s_j} \right).
\end{equation}

Then from \eqref{equ:9.1}, we have:
\begin{align}
(A) \leq \frac{1}{|\mathrm{Aut}(F)|^2} \sum_{v=1}^n \sum_{a=1}^{|V(F)|} \int_{I_v} \left( \frac{X_a(v,F,G_n^{(k)})}{n^{|V(F)|-1}} - t_a(x,F,W^{G_n^{(k)}}) \right)^2 \,dx  \label{equ:9.2}
\end{align}

Using the definition of the 1-point conditional homomorphism density for $i \in \Lambda_d$, we have
\begin{align}
t_a(x,F,W^{G_n^{(i)}}) = \int_{[0,1]^{|V(F)|-1}} &\prod_{u \in N_F(a)} W^{G_n^{(i)}}(x,x_u) \nonumber \\
& \times \prod_{(u,v) \in E(F \setminus \{a\})} W^{G_n^{(i)}} (x_u,x_v) \prod_{v \in V(F \setminus \{a\})} dx_v.
\end{align}

From the empirical graphon definition, $t_a(x,F,W^{G_n^{(i)}})$ can be equivalently written as follows : 
\begin{align}
t_a(x,F,W^{G_n^{(i)}}) &= \frac{1}{n^{|V(F)|-1}} \sum_{\tilde{\mathbf{s}}_{\{a\}^c}} 
\left( \prod_{j \in N_F(a)} A^{(i)}_{v, s_j} \right)
\left( \prod_{(k,j) \in E(F \setminus \{a\})} A^{(i)}_{s_k, s_j} \right) \label{eq:1-pointcondnoninj}
\end{align}
for all $x \in I_v$ and $1 \leq v \leq n$ and where $\tilde{\mathbf{{s}}}_{\{a\}^c}$ is the same as $\mathbf{s}_{\{a\}^c}$ but not necessarily with distinct indices. 
This representation is due to the fact that 
\begin{align}
t_a(x,F,W^{G_n^{(i)}}) = t_a(x,F,G_n^{(i)})= \frac{\# \{\varphi \in \text{hom}(F,G_n^{(i)}) \mid \varphi(a) = \lceil nx \rceil =v\}}{n^{|V(F)|-1}}.
\end{align}

Now, since the cardinality of $\tilde{\mathbf{{s}}}_{\{a\}^c}$ is $n^{|V(F)| - 1}$ (general homomorphisms), the cardinality of $\mathbf{s}_{\{a\}^c}$ is $\frac{n!}{(n - |V(F)| + 1)!}$ (injective homomorphisms), we have:
\begin{align}
\left| \frac{X_a(v,F,G_n^{(k)})}{n^{|V(F)|-1}} - t_a(x,F,W^{G_n^{(k)}}) \right| \leq \frac{1}{n},    
\end{align}
which implies from \eqref{equ:9.2}:
\begin{equation}
(A) \leq \frac{|V(F)|}{|\mathrm{Aut}(F)|^2} \cdot \frac{1}{n^2} \quad \text{a.s.} \label{eq:proofA}
\end{equation}

For term $(B)$ \eqref{eq:B}, 
\begin{align}
\left| \int_0^1 \widehat{t}(x,F,G_n^{(k)}) \,dx - \frac{|V(F)|}{|\mathrm{Aut}(F)|}t(F,W^{G_n^{(k)}}) \right| \leq \frac{|V(F)|}{|\mathrm{Aut}(F)|n} \label{equ:9.3}
\end{align}

Indeed, on one hand:
\begin{align}
\int_0^1 \widehat{t}(x,F,G_n^{(k)}) \,dx &= \sum_{v=1}^n \int_{I_v} \widehat{t}(x,F,G_n^{(k)}) \,dx \nonumber \\
&= \frac{1}{|\mathrm{Aut}(F)|} \sum_{v=1}^n \int_{I_v} \sum_{a=1}^{|V(F)|} \underbrace{\tilde{t}_a(x,F,G_n^{(k)})}_{= \frac{\# \left\{ \varphi : V(F) \to V(G_n^{(k)}) \ \middle| \ \varphi(a) = \lceil nx \rceil = v\text{ , } \varphi \text{ inj hom} \right\}}{n^{|V(F)| - 1}} } \,dx 
\end{align}

Also, on the other hand,
\begin{align}
\frac{|V(F)|}{|\mathrm{Aut}(F)|}t(F,W^{G_n^{(k)}}) &= \frac{1}{|\mathrm{Aut}(F)|} \sum_{a=1}^{|V(F)|} \int_0^1 t_a(x,F,W^{G_n^{(k)}}) \,dx \nonumber \\
&= \frac{1}{|\mathrm{Aut}(F)|} \sum_{v=1}^n \int_{I_v} \sum_{a=1}^{|V(F)|}  \underbrace{t_a(x,F,W^{G_n^{(k)}})}_{=\frac{\# \left\{ \varphi : V(F) \to V(G_n^{(k)}) \ \middle| \ \varphi(a) = \lceil nx \rceil \text{ and } \varphi \text{ hom} \right\}}{n^{|V(F)| - 1}} } \,dx,
\end{align}
since $ t(F,W^{G_n^{(k)}}) = \displaystyle \int_0^1 t_a(x,F,W^{G_n^{(k)}}) \,dx,$ for all $1 \leq a \leq |V(F)|$.

Thus, using \eqref{equ:9.3} and the boundedness property of stochastic integral, we get :
\begin{align}(B) \leq \frac{|V(F)|^2}{|\mathrm{Aut}(F)|^2 n^2} \quad \text{a.s.}  \label{eq:proofB}
\end{align}

Combining \eqref{eq:proofA} and \eqref{eq:proofB}, we conclude the case $k \in \Lambda'_d$.

Now consider $k \in \Lambda_d \setminus \Lambda'_d$ (i.e., when we have $W_{(k)}$-$F$-regularity).
For all $k \in \Lambda_d$, note that :
\begin{align}t(F,W^{G_n^{(k)}}) = \int_0^1 \int_0^1 W_F^{G_n^{(k)}}(x,y) \,dx \,dy.
\end{align}

Now, observe that 
\begin{align}
&\mathbb{E}\left[ \left| \widehat{Z'}_F(k) - Y_F(k) \right|^2 \mid \mathbf{G_n}\right] 
\leq 2 \underbrace{\mathbb{E} \left [ I_2^2 \left(\widehat{W}_F((x,y), G_n^{(k)}) - W_F^{G_n^{(k)}}(x,y)\right)\right]}_{=(C)} \label{eq:C}\\
&\quad + 2 \underbrace{\mathbb{E} \left [ I_2^2 \left( \int_{[0,1]^2}\widehat{W}_F((x,y), G_n^{(k)}) \,dx\,dy - \frac{|V(F)|(|V(F)|-1)}{2|\mathrm{Aut}(F)|}t(F,W^{G_n^{(k)}})\right)\right]}_{=(D)} \label{eq:D}
\end{align}

Again using the same reasoning as for terms $(A)$ \eqref{eq:A} and $(B)$ \eqref{eq:B} in the case $k \in \Lambda'_d$, we find:
\begin{align}
(C) &\leq \int_0^1 \int_0^1 \left(\widehat{W}_F((x,y), G_n^{(k)}) - W_F^{G_n^{(k)}}(x,y) \right)^2 \,dx\,dy \nonumber\\
&\leq \frac{|V(F)|^2}{4|\mathrm{Aut}(F)|^2} \cdot \frac{1}{n^2} \quad \text{a.s.} \label{eq:proofC}
\end{align}

We also find the same upper bound for $(D)$. This concludes the proof of Lemma \ref{lem:bootstrap-convergence}.

\end{proof}

Since 
\[
\widehat{\mathbf{Z}'}_F := (\widehat{Z'}_F(k))_{k \in \Lambda_d}
\]
has the same distribution as 
\[
\mathbf{Z}_F := (\widehat{Z}_F(k)_{k \in \Lambda_d},
\]
Lemma \ref{lem:bootstrap-convergence} implies that, to prove Theorem \ref{thm:multiplier-bootstrap}, it suffices to show the conditional convergence in distribution
\begin{equation}
\label{eq:conv}
\mathbf{Y}_F \mid \mathbf{G}_n = (Y_F(k))_{k \in \Lambda_d} \mid \mathbf{G}_n
\quad \xrightarrow{d} \quad
\mathbf{Z} = (Z(k))_{k \in \Lambda_d}.
\end{equation}

We now state one lemma and one proposition that establish this convergence and thereby complete the proof of Theorem \ref{thm:multiplier-bootstrap}. Their proofs are omitted here, as they can be readily adapted from \cite{Bhattacharya2024}. 

In particular, our goal is to show that for any \(\alpha \in \mathbb{R}^{2^d - 1}\), the moment generating function (MGF) of \(\alpha^T \mathbf{Y}_F\) conditioned on the graphs \(\left(G_n^{(k)}\right)_{k \in \Lambda_d}\) converges almost surely to the MGF of \(\alpha^T \mathbf{Z}\).

Before stating the lemma and proposition, we introduce the following definitions.

\medskip

Let $U: [0,1]^2 \to \mathbb{R}$ be any symmetric function. For $L \geq 2$, define the $L$-th \emph{path composition} of $U$ as
\begin{equation} \label{eq:path-composition}
U^{(L)}(x, y) := \int_{[0,1]^{L-1}} U(x, z_1) U(z_1, z_2) \cdots U(z_{L-1}, y) \, dz_1 \cdots dz_{L-1}.
\end{equation}

Now let $\alpha = (\alpha_k)_{k \in \Lambda_d}^T \in \mathbb{R}^{2^d - 1}$. Define the function $V_{\alpha} : [0,1] \to \mathbb{R}$ by
\begin{equation} \label{eq:valpha}
\begin{aligned}
V_{\alpha}(x) := \sum_{k \in \Lambda'_d} \alpha_k \bigg[ 
&\frac{1}{|\mathrm{Aut}(F)|} \sum_{a=1}^{|V(F)|} t_a(x, F, W_{(k)}) - \frac{|V(F)|}{|\mathrm{Aut}(F)|} \cdot t(F, W_{(k)})
\bigg],
\end{aligned}
\end{equation}
and define $U_{\alpha} : [0,1]^2 \to \mathbb{R}$ by
\begin{equation} \label{eq:ualpha}
\begin{aligned}
U_{\alpha}(x, y) := \sum_{k \in \Lambda_d \setminus \Lambda'_d} \alpha_k \bigg[
& W^{(k)}_F(x, y) - \frac{|V(F)|(|V(F)| - 1)}{2|\mathrm{Aut}(F)|} t(F, W_{(k)})
\bigg].
\end{aligned}
\end{equation}

Define also the variances:
\begin{equation} \label{eq:eta-alpha}
\eta_{\alpha} := \| V_{\alpha} \|_2^2,
\end{equation}
and
\begin{equation} \label{eq:eta-alpha-tilde}
\tilde{\eta}_{\alpha} := 2 \| U_{\alpha} \|_2^2 + \alpha_{+}^T \Sigma \alpha_{+},
\end{equation}
where $\alpha_{+} := (\alpha_k)_{k \in \Lambda_d \setminus \Lambda'_d}$ and $\Sigma$ is the covariance matrix defined in equation~\eqref{equ:5}.

We then state the following lemma and proposition, which together establish the convergence in \eqref{eq:conv}.

\begin{lemma}[MGF of the cross-layer motif counts] \label{lem:mgf-colored}
Fix $\alpha = (\alpha_k)_{k \in \Lambda_d}^T \in \mathbb{R}^{2^d - 1}$ and define
\[
\mathcal{C} := \sum_{k \in \Lambda_d \setminus \Lambda'_d} |\alpha_k| \cdot \frac{|V(F)|(|V(F)| - 1)}{|\mathrm{Aut}(F)|}.
\]
Then for all $\theta$ such that $|\theta| < \frac{1}{32 \mathcal{C}}$, the following holds:
\begin{align}
\log \mathbb{E} \left[e^{\theta \alpha^T \mathbf{Z}}\right]
= (\eta_{\alpha} + \tilde{\eta}_{\alpha}) \frac{\theta^2}{2}
&+ \sum_{L=1}^{\infty} 2^{L-1} \theta^{L+2} \int_{[0,1]^2} V_{\alpha}(x) V_{\alpha}(y) U_{\alpha}^{(L)}(x, y) \, dx\,dy + \frac{1}{2} \sum_{L=3}^{\infty} \frac{(2\theta)^L}{L} \int_0^1 U_{\alpha}^{(L)}(x, x) \, dx, \label{eq:mgf-Z}
\end{align}
where $U_{\alpha}^{(L)}$ is the $L$-th path composition defined as in~\eqref{eq:path-composition}.
\end{lemma}

\begin{proposition} \label{prop:mgf-convergence}
For any $\alpha \in \mathbb{R}^{2^d - 1}$ and $|\theta| < \frac{1}{32 \mathcal{C}}$, we have
\[
\lim_{n \to \infty} \log \mathbb{E} \left[e^{\theta \alpha^T \mathbf{Y}_F} \, \middle| \, \mathbf{G}_n \right]
= \log \mathbb{E} \left[e^{\theta \alpha^T \mathbf{Z}}\right],
\]
on a set $\mathcal{A}$ with probability one.
\end{proposition}

\end{proof}

\section{Proof of Proposition \ref{prop:SIGMA}}
\label{append:F}

\begin{proof}
By Itô isometry applied to the stochastic integral defining $Z_k$ in \eqref{equ:6}, we have
\begin{align}
\mathrm{Var}(Z_k) &= \int_0^1 \frac{1}{|\mathrm{Aut}(F)|^2} \mathbb{E}\left[ T_k(x)^2 \right] dx \notag \\
&= \frac{1}{|\mathrm{Aut}(F)|^2}  \int_0^1 T_k(x)^2 \ \, dx,
\label{eq:var_ito}
\end{align}
where for each $k \in \Lambda_d$ and $x \in [0,1]$, we denote
\begin{align}
T_k(x) := \sum_{a=1}^{|V(F)|} t_a(x,F,W_{(k)}) - |V(F)| \, t(F,W_{(k)}).
\end{align}

Next, for all $k\in \Lambda_d$, we expand the product $T_k(x)^2$ as
\begin{align}
T_k(x)^2 = &\sum_{1 \leq a,b \leq |V(F)|} t_a(x,F,W_{(k)}) \cdot t_b(x,F,W_{(k)})
- |V(F)| \, t(F,W_{(k)}) \sum_{a=1}^{|V(F)|} t_a(x,F,W_{(k)}) \notag \\
& - |V(F)| \, t(F,W_{(k)}) \sum_{a=1}^{|V(F)|} t_a(x,F,W_{(k)})+ |V(F)|^2 \, t(F,W_{(k)})^2.
\label{eq:TiTj_expand}
\end{align}

Integrating both sides of \eqref{eq:TiTj_expand} over the interval $[0,1]$ and using the identities
\[
\int_0^1 t_a(x,F,W_{(k)}) \, dx = t(F,W_{(k)}), \quad t_a(x,F,W_{(k)}) \cdot t_b(x,F,W_{(k)}) = t\bigl(F \bigoplus_{a,b} F, W_{(k)}\bigr),
\]
we obtain exactly the expression stated in \eqref{equ:vargen}, which completes the proof.

\end{proof}

\section{Testing $W_{(k)}-F-$regularity}
\label{append:G}

Having defined the empirical variance in \eqref{eq:empirical-varicance-gen}, we state the following result, which extends Proposition~5.1 in \cite{Bhattacharya2024}. 

\begin{proposition}
\label{prop:G1}
Fix $k \in \Lambda_d$. Suppose $\sigma^2(k, \mathbf{G}_n)$ is defined as in \eqref{eq:empirical-varicance-gen}. Then the following hold:
\begin{enumerate}
    \item When $W_{(k)}$ is $F$-regular, we have
    \[
    \sqrt{n} \, \sigma^2(k, \mathbf{G}_n) \xrightarrow{P} 0.
    \]
    \item When $W_{(k)}$ is $F$-irregular, we have
    \[
    \sqrt{n} \, \sigma^2(k, \mathbf{G}_n) \xrightarrow{P} \infty.
    \]
\end{enumerate}
\end{proposition}

\medskip
\noindent
Following the approach of \cite{Bhattacharya2024}, we define the test statistic
\begin{equation}
    \phi(k, F, \mathbf{G}_n) := \mathbf{1}\big\{\sqrt{n}\,\sigma^2(k, \mathbf{G}_n) > 1 \big\},
    \label{eq:test-statistic}
\end{equation}
for each $k \in \Lambda_d$.
Proposition \ref{prop:G1} implies that, under $H_0$ (i.e., when the $F-$regularity for $W_{(k)}$ is satisfied), we have 
\begin{equation}
    \mathbb{P}\big( \phi(k, F, \mathbf{G}_n) = 1 \big) \to 0,
\end{equation}
while under $H_1$ this probability converges to $1$. Hence, the test statistic~\eqref{eq:test-statistic} is consistent for detecting $W_{(k)}$-$F$-irregularity. This, in turn, enables us to define the set 
\begin{equation}
    S(F,\mathbf{G}_n) := \big\{ k \in \Lambda_d: \sqrt{n}\,\sigma^2(k, \mathbf{G}_n) > 1 \big\}
    \label{eq:rejection-set-gen}
\end{equation}
of indices for which the hypothesis of $W_{(k)}-F$
-regularity is rejected.

\begin{proof}[Proof of Proposition \ref{prop:G1}]
For all $k \in \Lambda_d$, define
\begin{align}
    \sigma^2(k, \mathbf{W}^{\mathbf{G}_n}) 
    &:= \frac{1}{|\mathrm{Aut}(F)|^2}  \times \left[ \sum_{1 \leq a,b \leq |V(F)|} 
    t\!\left(F \bigoplus_{a,b} F, W^{G_n^{(k)}}\right)
    - |V(F)|^2 \, t\!\left(F,W^{G_n^{(i)}}\right)^2  \right],
    \label{eq:variance-W}
\end{align}
where, for $k \in \Lambda_d$,
\begin{equation}
    t(F,W^{G_n^{(k)}}) 
    := \frac{1}{n^{|V(F)|}} \sum_{s \in [n]^{|V(F)|}} 
    \prod_{(a,b) \in E(F)} A^{(k)}_{s_a s_b}.
    \label{eq:t-F-W}
\end{equation}

\vspace{0.5em}

We proceed as in the proof of Theorem~\ref{thm:multiplier-bootstrap}, showing that
\begin{equation}
    \left| \sigma^2(k, \mathbf{G}_n) - \sigma^2(k, \mathbf{W}^{\mathbf{G}_n}) \right| 
    \leq O\!\left(\frac{1}{n}\right) \quad \text{a.s.}
    \label{eq:empirical-vs-W}
\end{equation}

In fact, observe that 
\begin{align}
    \big| \sigma^2(k, \mathbf{G}_n) - \sigma^2(k, \mathbf{W}^{\mathbf{G}_n}) \big| &\leq \frac{1}{|\mathrm{Aut}(F)|^2} \bigg[ \sum_{1 \leq a,b \leq |V(F)|} \left( \frac{1}{n} \sum_{v=1}^n \right.  \nonumber\\
    & \underbrace{\big| \tilde{t}_a(v,F,G_n^{(k)}) \cdot \tilde{t}_b(v,F,G_n^{(k)}) - t_a(v,F,W^{G_n^{(k)}}) \cdot t_b(v,F,W^{G_n^{(k)}}) \big|}_{=(A)} \bigg) \label{eq:A2}\\
    &+ |V(F)|^2 \underbrace{\big| \widehat{t}(F,G_n^{(k)}) ^2 - t(F,W^{G_n^{(k)}})^2}_{=(B)} \bigg]. \label{eq:B2}
\end{align}

Since $n^{|V(F)|} \geq |[n]_{|V(F)|}|,$ (recall that $|[n]_{|V(F)|}| = \binom{n}{|V(F)|} |V(F)|!$), we get 
    \begin{equation}
    (B) \leq \left| \frac{1}{n^{2|V(F)|}} \big[ {S^{(k)}}^2  - {S_{inj}^{(k)}}^2\big] \right|, \label{eq:bound-B}
    \end{equation}
    where for $k \in \Lambda_d,$ we write 
    \begin{align}
    S^{(k)} = \sum_{s \in [n]^{|V(F)|}} \prod_{(a,b) \in E(F)} A^{(k)}_{s_a s_b}, \label{eq:S_i}
    \end{align}
    and
    \begin{align}
    S_{inj}^{(k)} = \sum_{s \in [n]_{|V(F)|}} \prod_{(a,b) \in E(F)} A^{(k)}_{s_a s_b}. \label{eq:S_inj}
    \end{align}

Then,
    \begin{align}
    \left| {S^{(k)}}^2 - {S_{inj}^{(k)}}^2  \right| \leq \left| S^{(k)} - S_{inj}^{(k)} \right| \cdot \left| S^{(k)} \right| + \left| S_{inj}^{(k)} \right| \cdot \left| S^{(k)} - S_{inj}^{(k)} \right|. \label{eq:SS}
    \end{align}

Observe that, for all $k \in \Lambda_d$, we have
\begin{equation}
\label{eq:bound-diff-S}
\big| S^{(k)} - S_{\mathrm{inj}}^{(k)} \big| \leq O\!\left(n^{|V(F)| - 1}\right),
\end{equation}
since the entries of $A^{(i)}$ are bounded by $1$ and, recalling the cardinality of $[n]_{|V(F)|}$,
\begin{align}
\big| S^{(k)} - S_{\mathrm{inj}}^{(k)} \big| 
&\leq n^{|V(F)|} - |[n]_{|V(F)|}| \nonumber \\
&= n^{|V(F)|} - n^{|V(F)|} \left( 1 - \frac{|V(F)|(|V(F)|-1)}{2n} + O(n^{-2}) \right) \label{eq:bound-cardinality} \\
&= O\!\left(n^{|V(F)| - 1}\right). \nonumber 
\end{align}

From \eqref{eq:SS} and \eqref{eq:bound-diff-S}, we get 
    \begin{align}
    \left| {S^{(k)}}^2 - {S_{inj}^{(k)}}^2 \right| &\leq O(n^{|V(F)| - 1}) \cdot n^{|V(F)|} + n^{|V(F)|} \cdot O(n^{|V(F)| - 1}) \nonumber \\
    &= O(n^{2|V(F)| - 1}). \label{eq:final}
    \end{align}

Hence, \eqref{eq:final} shows that $(B) \leq O\left(\frac{1}{n}\right).$

\vspace{5mm}
    
Now, for term (A) as defined in \eqref{eq:A2}, we can write
\begin{equation}
\label{eq:A-term}
(A) = \left| \frac{1}{n^{2|V(F)| - 2}} \left[ S_a^{(k)} S_b^{(k)} - S_{\mathrm{inj}_a}^{(k)} S_{\mathrm{inj}_b}^{(k)} \right] \right|,
\end{equation}
where, for $i \in \Lambda_d$ and $1 \leq a \leq |V(F)|$,
\begin{align}
S_a^{(i)} 
&:= \sum_{\mathbf{s}_{\{a\}^c}} 
\left( \prod_{j \in N_F(a)} A^{(i)}_{v, s_j} \right) 
\left( \prod_{(k,j) \in E(F \setminus \{a\})} A^{(i)}_{s_k s_j} \right), 
\label{eq:Sa-def} \\
S_{\mathrm{inj}_a}^{(i)} 
&:= \sum_{\tilde{\mathbf{s}}_{\{a\}^c}} 
\left( \prod_{j \in N_F(a)} A^{(i)}_{v, s_j} \right) 
\left( \prod_{(k,j) \in E(F \setminus \{a\})} A^{(i)}_{s_k s_j} \right),
\label{eq:Sainj-def}
\end{align}
with the notation $\mathbf{s}_{\{a\}^c}$ (resp. $\tilde{\mathbf{s}}_{\{a\}^c}$) denoting a tuple of vertices without injectivity constraints (resp. with injectivity constraints).

Using the bound from \eqref{eq:bound-diff-S} and noting that both $S_a^{(k)}$ and $S_{\mathrm{inj}_a}^{(k)}$ are sums over configurations involving exactly $|V(F)| - 1$ vertices, we have
\begin{align}
\left| S_a^{(k)} S_b^{(k)} - S_{\mathrm{inj}_a}^{(k)} S_{\mathrm{inj}_b}^{(k)} \right| 
&\leq \underbrace{\left| S_a^{(i)} - S_{\mathrm{inj}_a}^{(k)} \right|}_{\leq O(n^{|V(F)| - 2})} 
\underbrace{\left| S_b^{(k)} \right|}_{\leq n^{|V(F)| - 1}}
+ \left| S_{\mathrm{inj}_a}^{(k)} \right| \cdot 
\underbrace{\left| S_b^{(k)} - S_{\mathrm{inj}_b}^{(k)} \right|}_{\leq O(n^{|V(F)| - 2})} 
\nonumber \\
&= O\!\left(n^{2|V(F)| - 3}\right).
\label{eq:bound-SA-Sinj}
\end{align}
Combining \eqref{eq:A-term} and \eqref{eq:bound-SA-Sinj} yields
\begin{equation}
(A) \leq O\!\left( \frac{1}{n} \right).
\label{eq:A-term-bound}
\end{equation}

Thus, we have  
\begin{equation}
\big| \sigma^2(k, \mathbf{G}_n) - \sigma^2(k, \mathbf{W}^{\mathbf{G}_n}) \big| \leq O\left(\frac{1}{n}\right).
\label{equ:bigO}
\end{equation}

Now, consider the case of $W_{(k)}$-$F$-regularity.  
It suffices to prove that
\begin{equation}
\mathbb{E} \left[ \sigma^2(k, \mathbf{W}^{\mathbf{G}_n}) \right] = O\left(\frac{1}{n}\right),
\label{equ:show}
\end{equation}
since, by \eqref{equ:bigO}, the condition $\sqrt{n}\,\sigma^2(k, \mathbf{G}_n) \xrightarrow{P} 0$ holds if and only if  
$\sqrt{n}\,\sigma^2(k, \mathbf{W}^{\mathbf{G}_n}) \xrightarrow{P} 0$.

\vspace{5mm}
Now observe that  
\begin{align}
    \mathbb{E} \big[ \sigma^2(k, \mathbf{W}^{\mathbf{G}_n}) \big] 
    &= \frac{1}{|\mathrm{Aut}(F)|^2} 
       \Bigg[ \sum_{a,b=1}^{|V(F)|} 
       \mathbb{E} \bigg[ t \!\left( F \bigoplus_{a,b} F, W^{G_n^{(k)}}\right) \bigg] - |V(F)|^2 \, \mathbb{E} \big[ t(F,W^{G_n^{(k)}})^2 \big] \Bigg].
    \label{equ:exp-sigma-W}
\end{align}

Using points~(b) and~(c) of Lemma~2.4 in~\cite{Lovasz2006}, we have for all $k \in \Lambda_d$,
\begin{equation}
\bigg| \mathbb{E} \!\left[ t(F,W^{G_n^{(k)}})^2  \right]  
      - t(F,W_{(k)})^2 \bigg| 
      \leq O\!\left(\frac{1}{n}\right).
\label{equ:bound-product}
\end{equation}

Similarly, for all $k \in \Lambda_d$ and $1 \leq a,b \leq |V(F)|$, one shows
\begin{align}
\bigg| \mathbb{E} \!\left[ t \!\left( F \bigoplus_{a,b} F, W^{G_n^{(k)}} \right) \right] 
      - t \!\left( F \bigoplus_{a,b} F, W_{(k)} \right) \bigg|
      \leq O\!\left( \frac{1}{n} \right).
\label{equ:bound-oplus}
\end{align}

Plugging \eqref{equ:bound-product} and \eqref{equ:bound-oplus} into \eqref{equ:exp-sigma-W}, we obtain
\begin{align}
    \mathbb{E} \big[ \sigma^2(k, \mathbf{W}^{\mathbf{G}_n}) \big]
    &\leq \frac{1}{|\mathrm{Aut}(F)|^2}  
       \Bigg[ \sum_{a,b=1}^{|V(F)|} 
       t \!\left( F \bigoplus_{a,b} F, W_{(k)}\right)  - |V(F)|^2 \, t(F,W_{(k)})^2 \Bigg] 
       + O\!\left( \frac{1}{n} \right) \notag \\
    &= \sigma^2(k) + O\!\left( \frac{1}{n} \right),
\end{align}
where $\sigma^2(k)$ is defined in~\eqref{equ:vargen}.

\vspace{5mm}

Recall that $\sigma^2(k) = 0$ under $W_{(k)}$-$F$-regularity, which establishes \eqref{equ:show}.  
We now turn to the case where $W_{(k)}$-$F$-regularity does not hold.  
By Corollary~10.4 in \cite{Lovasz2012}, the continuous mapping theorem,  
and the assumption that the sequence of graphs $(G_n^{(k)})_{k \in \Lambda_d}$ converges in the joint cut-metric to $(W_{(k)})_{k \in \Lambda_d}$, we have
\begin{equation}
    \sigma^2(k, \mathbf{W}^{\mathbf{G}_n}) \xrightarrow{P} \sigma^2(k).
\end{equation}
Since $\sigma^2(k) > 0$ in the $W_{(k)}$-$F$-irregular case, it follows that
\begin{equation}
    \sqrt{n} \, \sigma^2(k, \mathbf{W}^{\mathbf{G}_n}) \xrightarrow{P} \infty.
\end{equation}
Therefore, by \eqref{equ:bigO}, we conclude that
\begin{equation}
    \sqrt{n} \, \sigma^2(k, \mathbf{G}_n) \xrightarrow{P} \infty.
\end{equation}
This completes the proof of Proposition~\ref{prop:G1}.

\end{proof}

\section{Hypothesis Tests}

\subsection{Proof of Lemma \ref{lem:firsttest}}
\label{append:hypo1}

\begin{proof}
Let $r_i, i \in \{1,2\}$ be the normalization factor as defined in \eqref{eq:scaling_factors} and let $\mathbf{Z}_F$ be as defined in \eqref{eq:Zgen1}.
 
The distance from $\mathbf{Z}_F$ to the hyperplane $\mathcal{H}_0$ is given by 
\begin{equation}
    \label{eq:d1}
    \mathcal{D} = \inf_{t \in \mathcal{H}_0} \|\mathbf{Z}_F\|_2 = \inf_{t \in \mathcal{H}_0} \sqrt{Z_F(G_n^{(1)})^2 + Z_F(G_n^{(2)})^2 + Z_F(G_n^{(1)}, G_n^{(2)})^2}
\end{equation}

Define the function 
\begin{equation}
    \label{equ:function1}
\begin{aligned}
    f(t_1,t_2,t_{12}) &:= \|\mathbf{Z}_F\|_2^2\\
    &= r_1^2(X_F(G_n^{(1)}) - c\cdot t_1)^2 + r_2^2(X_F(G_n^{(2)}) -c \cdot t_2)^2 + r_{12}^2(X_F(G_n^{(1)}, G_n^{(2)}) - c \cdot t_{12})^2,
\end{aligned}
\end{equation}
where $c:= \frac{\binom{n}{|V(F)|}|V(F)|!}{|\mathrm{Aut}(F)|}$.

\smallskip
Since the hyperplane $\mathcal{H}_0$ only involves $t_1$ and $t_2$, we can minimize $f$ with respect to $t_1$ and $t_2$ under the constraint $t_1 - t_2 = 0$. Also, note that $t_{12}$ is free and appears only in the third term, which is minimized by setting $t_{12} = X_F(G_n^{(1)}, G_n^{(2)})/c$. Thus, the minimal distance depends only on $t_1$ and $t_2$:

\begin{equation}
\tilde{f}(t_1,t_2) := r_1^2(X_F(G_n^{(1)}) -c \cdot t_1)^2 + r_2^2(X_F(G_n^{(2)}) - c \cdot t_2)^2, \quad \text{with } t_1 - t_2 = 0.
\end{equation}

Let $t := t_1 = t_2$. Then
\begin{equation}
\label{equ:plug1}
\tilde{f}(t,t) = r_1^2(X_F(G_n^{(1)}) - c \cdot t)^2 + r_2^2(X_F(G_n^{(2)}) - c \cdot t)^2.
\end{equation}

Then, differentiating and setting to zero 
\begin{equation}
\frac{d}{dt} \tilde{f}(t,t) = -2 c \cdot r_1^2 (X_F(G_n^{(1)}) - c \cdot t) - 2 c \cdot r_2^2 (X_F(G_n^{(2)}) - c \cdot t) = 0
\end{equation}
gives
\begin{equation}
r_1^2 (c \cdot t - X_F(G_n^{(1)}) + r_2^2 (c \cdot t - X_F(G_n^{(2)}) = 0 \quad \Rightarrow \quad t = \frac{r_1^2 X_F(G_n^{(1)}) + r_2^2 X_F(G_n^{(1)}}{c (r_1^2 + r_2^2)}.
\end{equation}

Plugging back in \eqref{equ:plug1}, the minimal squared distance is
\begin{equation}
\label{equ:distancesquared1}
\begin{aligned}
\mathcal{D}^2 = \tilde{f}(t,t) &= r_1^2 \left(X_F(G_n^{(1)}) - \frac{r_1^2 X_F(G_n^{(1)}) + r_2^2 X_F(G_n^{(2)})}{r_1^2 + r_2^2}\right)^2  \\& + r_2^2 \left(X_F(G_n^{(2)}) - \frac{r_1^2 X_F(G_n^{(1)}) + r_2^2 X_F(G_n^{(2)})}{r_1^2 + r_2^2}\right)^2 \\
&=  \frac{r_1^2 r_2^2 (r_1^2 + r_2^2) (X_F(G_n^{(1)}) - X_F(G_n^{(2)}))^2}{(r_1^2 + r_2^2)^2} \\
&= \frac{r_1^2 r_2^2 (X_F(G_n^{(1)}) - X_F(G_n^{(2)}))^2}{r_1^2 + r_2^2}.
\end{aligned}
\end{equation}

Taking the square root in \eqref{equ:distancesquared1} gives the desired result:
\begin{equation}
\mathcal{D} = \frac{r_1 r_2 |X_F(G_n^{(1)}) - X_F(G_n^{(2)})|}{\sqrt{r_1^2 + r_2^2}}.
\end{equation}
\end{proof}

\subsection{Proof of Lemma \ref{lem:fourthtest}}
\label{append:hypo2}

\begin{proof}
Recall that, for each $(a,b)\in[K]^2$, the normalized vector is
\[
\mathbf{Z}_F^{ab}
=
\left(
\frac{X_F(\mathcal{G}_{ab}^{(1)}) - c_{ab}(\theta_{ab}^{(1)})^{|E(F)|}}{n_{ab}^{|V(F)|-1}},
\frac{X_F(\mathcal{G}_{ab}^{(2)}) - c_{ab}(\theta_{ab}^{(2)})^{|E(F)|}}{n_{ab}^{|V(F)|-1}},
\frac{X_F(\mathcal{G}_{ab}^{(1)},\mathcal{G}_{ab}^{(2)}) - c_{ab}(\theta_{ab}^{(12)})^{|E(F)|}}{n_{ab}^{|V(F)|-1}}
\right),
\]
where
\[
c_{ab}
=
\frac{\binom{n_{ab}}{|V(F)|}|V(F)|!}{|\mathrm{Aut}(F)|}.
\]

The distance from the collection $\{\mathbf{Z}_F^{ab}\}_{(a,b)\in[K]^2}$ to the null set
\[
\mathcal H_0
=
\bigcap_{(a,b)\in[K]^2}
\mathcal H_0^{ab},
\qquad
\mathcal H_0^{ab}
=
\Big\{
(\theta_{ab}^{(1)},\theta_{ab}^{(2)},\theta_{ab}^{(12)})
:
\theta_{ab}^{(12)}=\theta_{ab}^{(1)}\theta_{ab}^{(2)}
\Big\},
\]
is given by
\begin{equation}
\mathcal D
=
\sqrt{
\sum_{(a,b)\in[K]^2}
\inf_{t\in\mathcal H_0^{ab}}
\|
\mathbf{Z}_F^{ab}
\|_2^2
}, 
\end{equation}
since the null set $\mathcal{H}_0$ is the cartesian product of the sets $\mathcal{H}_0^{ab}$.
\smallskip

Fix $(a,b)\in[K]^2$. Define
\begin{equation}
\begin{aligned}
f(\theta_{ab}^{(1)},\theta_{ab}^{(2)},\theta_{ab}^{(12)})
:=
\|
\mathbf Z_F^{ab}
\|_2^2
=
\frac{1}{n_{ab}^{2(|V(F)|-1)}} \Big[
\big( &X_F(\mathcal G_{ab}^{(1)})-c_{ab}(\theta_{ab}^{(1)})^{|E(F)|}\big)^2
+
\big(X_F(\mathcal G_{ab}^{(2)})-c_{ab}(\theta_{ab}^{(2)})^{|E(F)|}\big)^2\\
&+
\big(X_F(\mathcal G_{ab}^{(1)},\mathcal G_{ab}^{(2)})
- c_{ab}(\theta_{ab}^{(12)})^{|E(F)|}\big)^2\Big].
\end{aligned}
\end{equation}

Under the constraint
\(
\theta_{ab}^{(12)}=\theta_{ab}^{(1)}\theta_{ab}^{(2)},
\)
this becomes
\begin{equation}
\begin{aligned}
f(\theta_{ab}^{(1)},\theta_{ab}^{(2)})
=
\frac{1}{n_{ab}^{2(|V(F)|-1)}} \Big[
\big(&X_F(\mathcal G_{ab}^{(1)})-c_{ab}(\theta_{ab}^{(1)})^{|E(F)|}\big)^2
+
\big(X_F(\mathcal G_{ab}^{(2)})-c_{ab}(\theta_{ab}^{(2)})^{|E(F)|}\big)^2
\\& +
\big(X_F(\mathcal G_{ab}^{(1)},\mathcal G_{ab}^{(2)})
- c_{ab}(\theta_{ab}^{(1)}\theta_{ab}^{(2)})^{|E(F)|}\big)^2\Big].
\end{aligned}
\end{equation}

This function is minimized for $\theta_{ab}^{(1)*}$ and $\theta_{ab}^{(2)*}$ solving the system
\begin{equation}
\begin{cases}
\dfrac{\partial f}{\partial \theta_{ab}^{(1)}}(\theta_{ab}^{(1)*},\theta_{ab}^{(2)*}) = 0,\\[10pt]
\dfrac{\partial f}{\partial \theta_{ab}^{(2)}}(\theta_{ab}^{(1)*},\theta_{ab}^{(2)*}) = 0.
\end{cases}
\end{equation}

Computing the derivatives, we obtain
\begin{equation}
\label{eq:system}
\begin{cases}
-\dfrac{2c_{ab}|E(F)|}{n_{ab}^{2(|V(F)|-1)}}(\theta_{ab}^{(1)*})^{|E(F)|-1}
\Big(X_F(\mathcal G_{ab}^{(1)})-c_{ab}(\theta_{ab}^{(1)*})^{|E(F)|}\Big)
\\ \qquad
-2c_{ab}|E(F)|(\theta_{ab}^{(1)*}\theta_{ab}^{(2)*})^{|E(F)|-1}\theta_{ab}^{(2)*}
\Big(X_F(\mathcal G_{ab}^{(1)},\mathcal G_{ab}^{(2)})
- c_{ab}(\theta_{ab}^{(1)*}\theta_{ab}^{(2)*})^{|E(F)|}\Big)
=0,
\\[10pt]
-2c_{ab}|E(F)|(\theta_{ab}^{(2)*})^{|E(F)|-1}
\Big(X_F(\mathcal G_{ab}^{(2)})-c_{ab}(\theta_{ab}^{(2)*})^{|E(F)|}\Big)
\\ \qquad
-2c_{ab}|E(F)|(\theta_{ab}^{(1)*}\theta_{ab}^{(2)*})^{|E(F)|-1}\theta_{ab}^{(1)*}
\Big(X_F(\mathcal G_{ab}^{(1)},\mathcal G_{ab}^{(2)})
- c_{ab}(\theta_{ab}^{(1)*}\theta_{ab}^{(2)*})^{|E(F)|}\Big)
=0.
\end{cases}
\end{equation}

The system \eqref{eq:system} does not admit an explicit solution, but it can be solved numerically.
Substituting the solutions of the system into the function $f$ completes the proof.
\end{proof}

\section{Auxiliary Results}

\subsection{Multivariate Bernoulli}
\label{append:multibern}
We recall some basic elements of the multivariate Bernoulli distribution, introduced in \cite{Teugels1990}. Assume that $\{X_i\}_{i\in [n]}$ is a sequence of coupled Bernoulli random variables. For $i=1,2,\ldots,n$, we have  
\begin{equation}\label{eq:bernoulli-probs}
    \mathbb{P}(X_i=0) = q_i, 
    \quad  
    \mathbb{P}(X_i =1)= p_i,
\end{equation}
where we denote $0 < p_i = 1-q_i < 1$.  

\cite{Teugels1990} developed an algebraically convenient representation for the \emph{multivariate Bernoulli distribution}, defined as  
\begin{equation}\label{eq:multivar-bernoulli}
    p_{k_1, k_2, \ldots, k_n} := \mathbb{P}\{X_1=k_1, X_2= k_2, \ldots, X_n=k_n\},
\end{equation}
where $k_i \in \{0,1\}$ for $i=1,2,\ldots,n$. Such a representation should be parameterized by the $n$ mean values $\{p_i\}_{i\in[n]}$ and by additional parameters that capture the cross-dependencies among the random variables $\{X_i\}_{i \in [n]}$, such as covariance, skewness, and kurtosis. It is known that $2^n-1$ parameters are required to fully specify this distribution.  
\medskip

\noindent 
We begin with the case $n=2$. We then have  
\begin{align}
&p_{00} = \mathbb{P}\{X_1=0, X_2=0\}, \quad 
p_{10} = \mathbb{P}\{X_1=1, X_2=0\}, \\
&p_{01} = \mathbb{P}\{X_1=0, X_2=1\}, \quad 
p_{11} = \mathbb{P}\{X_1=1, X_2=1\}.
\end{align}
There are three parameters specifiying this distribution, namely  
\begin{align}
   p_1 &= \mathbb{E}(X_1) = \mathbb{P}(X_1=1), \label{eq:p1}\\
   p_2 &= \mathbb{E}(X_2) = \mathbb{P}(X_2=1), \label{eq:p2}\\
   \sigma_{12} &= \mathbb{E}\big[(X_1-p_1)(X_2-p_2)\big]. \label{eq:cov}
\end{align}
Alternatively to $\sigma_{12}$, one can use the cross-moment
\begin{equation}\label{eq:mu12}
    \mu_{12} := \mathbb{E}(X_1X_2) = \sigma_{12} + p_1 p_2 = p_{11}.
\end{equation}
Solving for the probabilities $p_{00}, p_{10}, p_{01},$ and $p_{11}$, we obtain the representations  
\begin{align}
    p_{00} &= q_1 q_2 + \sigma_{12} = 1 - p_1 - p_2 + \mu_{12}, \label{eq:p00}\\
    p_{10} &= p_1 q_2 - \sigma_{12} = p_1 - \mu_{12}, \label{eq:p10}\\
    p_{01} &= q_1 p_2 - \sigma_{12} = p_2 - \mu_{12}, \label{eq:p01}\\
    p_{11} &= p_1 p_2 + \sigma_{12} = \mu_{12}. \label{eq:p11}
\end{align}
It follows that $X_1$ and $X_2$ are independent if and only if $\sigma_{12}=0$. As described in \cite{Teugels1990}, these probabilities admit a compact matrix representation. Define the vector  
\[
    \mathbf{p}^{(2)} = \begin{bmatrix} p_{00} \\ p_{10} \\ p_{01} \\ p_{11} \end{bmatrix}.
\]
Then, one has the relation  
\begin{equation}\label{eq:matrix-repr}
    \mathbf{p}^{(2)} =
    \begin{bmatrix} 1 & -1 & -1 & 1 \\
    0 & 1 & 0 & -1 \\
    0 & 0 & 1 & -1 \\
    0 & 0 & 0 & 1 \end{bmatrix} 
    \begin{bmatrix} 1 \\ p_1 \\ p_2 \\ \mu_{12} \end{bmatrix}.
\end{equation}
This can further be expressed using Kronecker products. In particular, we obtain  
\begin{equation}\label{eq:kronecker2}
    \mathbf{p}^{(2)} = 
    \begin{bmatrix} 1 & -1 \\ 0 & 1 \end{bmatrix} 
    \otimes 
    \begin{bmatrix} 1 & -1 \\ 0 & 1 \end{bmatrix} 
    \begin{pmatrix} 1 & p_1 & p_2 & \mu_{12} \end{pmatrix}^{T}.
\end{equation}

\medskip
\noindent
We now consider the general case. Let $n$ be arbitrary and let $\mathbf{p}^{(n)}$ denote the vector of all $2^n$ probabilities, with components  
\begin{equation}\label{eq:pn-def}
    p_k^{(n)} = p_{k_1,k_2,\ldots,k_n}, 
    \qquad 1\leq k \leq 2^n,
\end{equation}
as defined in \eqref{eq:multivar-bernoulli} and where 
\begin{equation}\label{eq:k-index}
    k = 1 + \sum_{i=1}^n k_i 2^{i-1}, 
    \qquad k_i \in \{0,1\}.
\end{equation}
Introduce $\overline{X}_i = 1-X_i$ for $i=1,2,\ldots,n$. Then  
\begin{equation}\label{eq:pk-repr}
    p_k^{(n)} = \mathbb{P}\left(\bigcap_{i=1}^n \{X_i=k_i\}\right) 
    = \mathbb{E}\left(\prod_{i=1}^n X_i^{k_i} \overline{X}_i^{1-k_i}\right).
\end{equation}
Since for arbitrary $a_i, b_i$ we have  
\[
\left[\begin{bmatrix} a_n \\ b_n \end{bmatrix} \otimes \begin{bmatrix} a_{n-1} \\ b_{n-1} \end{bmatrix} \otimes \cdots \otimes \begin{bmatrix} a_1 \\ b_1 \end{bmatrix}\right]_k 
= \prod_{i=1}^n a_i^{1-k_i} b_i^{k_i}, \quad 1 \leq k \leq 2^n,
\]
we obtain the compact formula  
\begin{equation}\label{eq:pn-kronecker}
\mathbf{p}^{(n)} 
= \mathbb{E}\left[\begin{bmatrix} \overline{X}_n \\ X_n \end{bmatrix} 
\otimes \begin{bmatrix} \overline{X}_{n-1} \\ X_{n-1} \end{bmatrix}
\otimes \cdots 
\otimes \begin{bmatrix} \overline{X}_1 \\ X_1 \end{bmatrix} \right].
\end{equation}
Following \cite{Teugels1990}, we introduce the vector of ordinary moments  
\begin{equation}\label{eq:mu-n}
\mathbf{\mu}^{(n)} = \begin{pmatrix} \mu_1^{(n)} & \mu_2^{(n)} & \cdots & \mu_{2^n}^{(n)} \end{pmatrix}^{T},
\end{equation}
where  
\begin{equation}\label{eq:mu-k}
\mu_k^{(n)} = \mathbb{E}\left(\prod_{i=1}^n X_i^{k_i}\right) 
= \mathbb{E}\left[\begin{bmatrix} 1 \\ X_n \end{bmatrix} \otimes \begin{bmatrix} 1 \\ X_{n-1} \end{bmatrix} \otimes \cdots \otimes \begin{bmatrix} 1 \\ X_1 \end{bmatrix}\right]_k,
\end{equation}
for $k \in \{1, \ldots 2^n\}.$ 

\noindent 
We can now state the following general representation theorem.

\begin{theorem}[Multivariate Bernoulli Representation \cite{Teugels1990}]\label{thm:teugels}
Let $\mathbf{p}^{(n)}$ and $\mathbf{\mu}^{(n)}$, and be as defined in \eqref{eq:pn-kronecker} and \eqref{eq:mu-n} respectively. Then the following representation hold:
\begin{align}
    \mathbf{p}^{(n)} &= \begin{bmatrix} 1 & -1 \\ 0 & 1 \end{bmatrix}^{\otimes n} \mathbf{\mu}^{(n)}. \label{eq:pn-mu}
\end{align}
\end{theorem}

\subsection{Some Elements of Graph Limit Theory}
\label{append:GLT}

In this section, we provide some important elements of graph limit theory, developed by \cite{Lovasz2012}.
The following definition formalizes the notion of subgraph counts in a finite graph.

\begin{definition}[Homomorphism densities \cite{Lovasz2012}]
Let \(F\) and \(G\) be two simple graphs.
The set of homomorphisms of \(F\) into \(G\) is denoted by \(\mathrm{Hom}(F,G)\), and we write \(\mathrm{hom}(F,G) := |\mathrm{Hom}(F,G)|\) for its cardinality.

The \emph{homomorphism density} of \(F\) in \(G\) is defined by
\begin{equation}
t(F,G) := \frac{\mathrm{hom}(F,G)}{|V(G)|^{|V(F)|}},
\end{equation}
which is the probability that a uniformly random map \(\phi: V(F) \to V(G)\) is a homomorphism. 
\end{definition}

\begin{remark}
The homomorphism numbers are closely related to the moments of the adjacency matrix with respect to the trace. If \(C_k\) denotes the cycle on \(k\) vertices, then
\[
\mathrm{hom}(C_k, G) = \mathrm{Tr}(A^k) = \sum_{i=1}^{n} \lambda_i^k,
\]
where \(A\) is the adjacency matrix of \(G\) and \(\Lambda'_d, \ldots, \lambda_n\) are its eigenvalues.
\end{remark}

\noindent
We now define homomorphism densities in graphons, which represent the first-order moments of homomorphism densities.

\begin{definition}[Homomorphism densities in graphons \cite{Lovasz2012}]
Let \(\mathcal{W}\) denote the space of all bounded symmetric measurable functions \(W : [0,1]^2 \to [0,1]\). The elements of \(\mathcal{W}\) are called \emph{graphons}. For a simple graph \(F\) and a graphon \(W \in \mathcal{W}\), the homomorphism density of \(F\) in \(W\) is defined by 
\begin{equation}
\label{eq:homodensdef}
t(F, W) := \int_{[0,1]^{|V(F)|}} \prod_{(i,j)\in E(F)} W(x_i,x_j) \, \prod_{i \in V(F)} dx_i
= \mathbb{E}\!\left(\prod_{(i,j)\in E(F)} W(U_i, U_j)\right),
\end{equation}
where \(U_1, U_2, \ldots\) are i.i.d.\ random variables uniformly distributed on \([0,1]\).
\end{definition}

\noindent
Any finite graph can be represented as a piecewise constant function on the unit square. This allows us to introduce the following definition. 

\begin{definition}[Empirical graphon \cite{Lovasz2012}]
\label{def:empiricalgraphon}
Let \(G\) be a finite simple graph with vertex set \(V(G)\). The \emph{empirical graphon} associated with \(G\), denoted by \(W^G\), is defined for $n =|V(G)|$ by
\begin{equation}
W^G(x,y) := \mathbf{1}\!\left\{ \big(\lceil n \cdot x \rceil, \lceil n \cdot y \rceil\big) \in E(G) \right\}, \quad (x,y) \in [0,1]^2.
\end{equation}
Equivalently, \(W^G\) is obtained by partitioning the unit square \([0,1]^2\) into \(n^2\) squares of side length \(1/n\), and setting \(W^G(x,y) = 1\) in the \((i,j)\)-th square if \((i,j) \in E(G)\), and \(W^G(x,y) = 0\) otherwise.
\end{definition}

\noindent
To compare graphons and define convergence of graph sequences, we introduce the cut-distance and cut-metric.

\begin{definition}[Cut-distance \cite{Lovasz2012}]
\label{def:cut-distance}
Let $\mathcal{W}$ be the space of all measurable functions $[0,1]^2 \to [0,1]$ which are symmetric. 
The cut-distance between two graphons $W_1, W_2 \in \mathcal{W}$ is defined as
\[
\|W_1 - W_2\|_{\square} := \sup_{A,B \subseteq [0,1]} \left| \int_{A \times B} \big(W_1(x,y) - W_2(x,y)\big) \, dx\, dy \right|,
\tag{1.7}
\]
for $W_1, W_2 \in \mathcal{W}$.

The cut-metric between two graphons $W_1, W_2 \in \mathcal{W}$ is defined as
\[
\delta_{\square}(W_1, W_2) := \inf_{\varphi} \| W_1 - W_2^{\varphi} \|_{\square},
\]
where the infimum is taken over all measure-preserving bijections $\varphi : [0,1] \to [0,1]$ and $W_2^{\varphi}(x,y) := W_2(\varphi(x),\varphi(y))$.

We say that a sequence of dense graphs $(G_n)$ converges to a graphon $W$ if
\[
\delta_{\square}\big(W^{G_n}, W\big) \xrightarrow[n \to \infty]{a.s} 0,
\]
where $W^{G_n}$ denotes the empirical graphon associated with $G_n$.

\end{definition}

\subsection{The Stochastic Block Model}
\label{append:SBM}

A commonly adopted approach in graphon estimation is to approximate the graphon function $W(x,y)$, for $0 < x,y < 1$, by a Stochastic Block Model (SBM) with $K$ communities \cite{OlhedeWolfe2014}. In an SBM, the probability of an edge between any two nodes depends solely on the communities to which those nodes belong. To capture this structure, one estimates a \emph{connectivity matrix} $\theta$ of size $K \times K$, which contains the probabilities of edge formation both within communities and between different communities.

Given a single adjacency matrix $A$ of size $n \times n$, the graphon $W(x,y)$ can be estimated using an SBM with community size $h$, yielding a network histogram \cite{OlhedeWolfe2014}. We write $n = hK + r$, where $K = \lfloor n/h \rfloor$ is the total number of communities, $h$ is the corresponding community size, and $r$ is a remainder term satisfying $0 \le r < h$. Nodes are grouped into communities according to a membership vector $z$ of length $n$, with each component taking values in $\{1, \dots, K\}$. Let $\mathcal{Z}_K \subseteq \{1, \dots, K\}^n$ denote the set of all possible community membership vectors that assign $h$ nodes to each of the first $K-1$ communities and $h+r$ nodes to the $K$-th community. 

The main challenge in the SBM framework is estimating the community membership vector $z$. Several methods are available, including spectral clustering, likelihood maximization, and modularity optimization. The classical method is the maximum likelihood estimator (MLE) as in \cite{OlhedeWolfe2014}, and the estimator is given by
\begin{equation}
\label{eq:MLE}
\widehat{z} = \arg\max_{z \in \mathcal{Z}_K} \sum_{i<j} \Big[ A_{ij} \log \widehat{\theta}_{z_i z_j} + (1 - A_{ij}) \log (1 - \widehat{\theta}_{z_i z_j}) \Big],
\end{equation}
where the connectivity matrix entries are
\begin{equation}
\label{eq:connect-matrix}
\widehat{\theta}_{ab} = \frac{\sum_{i<j} A_{ij} \, \mathbf{1}(\widehat{z}_i = a) \, \mathbf{1}(\widehat{z}_j = b)}{\sum_{i<j} \mathbf{1}(\widehat{z}_i = a) \, \mathbf{1}(\widehat{z}_j = b)}, \quad 1 \le a,b \le K.
\end{equation}
Each entry $\widehat{\theta}_{ab}$ corresponds to the proportion of edges present among pairs of nodes assigned to communities $a$ and $b$. The network histogram approximation of the graphon is then defined as
\begin{equation}
\widehat{W}(x,y;h) = \widehat{\theta}_{\min\{\lceil nx/h \rceil, K\}, \min\{\lceil ny/h \rceil, K\}}, \quad 0 < x,y < 1.
\end{equation}

Now considering a $2$-layer multiplex network, there are $3$ parameters to estimate : the graphons $W^{(1)}(x,y)$, $W^{(2)}(x,y)$, and the joint graphon $W^{(12)}(x,y)$. The estimation relies on a correlated two-layer SBM, introduced by  \cite{ChandnaOlhedeWolfe2022}, that shares a single community membership vector $z$ and community size $h$. The vector $z$ groups nodes that belong to the same community across both layers, and we let $\mathcal{Z}_K$ denote all possible vectors respecting $n = hK + r$. Each component of $z$ takes values in $[K]$.

The estimation requires three connectivity matrices, each of size $K \times K$: $\theta^{(1)}$ for the first layer, $\theta^{(2)}$ for the second layer, and $\theta^{(12)}$ for the joint probabilities of edge formation across both layers. Following the single-layer case, we the shared community vector $\widehat{z}$ is estimated using the profile MLE \cite{ChandnaOlhedeWolfe2022}:
\begin{align}
\label{eq:profileMLE}
\widehat{z} = \arg\max_{z \in \mathcal{Z}_K} \sum_{i<j} \Big[& 
A_{ij}^{(1)} A_{ij}^{(2)} \log \widehat{\theta}_{z_i z_j}^{(12)} 
+ A_{ij}^{(1)} (1 - A_{ij}^{(2)}) \log (\widehat{\theta}_{z_i z_j}^{(1)} - \widehat{\theta}_{z_i z_j}^{(12)}) \notag \\
&+ (1 - A_{ij}^{(1)}) A_{ij}^{(2)} \log (\widehat{\theta}_{z_i z_j}^{(2)} - \widehat{\theta}_{z_i z_j}^{(12)}) \notag \\ 
&+ (1 - A_{ij}^{(1)}) (1 - A_{ij}^{(2)}) \log (1 - \widehat{\theta}_{z_i z_j}^{(1)} - \widehat{\theta}_{z_i z_j}^{(2)} + \widehat{\theta}_{z_i z_j}^{(12)}) \Big],
\end{align}
where the connectivity matrix entries are defined for $1 \le a,b \le K$ and $d \in \{1,2,(12)\}$ by
\begin{equation}
\label{eq:connect-matrix-2}
\widehat{\theta}_{ab}^{(d)} = \frac{\sum_{i<j} A_{ij}^{(d)} \, \mathbf{1}(\widehat{z}_i = a) \, \mathbf{1}(\widehat{z}_j = b)}{\sum_{i<j} \mathbf{1}(\widehat{z}_i = a) \, \mathbf{1}(\widehat{z}_j = b)}.
\end{equation}
Each entry $\widehat{\theta}^{(d)}_{ab}$ represents the proportion of edges between nodes in communities $a$ and $b$ across the layers of the multiplex (with $d=1$ for the first layer, $d=2$ for the second, and $d=(12)$ for both). The network histogram approximation of the multivariate graphon $W^{(d)}$ is then given by 
\begin{equation}
\label{eq:networkhistogram}
\widehat{W}^{(d)}(x,y;h) = \widehat{\theta}^{(d)}_{\min\{\lceil nx/h \rceil, K\}, \min\{\lceil ny/h \rceil, K\}}, \quad 0 < x,y < 1,  \quad d \in \{1,2,(12)\}.
\end{equation}

\subsection{Orthogonal Decomposition of Generalized $U-$statistics.}
\label{append:Ortho}
Suppose that \(\{U_i : 1 \leq i \leq n\}\) and \(\{Y_{ij} : 1 \leq i < j \leq n\}\) are independent sequences of i.i.d. random variables uniformly distributed on \([0,1]\). Let \(K_n\) denote the complete graph on the vertex set \(\{1, 2, \ldots, n\}\), and let \(G = (V(G), E(G))\) be a subgraph of \(K_n\).  

For any subgraph \(G\), we define \(\mathcal{F}_G\) to be the \(\sigma\)-algebra generated by the collection of variables \(\{U_i\}_{i \in V(G)}\) and \(\{Y_{ij}\}_{(i,j) \in E(G)}\). Let \(L^2(G) := L^2(\mathcal{F}_G)\) denote the space of all square-integrable random variables that are measurable with respect to \(\mathcal{F}_G\).  

Within this space, we consider the following subspace:
\[
M_G := \left\{ Z \in L^2(G) : \mathbb{E}[ZV] = 0 \text{ for every } V \in L^2(F) \text{ with } F \subset G \right\}.
\]
In other words, \(M_G\) consists of all elements in \(L^2(G)\) that are orthogonal to functions measurable with respect to proper subgraphs of \(G\). For the empty graph, \(M_\emptyset\) is simply the space of constant random variables.

The following orthogonal decomposition holds (see \cite{JansonNowicki1991}, Lemma 1):
\begin{equation}
L^2(G) = \bigoplus_{F \subseteq G} M_F,
\label{eq:orthogonal-decomp}
\end{equation}
which expresses \(L^2(G)\) as the orthogonal direct sum of the subspaces \(M_F\) over all subgraphs \(F \subseteq G\). This decomposition allows any function in \(L^2(G)\) to be uniquely expressed as a sum of orthogonal projections onto the \(M_F\) components.

Let \(M\) be a subspace of \(L^2(K_n)\), and denote by \(P_M\) the orthogonal projection operator onto \(M\). The kernel function \(f^{(k)}\) defined in Equation~\eqref{eq:kernel} can be viewed as an element of \(L^2(K_{|V(F)|})\). Applying the decomposition in Equation~\eqref{eq:orthogonal-decomp}, we write:
\begin{equation}
f^{(k)} = \sum_{H \subseteq F} f_H^{(k)},
\label{eq:f-decomp}
\end{equation}
where \(f_H^{(k)} := P_{M_H} f^{(k)}\) is the orthogonal projection of \(f^{(k)}\) onto \(M_H\).

As a result, the U-statistic defined in Equation~\eqref{eq:u-stat} can be decomposed as
\begin{equation}
S_{n,|V(F)|}(f^{(k)}) = \sum_{H \subseteq K_{|V(F)|}} S_{n,|V(F)|}(f_H^{(k)}).
\label{eq:u-decomp}
\end{equation}

Furthermore, for each \(1 \leq s \leq |V(F)|\), define
\begin{equation}
f^{(k)}_{(s)} := \sum_{\substack{H \subseteq F\\ |V(H)| = s}} f_H^{(k)},
\label{eq:principal-degree}
\end{equation}
which is the component of \(f^{(k)}\) supported on subgraphs \(H\) with exactly \(s\) vertices. The smallest integer \(l\) such that \(f^{(k)}_{(l)} \neq 0\) is called the principal degree of \(f^{(k)}\). The asymptotic distribution of \(X_F(G_n^{(k)})\) is governed by this principal degree.

Finally, for any subgraph \(G \subseteq K_n\), the orthogonal projection onto \(L^2(G)\) is given by the conditional expectation:
\begin{equation}
P_{L^2(G)} = \mathbb{E}[\cdot \mid \mathcal{F}_G].
\label{eq:cond-expectation}
\end{equation}
Combining this with the orthogonal decomposition in Equation~\eqref{eq:orthogonal-decomp}, we obtain the identity:
\begin{equation}
P_{L^2(G)} = \sum_{F \subseteq G} P_{M_F}.
\label{eq:proj-sum}
\end{equation}

\end{document}